\documentclass[twoside,11pt]{article}

\usepackage{jmlr2e}
\usepackage{graphicx,amssymb,amsmath,amsfonts,bbm,psfrag,xcolor}
\usepackage{thmtools,thm-restate}

\usepackage{bm}

\usepackage{natbib}

\usepackage{algorithm}
\usepackage{algorithmic}

\usepackage{hyperref}

\newcommand{\GT}{${\mathcal T}_F$}
\newcommand{\calA}{{\mathcal A}}
\newcommand{\calB}{{\mathcal B}}
\newcommand{\calX}{{\mathcal X}}

\newcommand{\reals}{\mathrm{I\! R}}

\newcommand{\E}{\mathrm{I\! E}}
\newcommand{\Ind}{\mathbf{1}}

\renewcommand{\Pr}{\mathbf{P}}
\renewcommand{\E}{\mathbf{E}}
\renewcommand{\vec}{\mathbf}
\renewcommand{\reals}{\mathbf{R}}
\newcommand{\mse}{\hbox{\textup{\textsf{MSE}}}}
\newcommand{\tr}{\mathrm{trace}}
\newcommand{\cov}{cov}

\usepackage[off]{auto-pst-pdf} % If no figure recompilation necessary

\jmlrheading{X}{X}{X}{X}{X}{Milan Vojnovic and Se-Young Yun}

\ShortHeadings{Parameter Estimation for Thurstone Choice Models}{Milan Vojnovic and Se-Young Yun}
\firstpageno{1}

\begin{document}

\title{Parameter Estimation for Thurstone Choice Models~\thanks{A preliminary version of this work was published in ICML 2016.}}

\author{\name Milan Vojnovi\' c\email m.vojnovic@lse.ac.uk\\
       \addr Department of Statistics\\
			 London School of Economics (LSE)\\
       London, United Kingdom\\
       \AND
       \name Se-Young Yun \email yunseyoung@gmail.com\\
       \addr Los Alamos National Laboratory\\
       New Mexico
			}

\editor{x}

\maketitle

\begin{abstract} 
We consider the estimation accuracy of individual strength parameters of a Thurstone choice model when each input observation consists of a choice of one item from a set of two or more items (so called top-1 lists). This model accommodates the well-known choice models such as the Luce choice model for comparison sets of two or more items and the Bradley-Terry model for pair comparisons.   

We provide a tight characterization of the mean squared error of the maximum likelihood parameter estimator. We also provide similar characterizations for parameter estimators defined by a rank-breaking method, which amounts to deducing one or more pair comparisons from a comparison of two or more items, assuming independence of these pair comparisons, and maximizing a likelihood function derived under these assumptions. We also consider a related binary classification problem where each individual parameter takes value from a set of two possible values and the goal is to correctly classify all items within a prescribed classification error. 

The results of this paper shed light on how the parameter estimation accuracy depends on given Thurstone choice model and the structure of comparison sets. In particular, we found that for unbiased input comparison sets of a given cardinality, when in expectation each comparison set of given cardinality occurs the same number of times, for a broad class of Thurstone choice models, the mean squared error decreases with the cardinality of comparison sets, but only marginally according to a diminishing returns relation. On the other hand, we found that there exist Thurstone choice models for which the mean squared error of the maximum likelihood parameter estimator can decrease much faster with the cardinality of comparison sets. 

We report empirical evaluation of some claims and key parameters revealed by theory using both synthetic and real-world input data from some popular sport competitions and online labor platforms.      
\end{abstract} 

\section{Introduction}
\label{intro}
We consider the statistical inference problem of estimating individual strength or skill parameters of items based on observed choices of items from sets of two or more items. This accommodates the case of pair comparisons as a special case, where each comparison set consists of two items. In our more general case, each observation consists of a comparison set of two or more items and the identity of the chosen item from this set. In other words, each observation is a partial ranking, which is often referred to as a top-1 list. Many applications are accommodated by this framework, e.g., choices indicated by user clicks in various information retrieval systems, outcomes of single-winner contests in crowdsourcing services such as TopCoder or Taskcn, outcomes of hiring decisions where one applicant is hired among those who applied for a job, e.g., in online labour markets such as Fiverr and Upwork, as well as numerous sport competitions and online gaming platforms. 

We consider the parameter estimation for the statistical choice model known as the \emph{Thurstone choice} model; also referred to as the \emph{random utility model}. According to the Thurstone choice model, items are associated with latent performance random variables that are independent across different items and different comparisons. For any given comparison set, the choice is the item from this set with the largest performance random variable. For any given item, the performance random variable is equal to the sum of a strength parameter, whose value can be specific to this item, and a noise random variable according to a given cumulative distribution function. The values of the strength parameters are unknown and have to be estimated from the observed comparisons, and the distribution of noise random variables is assumed to be known. The Thurstone choice model accommodates many known choice models by admitting different distribution for noise random variables, e.g., Luce choice model (\cite{L59}) for comparison sets of two or more items, and its special case for pair comparisons, often referred to as the Bradley-Terry model (\cite{BT54}). 

In this paper, we study the accuracy of the maximum likelihood estimator of the parameter of the Thurstone choice model. Our goal is to characterize the accuracy of the maximum likelihood estimator and shed light on how it depends on the given Thurstone choice model and properties of the observed input data such as the number of observations and the structure of comparison sets. In particular, we consider the following statistical inference question. Suppose that the input observations are such that all comparison sets are of the same cardinality $k\geq 2$ and are unbiased, meaning that in expectation every comparison set of given cardinality occurs the same number of times in the input data. Then, we would like to understand how does the accuracy of the maximum likelihood estimator of the strength parameters depend on the cardinality of comparison sets. Notice that from any comparison set of cardinality $k$, we can deduce at most $k(k-1)/2$ pair comparisons. Intuitively, we would expect that the parameter estimation accuracy would increase with the cardinality of comparison sets. However, it is not a priori clear how fast the accuracy would improve and whether any significant gains can be achieved by increasing the sizes of comparison sets. Moreover, it is also not a priori clear whether or not there can be any significant difference between different Thurstone choice models, with respect to how the parameter estimation accuracy is related to the cardinality of comparison sets. We also consider these questions for parameter estimators that are derived by rank breaking methods, which amount to deducing one or more pair comparisons from each comparison of two or more items, assuming independence of these pair comparisons, and defining the estimator as the maximum likelihood estimator under these assumptions.  

The main contributions of this paper can be summarized as follows. 

We provide upper bounds on the mean squared error of the maximum likelihood estimator, and a lower bound that establishes their minimax optimality. We show that the effect of the structure of comparison sets on the mean squared error is captured by one key parameter: algebraic connectivity of a suitably defined weighted-adjacency matrix. The elements of this matrix correspond to distinct pairs of items and are equal to a weighted sum of the number of input comparisons of different cardinalities that contain the corresponding pair of items, where the weights are specific to given Thurstone choice model. 

For the statistical inference question of how the estimation accuracy improves with the cardinality of comparison sets, we derive a tight characterization of the mean squared error in terms of the cardinality of comparison sets (Corollary~\ref{cor:ksize}). This characterization reveals that for a broad class of Thurstone choice models, which includes the well-known cases such as the Luce choice model, there is a diminishing returns decrease of the mean squared error with the cardinality of comparison sets. For this class of Thurstone choice models, the mean squared error tends to be largely insensitive to the cardinality of comparison sets. On the other hand, we show that there exist Thurstone choice models for which the mean squared error decreases much faster with the cardinality of comparison sets. Perhaps suprisingly, in these cases, the amount of information extracted from a comparison set of cardinality $k$ is in the order of $k^2$ independent pair comparisons, which yields a $1/k^2$ reduction of the mean squared error of the maximum likelihood estimator. Section~\ref{sec:disc} provides more discussion.   

We consider two natural rank-breaking methods, one that deduces $k-1$ pair comparisons and one that deduces $1$ pair comparison from a comparison set of cardinality $k$ (Section~\ref{sec:rankbreaking}). We derive mean squared error upper bounds when choices are according to the Luce choice model for these two rank-breaking methods in Theorem~\ref{thm:break-1} and Theorem~\ref{thm:break}, respectively. These results show that both estimators are consistent. Interestingly, both mean squared error upper bounds are equal to that of the maximum likelihood estimator up to a constant factor. Hence, they both inherit all the properties that we established to hold for the mean squared error upper bound for the maximum likelihood estimator.   

We also consider a binary classification problem where all strength parameters associated with items take one of two possible values (separating items into two classes), and the goal is to correctly classify each item within a prescribed probability of classification error (Section~\ref{sec:classy}). We identify sufficient conditions for correctness of a simple point score classification algorithm (Theorem~\ref{thm:clustering-example}) and establish their tightness (Theorem~\ref{thm:clustering-low}). These conditions are of the same form as those that we imposed for deriving upper bounds on the mean squared error of the maximum likelihood parameter estimator.  

We present experimental results using both simulation and real-world data (Section~\ref{sec:exp}). In particular, we validate the claim that the mean squared error can decrease with the cardinality of comparison sets in a qualitatively different way depending on the given Thurstone choice model. We also evaluate algebraic connectivity of weighted-adjacency matrices for several real-world input data, demonstrating that it can cover a wide range of values depending on specific application scenario.

\subsection{Related Work}
\label{sec:related}

A model of comparative judgement for pair comparisons was introduced by \cite{T27}, which is a special case of a model that we refer to as a Thurstone choice model, for the case of pair comparisons and Gaussian random noise variables. A statistical model of pair comparisons that postulates that an item is chosen from a set of two items with probability proportional to the strength parameter of this item was introduced by \cite{Z29}, and was then popularized by the work of \cite{BT52,BT54} and others, and is often referred to as the Bradley-Terry model. The statistical model of choice where for any set of two or more items an item is chosen with probability proportional to its strength parameter was shown to be a unique model satisfying a set of axioms introduced by \cite{L59}, and is referred as the Luce choice model. The Bradley-Terry model is the special case of the Luce choice model for pair comparisons. The choice probabilities of the Luce choice model correspond to those of a Thurstone choice model with noise random variables according to a double-exponential distribution. Relationships between the Luce choice model and Thurstone choice model were studied by \cite{Y77}. A statistical model for full ranking outcomes (the outcome of a comparison is an ordered list of the compared items) where the ranking is in the order of sampling of items without replacement from the set of compared items and the sampling probabilities are proportional to the strengths of items is referred as the Plackett-Luce model (\cite{L59} and \cite{P75}).  

The Thurstone choice models have been used in the design of several popular skill rating systems, e.g., Elo rating system by \cite{E78} used for skill rating of chess players and in other sports, and TrueSkill by \cite{GMH07} used for skill rating of gamers of a popular online gaming platform. All these models are instances of Thurstone models, and are special instances of generalized linear models, see, e.g., \cite{NW72}, \cite{MN89}, and Chapter~9 in ~\cite{M12}. An exposition to the principles of skill rating systems is available in Chapter~9~in~\cite{V16}. 

The parameter estimation problem for the Bradley-Terry model of pair comparisons has been studied by many. The iterative methods for computing a maximum likelihood parameter estimate have been studied in the early work by \cite{hunter2004mm} and the recent work by \cite{maystre2015fast}. \cite{SY99} shown that the maximum likelihood parameter is consistent and asymptotically normal as the number of items $n$ grows large, under assumption that each pair is compared the same number of times and that the true parameter vector is such that the maximum distance between any of its coordinates is $o(\log(n))$. \cite{M99} studied Thurstonian model parameter estimation with noise random variables according to a Gaussian distribution.

The accuracy of the parameter estimation for various instances of Thurstone models has been studied in recent work. \cite{negahban2012rank} found a sufficient number of input pair comparisons to achieve a given mean squared error of a parameter estimator for the Bradley-Terry model and the input comparisons such that in expectation each distinct pair is compared the same number of times. In particular, they shown that under given assumptions, it suffices to observe $O(n \log(n))$ input pair comparisons. \cite{RA14} studied a statistical convergence property of ranking aggregation algorithms for pair comparisons not only for the Bradley-Terry model but also under some more general conditions referred to as low-noise and generalized low-noise. \cite{hajek2014minimax} provided a characterization of the mean squared error of the maximum likelihood parameter estimate for the Plackett-Luce model of full ranking outcomes. This work found that the algebraic connectivity of a weighted-adjacency matrix captures the effect of the structure of input comparison sets on the mean squared error of the maximum likelihood parameter estimator. \cite{SBBPRW16} established similar characterization results for the case of pair comparisons according to Thurstone choice models.   

Our work differs from previous work in that we consider characterization of the estimation accuracy for a general class of Thurstone choice models for arbitrary sizes of comparisons. Specifically, our work provides a first characterization of the estimation accuracy with respect to the cardinality of comparison sets for unbiased input comparisons, which reveals an insight into the fundamental limits of statistical inference for given cardinality of comparison sets. It is a folklore that different models of pair comparisons yield similar performance with respect to the prediction error, e.g.~\cite{S92}, which suggests that the precise choice of a Thurstone choice model does not matter much in applications. Our results show that there can be significant difference between two Thurstone choice models with respect to the statistical inference of model parameters. 

Parameter estimators derived by various rank-breaking methods have been studied by various authors. For instance, \cite{soufiani2013generalized} and \cite{soufiani14} studied rank-breaking methods for full ranking data and \cite{khetan2016data} studied rank-breaking methods for partial rankings. Recently, \cite{khetan2016} characterized a trade-off between the amount of information used per comparison and the mean squared error of a parameter estimate based on rank breaking. Our work is different in that we are interested in top-1 list observations and the effect of the structure of comparison sets for rank-breaking methods. For the top-1 list observations and any given comparison structure, our work provides upper bounds for the mean squared error of two natural ranking breaking methods and shows the optimality of them.

\subsection{Organization of the Paper}

Section~\ref{sec:defs} introduces problem formulation, some basic concepts and key technical results used to establish our main results. Section~\ref{sec:mse} provides a characterization of the mean squared error for the maximum likelihood parameter estimation, including both upper and lower bounds. Section~\ref{sec:rankbreaking} shows the same type of characterizations for two rank-breaking based parameter estimators. Section~\ref{sec:classy} establishes tight conditions for correct classification of items, when the strength parameters of items are of two possible types. Section~\ref{sec:disc} discusses how the estimation accuracy depends on the cardinality of comparison sets. Section~\ref{sec:exp} contains results of our experiments. Finally, Section~\ref{sec:conc} concludes the paper. Appendix contains some background facts and proofs of our theorems.

\section{Problem Formulation and Notation}
\label{sec:defs}

Let $N=\{1,2,\dots,n\}$ be a set of two or more items. The input data consists of a sequence of one or more observations $(S_1,y_1)$, $(S_2,y_2)$, $\ldots$, $(S_m,y_m)$, where each observation $t$ consists of (a) a \emph{comparison set} $S_t\subseteq N$ and (b) a \emph{choice} of an item $y_t$ from $S_t$. The case of \emph{pair comparisons} is accommodated as a special case when each comparison set consists of two items. Let $w_{i,S}$ denote the number of observations for which the comparison set is $S$ and the chosen item is $i\in S$. With a slight abuse of notation, for pair comparisons, let $w_{i,j}$ be the number of observations for which the comparison set is $\{i,j\}$ and the chosen item is $i$.

A Thurstone choice model, denoted as \GT, is defined by a cumulative distribution function $F$ of a zero mean random variable and parameter vector $\theta = (\theta_1,\theta_2,\ldots,\theta_n)$ that takes value in $\Theta\subseteq \reals^n$. We can interpret $\theta_i$ as the strength of item $i\in N$. The cumulative distribution function $F$ is assumed to have a density function, denoted by $f$.

According to Thurstone choice model \GT, for any given sequence of comparison sets, choices are independent random variables according to the following distribution: conditional on that the comparison set is $S$, with $k$ denoting the cardinality of $S$, the distribution of choice is 
\begin{equation}
p_{i,S}(\theta) = p_{k}(\vec{x}_{i,S}(\theta)), \hbox{ for } i\in S
\label{equ:gtmodel}
\end{equation} 
where $\vec{x}_{i,S}(\theta)$ is a vector in $\reals^{k-1}$ with elements $\theta_i-\theta_j$, for $j\in S\setminus\{i\}$, and
\begin{equation}
p_k(\vec{x}) = \int_{\reals} \left(\prod_{v=1}^{k-1} F(x_v+z)\right) f(z)dz, \hbox{ for } \vec{x} \in \reals^{k-1}.
\label{equ:pkfunc}
\end{equation}

With a slight abuse of notation, for the case of pair comparisons, we let $p_{i,j}(\theta)$ denote the probability that item $i$ is chosen from the comparison set $\{i,j\}$. In this case, we have 
$$
p_{i,j}(\theta)=p_2(\theta_i-\theta_j)
$$ 
where
$$
p_2(x) = \int_{\reals}F(x+z) f(z) dz, \hbox{ for } x\in \reals.
$$

A Thurstone choice model \GT\ corresponds to the following \emph{probabilistic generative model} of choice, also referred to as a \emph{random utility model}. The items of each comparison set are associated with latent \emph{performance random variables}, which are independent across different items and different comparison sets. For any comparison set $S$ and item $i\in S$, the performance random variable $X_i$ is equal to the sum of a strength parameter $\theta_i$ and a noise random variable with distribution $F$. The given probabilistic generative model assumes that for any given comparison set $S$, the chosen item is the one with the largest performance. Hence, the distribution of choice is $p_{i,S}(\theta) = \Pr[X_i \geq \max_{j\in S}X_j]$ for $i\in S$, which can be expressed as asserted in (\ref{equ:gtmodel}).

It is noteworthy that under a Thurstone choice model, the probability distribution of choice depends only on pairwise differences of the strength parameters. This implies that the probability distribution of choice is shift-invariant with respect to the parameter vector. In order to allow for identifiability of the parameter vector, we assume that $\theta$ is such that $\sum_{i=1}^n \theta_i =0$.

We refer to several examples of Thurstone choice models as follows: (i) \emph{Gaussian distribution}  : $f(x) = 1/(\sqrt{2\pi}\sigma)\, \exp(-x^{2/(2\sigma^2}))$ with variance $\sigma^2$; (ii) \emph{Double-exponential distribution}: $F(x) = \exp(-\exp(-(x+\beta\gamma)/\beta))$ with parameter $\beta > 0$ and $\gamma$ denoting the Euler-Mascheroni constant, which has variance $\sigma^2 = \pi^2 \beta^2/6$; (iii) \emph{Laplace distribution}: $F(x) = 1/2\,\exp(x/\beta)$, for $x<0$, and $F(x) = 1-1/2\, \exp(-x/\beta)$, for $x \geq 0$, with parameter $\beta$, which has variance $\sigma^2 = 2 \beta^2$; and (iv) \emph{Uniform distribution}: $f(x) = 1/(2a)$ for $x\in [-a,a]$, which has variance $\sigma^2 = a^2/3$.

For the general case of comparison sets of two or more items, the distribution of choice admits an explicit form only for some special cases. For example, when noise random variables are according to a double-exponential distribution, we have
$$
p_k(\vec{x}) = \frac{1}{1+\sum_{i=1}^{k-1}e^{-x_i/\beta}}, \hbox{ for } \vec{x} \in \reals^{k-1}.
$$
This amounts to the choice probabilities $p_{i,S}(\theta) = \exp(\theta_i/\beta)/\sum_{v\in S} \exp(\theta_v/\beta)$, for $i \in S$ and $S\subseteq N$, which under suitable re-parametrisation corresponds to the well-known Luce choice model. In particular, for pair comparisons, we have $p_2(x) = 1/(1+\exp(-x/\beta))$, which under suitable re-parametrisation corresponds to the well-known Bradley-Terry model. For pair comparisons, the choice probabilities admit an explicit form also for some other Thurstone choice models; for example, when noise has a Gaussian distribution, we have $p_2(x) = \Phi(x/(2\sigma))$ where $\Phi$ is the cumulative distribution function of a standard normal random variable. 

\paragraph{Maximum Pseudo Likelihood Estimation} We consider parameter estimators that are defined as maximizers of a pseudo log-likelihood function $\widetilde{\ell}: \Theta \rightarrow \reals$. We refer to the parameter estimator $\widehat{\theta}\in \arg\max_{\theta \in \Theta}\widetilde{\ell}(\theta)$ as a maximum pseudo likelihood estimator.  

We devote a special attention to maximum likelihood estimator, defined as a maximizer of the log-likelihood function, which for a Thurstone choice model is given by  
\begin{equation}
\ell(\theta) = \sum_{t=1}^m \log(p_{y_t,S_t}(\theta)).
\label{equ:loglik0}
\end{equation}
The log-likelihood function can be written as follows
\begin{equation}
\ell(\theta) = \sum_{S\subseteq N}\sum_{i\in S} w_{i,S}\log(p_{|S|}(\vec{x}_{i,S}(\theta))) + \hbox{const}
\label{equ:loglik}
\end{equation}
where recall $w_{i,S}$ is the number of observations for which the comparison set is $S$ and $i$ is the choice from $S$. In particular, for pair comparisons, we have
\begin{equation}
\ell(\theta) = \sum_{i=1}^n\sum_{j=1}^n w_{i,j} \log \left( p_2(\theta_i - \theta_j) \right)  + \hbox{const}.
\label{eq:MLtheta}
\end{equation}  

Evaluating the value of the log-likelihood function in (\ref{equ:loglik}) for given parameter vector requires evaluating a sum that in the worst-case consists of exponentially many elements in $n$ (all possible combinations of two or more elements from the ground set of $n$ elements). On the other hand, for pair comparisons, the log-likelihood function in (\ref{eq:MLtheta}) is a sum of at most $n^2$ elements; thus, polynomially many elements in $n$. A common approach to reduce computational complexity is to use the so-called \emph{rank breaking}, which amounts to deducing pair comparisons from any given comparison set of two or more items, and assuming that these pair comparisons are independent (if this is not the case). Using these pair comparisons, one then defines a pseudo log-likelihood function as the log-likelihood function under the assumption that the deduced pair comparisons are independent.   

We shall consider two natural rank-breaking methods. The first rank-breaking method deduces $k-1$ pair comparisons from each comparison set of $k$ items, by taking all pairs that consist of the chosen item and each other item in the given comparison set. The pseudo log-likelihood function in this case is given by
\begin{equation}
\ell_{k-1}(\theta) = \sum_{t=1}^m \sum_{v\in S_t\setminus \{y_t\}} \log(p_{y_t,v}(\theta)).
\label{equ:pseudolikk}
\end{equation}
The second rank-breaking method that we consider deduces $1$ pair comparison from each comparison set of $k$ items, by taking a pair that consists of the chosen item and a randomly picked item from the remaining set of items in the given comparison set. The pseudo log-likelihood function in this case is given by
\begin{equation}
\ell_1(\theta) = \sum_{t=1}^m \log(p_{y_t,z_t}(\theta)).
\label{equ:pseudolik1}
\end{equation}

The first rank-breaking method uses maximum amount of information that is contained in a comparison; by observing choice of one item from a comparison set of $k$ items, we can indeed deduce at most $k-1$ pair comparisons (between the chosen item and each other item in the given comparison set). In general, these pair comparisons are not mutually independent. The second rank-breaking method is conservative in deducing only one pair comparison from each comparison set of two or more items.  

\paragraph{Parameter Estimation Accuracy} We study the accuracy of a maximum pseudo log-likelihood estimator $\widehat \theta$ of the true parameter vector $\theta^\star$ by using the \emph{mean squared error} defined as follows:
\begin{equation}
\mse(\widehat \theta, \theta^\star) = \frac{1}{n}\|\widehat \theta - \theta^\star\|_2^2.
\label{equ:msedef}
\end{equation}

We also consider the probability of classification error for the case when the strength parameters belong to one of two classes and the goal is to correctly classify each item.  

\subsection{Eigenvalues, Adjacency, and Laplacian Matrices}

Here we review some basic definitions that are used throughout the paper. We denote eigenvalues of a matrix $\vec{A}\in \reals^{n\times n}$ as $\lambda_1(\vec{A}), \lambda_2(\vec{A}), \ldots, \lambda_n(\vec{A})$. By convention, we assume that $\lambda_1(\vec{A})\leq \lambda_2(\vec{A})\leq \cdots \leq \lambda_n(\vec{A})$. The \emph{spectral norm} $||\vec{A}||_2$ of matrix $\vec{A}\in \reals^{n\times n}$ is defined by $||\vec{A}||_2 = \sqrt{\lambda_n(\vec{A}^\top \vec{A})}$. The spectral norm of $\vec{A}$ is induced by the Euclidean vector norm as follows $||\vec{A}||_2 = \max\{||\vec{A} \vec{x}||_2: \vec{x}\in \reals^n, ||\vec{x}||=1\}$. If $\vec{A}\in \reals^{n\times n}$ is a real symmetric matrix, then $||\vec{A}||_2 = \lambda_n(\vec{A})$. 

For any \emph{weighted-adjacency} matrix $\vec{A}\in \reals^{n\times n}_+$, we consider a \emph{Laplacian} matrix $L_{\vec{A}}$ defined by
$$
L_{\vec{A}} = \mathrm{diag}(\vec{A}\vec{1}) - \vec{A}
$$ 
where for any given vector $\vec{a}$, $\mathrm{diag}(\vec{a})$ denotes the diagonal matrix with diagonal $\vec{a}$. 

For any weighted-adjacency matrix $\vec{A}\in \reals^{n\times n}_+$, we refer to $\lambda_2(L_{\vec{A}})$ as the \emph{Fiedler value} of $\vec{A}$ (\cite{F73,F89}). The Fielder eigenvalue of a weighted-adjacency matrix quantifies its algebraic connectivity. A weighted-adjacency matrix $\vec{A}$ corresponds to a connected graph if and only if it has strictly positive Fielder value, i.e., $\lambda_2(L_{\vec{A}}) > 0$.

For any given observations and given \emph{weight function} $w:\{1,2,\ldots,n\}\rightarrow \reals_+$, we define the weighted-adjacency matrix $\vec{M}_w$ as follows. Let $m_{i,j}(k)$ be the number of comparisons of cardinality $k$ that contain the pair of items $\{i,j\}$. Then, we define $\vec{M}_w$ to be the matrix in $\reals_+^{n\times n}$ with zero diagonal elements and other elements given by 
\begin{equation}
m_{i,j} = \frac{n}{m}\sum_{k\geq 2} w(k) m_{i,j}(k).
\label{equ:weightmatrix}
\end{equation}
With a slight abuse of notation, let $\vec{M}_a$ be the weighted-adjacency matrix defined for the weight function that takes constant value $a > 0$, and let $\vec{M}$ be written in lieu of $\vec{M}_1$. 

If all comparison sets have identical cardinalities, then each element of the weighted-adjacency matrix is equal to the number of comparisons that contain the corresponding pair of items up to a multiplicative factor. The factor $n/m$ can be interpreted as a normalization with the mean number of comparison sets per item. This normalization is admitted so that for the canonical case of pair comparisons when each pair is compared the same number of times, $\lambda_2(L_{\vec{M}_a})$ is a constant independent of the number of observations $m$ and the number of items asymptotically for large $n$. Indeed, in this case, $\lambda_2(L_{\vec{M}_a}) = \cdots = \lambda_n(L_{\vec{M}_a}) = (1-1/n)a$, which is equal to constant $a$, asymptotically for large $n$.

We say that comparison sets are \emph{unbiased} if for any given cardinality, each set of the given cardinality occurs the same number of times. In particular, for pair comparisons, this means that each distinct pair is compared the same number of times. For any unbiased comparison sets, the weighted-adjacency matrix can be expressed as follows. Let $\mu(k)$ be the fraction of comparison sets of cardinality $k$. Then, for every integer $k\geq 2$ and pair of items $\{i,j\}$, $m_{i,j}(k) = \left(\binom{n-2}{k-2}/\binom{n}{k}\right) \mu(k) m = \left(k(k-1)/[n(n-1)]\right) \mu(k) m$. Hence, for every pair of items, we have
\begin{equation}
m_{i,j} = \frac{1}{n-1}\sum_{k\geq 2} w(k) k(k-1)\mu(k).
\label{equ:unbiased}
\end{equation}
It follows that for any unbiased comparison sets, we have
\begin{equation}
\lambda_2(L_{\vec{M}_w}) = \left(1-\frac{1}{n}\right)\sum_{k\geq 2} w(k) k(k-1)\mu(k)
\label{equ:lambda2L}
\end{equation}
which is a constant independent of $n$, asymptotically for large $n$. 

We shall also consider comparison sets that are assumed to be an independent random sequence according to a given distribution. Specifically, we shall consider the case where all comparison sets are of the same cardinality, and are independent samples according to uniform random sampling without replacement from the set of all items. We denote with $\overline{\vec{M}}_w$ the \emph{expected weighted-adjacency} matrix, where the expectation is with respect to the distribution of the sequence of comparison sets. We say that comparison sets are \emph{a priori unbiased} if all non-diagonal elements of $\overline{\vec{M}}_w$ are equal.  

\subsection{A Key Lemma and Probability Tail Bounds}

All upper bounds for the mean squared error of a maximum pseudo log-likelihood estimator in this paper are established by using the following key lemma.

\begin{lemma} Suppose that $g:\reals^n \rightarrow \reals$ satisfies (i) $\nabla^2 g(\vec{\theta})\vec{1} = \vec{0}$ and (ii) $\lambda_2(\nabla^2g(\vec{\theta}))>0$ for all $\vec{\theta} \in \Theta$, where $\Theta = \{\vec{\theta}\in [-b,b]^n: \vec{\theta}^\top \vec{1} = 0\}$ for $b > 0$. Let $\vec{\theta}^\star$ be an arbitrary vector in $\Theta$ and $\widehat{\vec{\theta}}\in \arg\min_{\vec{\theta}\in \Theta} g(\vec{\theta})$. Then, we have
$$
\|\widehat{\vec{\theta}} -\vec{\theta}^\star \|_2 \le 
\frac{2\left\|\nabla g(\vec{\theta}^\star) \right\|_2}{\min_{\vec{\theta}\in \Theta}\lambda_2(\nabla^2 g(\vec{\theta}))}.
$$
\label{lem:mle-taylor}
\end{lemma}

We shall apply this lemma to the case where $g$ is a negative pseudo log-likelihood function, $\widehat{\theta}$ is a maximizer of the pseudo log-likelihood function, and $\theta^\star$ is the true parameter vector. We upper bound the mean squared error by the following two steps: 
\begin{itemize}
\item[{\bf S1}] find $\alpha > 0$ such that $\left\|\nabla g(\vec{\theta}^\star) \right\|_2 \leq \alpha$, and 
\item[{\bf S2}] find $\beta > 0$ such that $\min_{\vec{\theta}\in \Theta}\lambda_2(\nabla^2 g(\vec{\theta})) \geq \beta$
\end{itemize}
which imply the upper bound $\|\widehat{\vec{\theta}} -\vec{\theta}^\star \|_2 \leq 2\alpha / \beta$. 

All our proofs of the mean squared estimation error upper bounds follow the above two-step procedure, including the proof of Theorem~\ref{thm:mle} in Section~\ref{sec:pairs} and other proofs provided in Appendix. 

In step {\bf S1}, $\nabla g(\vec{\theta}^\star)$ is a sum of random vectors. We will make use of the following vector version of Azuma-Hoeffding bound (Theorem~1.8 in~\cite{H03}) for a sum of random vectors.

\begin{lemma}[vector Azuma-Hoeffding bound] Suppose that $S_m = \sum_{t=1}^m X_t$ is a martingale where $X_1, X_2, \ldots, X_m$ are random variables that take values in $\reals^n$ and are such that $\E[X_t] = \vec{0}$ and $\|X_t\|_2 \leq \sigma$ for all $t\in \{1,2,\ldots,m\}$, for $\sigma > 0$. Then, for every $x > 0$,
$$
\Pr\left[\left\|S_m\right\|_2 \geq x\right] \leq 2e^2 e^{-\frac{x^2}{2m \sigma^2}}.
$$
\label{prop:azuma-hoeffding}
\end{lemma}

In step {\bf S2}, we need to find a lower bound for the second-smallest eigenvalue of the Hessian matrix $\nabla^2 g(\vec{\theta})$ for all $\theta \in \Theta$. For pair comparisons according to a Thurstone choice model or comparisons of two or more items according to the Luce choice model, $\nabla^2 g(\vec{\theta})$ is determined by comparison sets and does not depend on the choices. We can find $\beta$ from a Laplacian matrix when the comparison sets are given. In other cases, $\nabla^2 g(\vec{\theta})$ is a sum of random matrices. We will make use the following matrix version of Chernoff's bound along with properties of eigenvalues of a Laplacian matrix (which are given in Appendix~\ref{sec:back}).

\begin{lemma}[matrix Chernoff bound] Let $S_m = \sum_{t=1}^m X_t$ where $X_1, X_2, \ldots, X_m$ are random independent real symmetric matrices in $\reals^{n\times n}$ such that $\lambda_1 (X_t)\ge 0$ and, $\| X_t \|_2 \le \sigma$ for $t \in \{1,2,\ldots,m\}$, for $\sigma > 0$. Then, for $\epsilon \in [0,1)$,
$$
\Pr\left[ \lambda_{1}(S_m) \le (1-\epsilon) \lambda_1(\E[S_m]) \right] \leq n e^{-\frac{\epsilon^2\lambda_1(\E[S_m])}{2\sigma}}.
$$
\label{cor:matrix}
\end{lemma}

\section{Maximum Likelihood Estimation}
\label{sec:mse}
In this section, we present upper and lower bounds for the mean squared error of the maximum likelihood parameter estimator for the Thurstone choice model. We will first consider the case of pair comparisons. We then consider the more general case when each comparison set consists of two or more items. For this more general case, we first give an upper bound for the Luce choice model, and then present similar characterization for a class of Thurstone choice models. We end this section with a lower bound on the mean squared error of the maximum likelihood parameter estimator, which establishes minimax optimality of our upper bounds.   

\subsection{Pair Comparisons}
\label{sec:pairs}

We consider pair comparisons according to a Thurstone choice model \GT\ with parameter vector $\theta^\star$ that takes value in $\Theta = [-b,b]^n$ and that satisfies the following conditions:
\begin{itemize}
\item[\bf P1] There exists $A > 0$ such that
\begin{equation}
\frac{d^2 \log(p_2(x))}{dx^2} \leq -A \hbox{ for all } x\in [-2b,2b].
\label{equ:Bcond}
\end{equation} 
\item[\bf P2] There exists $B > 0$ such that
\begin{equation}
\frac{d \log(p_2(x))}{dx} \leq B \hbox{ for all } x \in [-2b,2b].
\label{equ:Acond}
\end{equation}
\item [\bf P3] The weighted-adjacency matrix $\vec{M}$ is irreducible, i.e. $\lambda_2(L_{\vec{M}}) > 0$.
\end{itemize}

Condition {\bf P1} means that $p_2$ is a strictly log-concave function on $[-2b,2b]$. Condition {\bf P2} means that $\log(p_2(x))$ has a bounded derivative on $[-2b,2b]$. Notice that this condition is equivalent to $dp_2(x)/dx \leq B p_2(x)$ for all $x\in [-2b,2b]$. Constants $A$ and $B$ are specific for given Thurstone choice model and the value of the parameter $b$. In particular, for the Bradley-Terry model, it can be easily checked that  {\bf P1} and {\bf P2} hold with $A = e^{-2b/\beta}/[\beta^2(1+e^{-2b/\beta})^2]$ and $B = 1/[\beta(1+e^{-2b/\beta})]$. Condition {\bf P3} means that the observations are such that the graph defined by the weighted-adjacency matrix is connected. Equivalently, there exists no partition of the set of items $N$ into two non-empty sets $S$ and $N\setminus S$ such that some pair of items $i\in S$ and $j\in N\setminus S$ is not compared in the input observations.

\begin{theorem} Under conditions {\bf P1}, {\bf P2} and {\bf P3}, with probability at least $1-2/n$, 
\begin{equation}
\mse(\widehat{\theta},\theta^\star) \le D^2\frac{n(\log(n)+2)}{\lambda_{2}(L_{\vec{M}_{1/4}})^2}\frac{1}{m}
\label{eq:hajek}
\end{equation}
where $D = B/A$.
\label{thm:mle}
\end{theorem}

Before we show a proof of the theorem, we note the following remarks. 

First, notice that $D$ is a constant whose value is specific to given Thurstone choice model and the value of parameter $b$. In particular, for the Bradley-Terry model, we have $D = \beta(e^{2b/\beta} + 1)$.

Second, from (\ref{eq:hajek}), for the mean squared error to be smaller than or equal to $\epsilon^2$, for given $\epsilon > 0$, it suffices that the number of observations is such that  
\begin{equation}
m \geq \frac{1}{\epsilon^2} D^2 \frac{1}{\lambda_2(L_{\vec{M}_{1/4}})^2}\, n(\log(n)+2).
\label{equ:suffm}
\end{equation}
Third, and last, if each pair of items is compared the same number of times, then, from (\ref{equ:unbiased}), we have $m_{i,j} = 1/(2(n-1))$ for all $i\neq j$. Hence, in this case $\lambda_2(L_{\vec{M}_{1/4}}) = n/(2(n-1))$, and, from (\ref{equ:suffm}), it suffices that the number of observations $m$ is such that 
$$
m \geq \frac{4}{\epsilon^2} D^2\, n(\log(n)+2).
$$ 

\paragraph{Proof of Theorem~\ref{thm:mle}} We now go on to provide a proof of Theorem~\ref{thm:mle}. For pair comparisons, the log-likelihood function (\ref{equ:loglik0}) can be written as
$$
\ell(\theta) = \sum_{t=1}^m \log(p_2(\theta_{y_t}-\theta_{z_t}))
$$
where $y_t$ denotes the choice from the comparison pair $S_t = \{y_t,z_t\}$ for each observation $t$. The negative log-likelihood function satisfies the relation in Lemma~\ref{lem:mle-taylor}, which combined with the following two lemmas, establishes the statement of the theorem.

\begin{lemma} The following relation holds:
$$
\min_{\theta\in \Theta}\lambda_2(\nabla^2(-\ell(\theta))) \geq 4 A\frac{m}{n}\lambda_2(L_{\vec{M}_{1/4}}).
$$
\label{lem:LT21}
\end{lemma}

\begin{proof} It is easy to verify that for all $\theta \in \reals^n$ and $i,j\in N$, we have the following identities
\begin{eqnarray*}
\frac{d^2 \log(p_2(\theta_i - \theta_j))}{dx^2} &=& \frac{\partial^2 \log(p_2(\theta_i - \theta_j))}{\partial \theta_i^2}\\
& = & \frac{\partial^2 \log(p_2(\theta_i - \theta_j))}{\partial \theta_j^2}\\
& = & \frac{\partial^2 (-\log(p_2(\theta_i - \theta_j)))}{\partial \theta_i\partial \theta_j}.
\end{eqnarray*}

For all $i,j\in N$ such that $i\neq j$, we have
$$
\frac{\partial^2 (-\log(p_2(\theta_{y_t}-\theta_{z_t})))}{\partial \theta_i \partial \theta_j} = \left\{
\begin{array}{ll}
\frac{d^2}{dx^2}\log(p_2(\theta_{y_t}-\theta_{z_t})), & \hbox{ if } \{i,j\} = \{y_t,z_t\}\\
0, & \hbox{ otherwise}
\end{array}
\right .
$$
and 
$$
\frac{\partial^2 (-\log(p_2(\theta_{y_t}-\theta_{z_t})))}{\partial \theta_i^2} = - \sum_{j\neq i} \frac{\partial^2 (-\log(p_2(\theta_{y_t}-\theta_{z_t})))}{\partial \theta_i \partial \theta_j}.
$$

Combining with condition {\bf P1}, we have
$$
\frac{\partial^2 (-\log(p_2(\theta_{y_t}-\theta_{z_t})))}{\partial \theta_i \partial \theta_j} \leq -A 1_{\{y_t,z_t\}=\{i,j\}}.
$$
It follows that 
\begin{equation}
\nabla^2(-\ell(\theta)) \succeq 4A\frac{m}{n}L_{\vec{M}_{1/4}}, \hbox{ for all } \theta \in [-b,b]^n
\label{equ:succtemp}
\end{equation}
where for two matrices $\vec{A}$ and $\vec{B}$, $\vec{A} \succeq \vec{B}$ is equivalent to saying that $\vec{A}-\vec{B}$ is positive definite; see Appendix~\ref{sec:back}. In (\ref{equ:succtemp}), both $\nabla^2(-\ell(\theta))$ and $4A\frac{m}{n}L_{\vec{M}_{1/4}}$ are positive-definite matrices. Hence, by the elementary fact stated in Lemma~\ref{lem:EigAsuccB} (Appendix), we obtain the assertion of the lemma.
\end{proof}

\begin{lemma} With probability at least $1-2/n$,
$$
\|\nabla(-\ell(\theta^\star))\|_2 \leq 2B\sqrt{m(\log(n)+2)}.
$$
\label{lem:LT22}
\end{lemma}

\begin{proof}
$\nabla(-\ell(\theta))$ is a sum of independent random vectors in $\reals^n$ given by
$$
\nabla(-\ell(\theta)) = \sum_{t=1}^m \nabla(-\log(p_2(\theta_{y_t}-\theta_{z_t}))).
$$
The elements of $\nabla(-\log(p_2(\theta_{y_t}-\theta_{z_t})))$ can be expressed as follows
$$
\frac{\partial (-\log(p_2(\theta_{y_t}-\theta_{z_t})))}{\partial \theta_i} = \left\{
\begin{array}{ll}
-\frac{d\log(p_2(\theta_{y_t}- \theta_{z_t}))}{dx}, & \hbox{ if } i = y_t\\
+\frac{d\log(p_2(\theta_{y_t}- \theta_{z_t}))}{dx}, & \hbox{ if } i = z_t\\
0, & \hbox{ otherwise}.
\end{array}
\right .
$$

If $i\notin \{y_t,z_t\}$, then clearly 
$$
\E\left[\frac{\partial (-\log(p_2(\theta_{y_t}^\star-\theta_{z_t}^\star)))}{\partial \theta_i}\right] = 0.
$$
Otherwise, we have
\begin{eqnarray*}
\E\left[\frac{\partial (-\log(p_2(\theta_{y_t}^\star-\theta_{z_t}^\star)))}{\partial \theta_i}\right] & = & -p_2(\theta_{y_t}^\star-\theta_{z_t}^\star)\frac{d\log(p_2(\theta_{y_t}^\star - \theta_{z_t}^\star))}{dx}\\
&& + p_2(\theta_{z_t}^\star-\theta_{y_t}^\star)\frac{d\log(p_2(\theta_{z_t}^\star - \theta_{y_t}^\star))}{dx}\\
&=& - \frac{dp_2(\theta_{y_t}^\star-\theta_{z_t}^\star)}{dx} + \frac{dp_2(\theta_{z_t}^\star-\theta_{y_t}^\star)}{dx}\\
&=& 0
\end{eqnarray*}
where the last equation is by the fact that $d p_2(x)/dx$ is an even function.

By condition {\bf P2}, we have
\begin{eqnarray*}
\|\nabla(-\log(p_2(\theta_{y_t}-\theta_{z_t})))\|_2^2  &=& \left(\frac{d\log(p_2(\theta_{y_t}-\theta_{z_t}))}{dx}\right)^2 + \left(-\frac{d\log(p_2(\theta_{y_t}-\theta_{z_t}))}{dx}\right)^2\\
& \leq & 2B^2.
\end{eqnarray*}
Hence, by the vector Azuma-Hoeffding bound in Lemma~\ref{prop:azuma-hoeffding}, we have
$$
\Pr[\|\nabla(-\ell(\theta^\star)\|_2 \geq 2B\sqrt{m(\log(n)+2)}] \leq \frac{2}{n}.
$$
\end{proof}

\subsection{Comparison Sets of Two or More Items}
\label{sec:kary}

We now consider a more general case were each comparison set consists of two or more items. We first show an upper bound for the mean squared error of the maximum likelihood parameter estimator when the choices are according to the Luce choice model and comparison sets are of identical sizes. We then present a similar characterization under more general assumption that allow for a broader set of Thurstone choice models and non-identical sizes of comparison sets.  

\begin{theorem} Suppose that choices are according to the Luce choice model, all comparison sets are of cardinality $k\geq 2$, and $\lambda_2(L_{\vec{M}})>0$. Then, with probability at least $1-2/n$, 
$$
\mse(\widehat\theta, \theta^\star)
\le D^2\frac{n(\log (n)+2)}{\lambda_2(L_\vec{M})^2}\frac{1}{m}
$$
where $D = 4k^2e^{4b}$.
\label{thm:full}
\end{theorem}

The proof of Theorem~\ref{thm:full} is provided in Appendix~\ref{sec:proof-full}. The mean squared error upper bound in Theorem~\ref{thm:full} corresponds that in Theorem~\ref{thm:mle} up to a constant factor. If the comparisons sets are unbiased, from (\ref{equ:lambda2L}), we have that $\lambda_2(L_{\vec{M}}) = (1-1/n)k(k-1)$. Hence, the mean squared error upper bound in Theorem~\ref{thm:full} depends on $k$ only through the factor $1/(1-1/k)^2$, which decreases to value $1$ with $k$ in a diminishing returns fashion. This suggests that there is a limited dependence of the mean squared error on the size of comparison sets.  

We now go on to establish a mean-squared upper bound for a class of Thurstone choice models. We will allow for comparison sets of different cardinalities taking values in a set $K$. We will admit the following conditions: 
\begin{itemize}
\item[{\bf A1}] There exists $A > 0$ such that for all $S\subseteq N$ with $|S|\in K$, all $y\in S$, all $i,j\in S$ with $i\neq j$, and all $\theta\in [-b,b]^n$,
$$
\frac{\partial^2 (-\log(p_{y,S}(\theta)))}{\partial \theta_i \partial \theta_j} \leq A\frac{\partial^2 (-\log(p_{y,S}(\vec{0})))}{\partial \theta_i \partial \theta_j} \leq 0
$$
and, moreover, the following holds
$$
\E\left[\frac{\partial^2 (-\log(p_{y,S}(\vec{0})))}{\partial \theta_i \partial \theta_j}  \right]<0.
$$

\item[{\bf A2}] There exists $B > 0$ such that for all $S\subseteq N$ with $|S|\in K$, all $y\in S$, and all $\theta \in [-b,b]^n$,
$$
\|\nabla p_{y,S}(\theta)\|_2 \leq B \| \nabla p_{y,S}({\bm 0})\|_2.
$$
\item[{\bf A3}] There exists $C > 0$ such that for all $S\subseteq N$ with $|S|\in K$, all $y\in S$, and all $\theta \in [-b,b]^n$,
$$
p_{y,S}(\theta) \geq C p_{y,S}({\bm 0}). 
$$
\end{itemize}

Condition {\bf A1} ensures that $\nabla^2(-\log(p_{y,S}(\vec{0})))$ is a Laplacian matrix with non-negative weights, and that $\nabla^2(-\log(p_{y,S}(\vec{\theta})))\succeq A \nabla^2(-\log(p_{y,S}(\vec{0})))$ for all $\theta\in [-b,b]^n$. Condition {\bf A1} also ensures that the expected value of $\nabla^2 (-\log(p_{y,S}(\vec{\theta})))$ is a positive definite matrix where $\nabla^2(-\log(p_{y,S}(\vec{0})))$ is a random matrix when $|S| >2$. Condition {\bf A2} requires that $\|\nabla p_{y,S}(\theta)\|_2$ is bounded for all $\theta \in [-b,b]^n$, while condition {\bf A3} ensures that the choice probabilities are not too much imbalanced. Conditions {\bf A1} and {\bf A2} may be seen as generalizations of conditions {\bf P1} and {\bf P2} for the case of pair comparisons.

Conditions {\bf A1}, {\bf A2} and {\bf A3} can be easily shown to hold for the Luce choice model. For the Luce choice model, we have   
$$
\frac{\partial^2(-\log(p_{y,S}(\theta)))}{\partial \theta_i \partial \theta_j} = -\frac{1}{\beta^2}p_{i,S}(\theta)p_{j,S}(\theta)
$$
hence, {\bf A1} holds with $A = e^{-4b/\beta}$ and $\partial^2(-\log(p_{y,S}(\vec{0})))/\partial \theta_i \partial \theta_j = -1/(|S|\beta)^2$. Conditions {\bf A2} and {\bf A3} hold with $B = 4$ and $C = e^{-2b/\beta}$. 

Note that constants $A$, $B$ and $C$ that appear in {\bf A1}, {\bf A2} and {\bf A3}, respectively, may depend on $F$, the cardinalities of comparison sets, and the parameter $b$, but are independent of any other parameters. In particular, these constants are independent of the number of observations. For any Thurstone choice model, the constants $A$, $B$, and $C$ can be taken to have values arbitrarily near to value $1$ by taking $b$ small enough. 

We next show an upper bound for the mean squared error of the maximum likelihood parameter estimator for a class of Thurstone choice models that satisfy the above stated conditions. Before we do that, we need to introduce some new definitions. 

\begin{definition}[weight function] Let $w^*$ be a function defined on positive integers greater than or equal to $2$, we refer to as \emph{a weight function}, which is defined by
\begin{equation}
w^\star(k) = \left(k\frac{\partial p_k(\vec{0})}{\partial x_1}\right)^2, \hbox{ for } k = 2,3,\ldots
\label{equ:theweight}
\end{equation}
where
\begin{equation}
\frac{\partial p_k(\vec{0})}{\partial x_1} = \int_\reals f(x)^2 F(x)^{k-2}dx.
\label{equ:partp}
\end{equation}
\end{definition}

Notice that for the Luce choice model, $\partial p_k(\vec{0})/\partial x_1 = 1/(\beta k)^2$. Hence, in this case, $w^*(k) = 1/(\beta k)^2$. We will see later in Section~\ref{sec:disc} that for a broad class of Thurstone choice models, which includes well-known cases with noise according to either Gaussian or double-exponential distribution, $\partial p_k(\vec{0})/\partial x_1 = \Theta(1/k^2)$ and, hence, $w^*(k) = \Theta(1/k^2)$.  

\begin{definition}[$\gamma_{F,k}$ parameter] Let $\gamma_{F,k}$ be a parameter defined by
\begin{equation}
\frac{1}{\gamma_{F,k}} = k^3(k-1) \left(\frac{\partial p_k(\vec{0})}{\partial x_1}\right)^2.
\label{equ:gamma}
\end{equation}
\end{definition}

We note that for any comparison set $S$ of cardinality $k$ and all $y\in S$, we have
$$
\frac{1}{\gamma_{F,k}} = \left\|\nabla \log(p_{y,S} ({\bm 0})) \right\|_2^2 = k(k-1)w^\star (k)
$$
which is discussed in more detail in the proof of Lemma~\ref{lem:kgrad}. In particular, for the Luce choice model, we have $\gamma_{F,k} = (1-1/k)/\beta^2$. 

\begin{theorem} Assume {\bf A1}, {\bf A2} and {\bf A3}, let $\sigma_{F,K}$ be such that $1/\gamma_{F,k}\leq \sigma_{F,K}$ for all $k\in K$, and $\lambda_2(L_{\overline{\vec{M}}_{w^\star}}) \ge 32 (\sigma_{F,K}/C)n\log(n)/m$. Then, with probability at least $1-3/n$,
$$
\mse(\widehat{\theta},\theta^\star) \le 
32 D^2 \sigma_{F,K} \frac{n(\log(n)+2)}{\lambda_2(L_{\overline{\vec{M}}_{w^\star}})^2}\frac{1}{m}
$$
where $D = B/(A C)$.
\label{thm:karyub}
\end{theorem}

The proof of the theorem is provided in Appendix~\ref{sec:proof-karyub}. The main technical difference of the proof with respect to that of Theorem~\ref{thm:full} is that $\nabla^2 (-\ell (\theta))$ is a sum of random matrices. Every $\nabla^2 (-\log(p_{y_t, S_t} (\theta)))$ is a random matrix for the following two reasons: (a) $\nabla^2 (-\log(p_{y_t, S_t} (\theta)))$ depends on the randomly chosen $y_t$ and (b) $S_t$ is allowed to be a random set of items. We use the matrix Chernoff bound in the proof of Theorem~\ref{thm:karyub}.

We next show two corollaries of Theorem~\ref{thm:karyub}, which cover two interesting special cases.

\begin{corollary} Suppose that all comparison sets are of identical cardinality of value $k\geq 2$, {\bf A1}, {\bf A2}, {\bf A3} hold, and $\lambda_2(L_{\overline{\vec{M}}_{1/k^2}}) \geq 32(k-1)/(C k)$. Then, with probability at least $1-3/n$,
$$
\mse(\widehat{\theta},\theta^\star) \le 
32 D^2 \left(1-\frac{1}{k}\right)^2\gamma_{F,k}\frac{n(\log(n)+2)}{\lambda_2(L_{\overline{\vec{M}}_{1/k^2}})^2}\frac{1}{m}
$$
where $\gamma_{F,k}$ is given by (\ref{equ:gamma}), and $\vec{M}_{1/k^2}$ is the weighted-adjacency matrix with the weight function $w(k) = 1/k^2$.
\end{corollary}

\begin{corollary} Suppose that comparison sets are independent with each comparison set being a sample without replacement from the set of all items, conditions {\bf A1}, {\bf A2}, {\bf A3} hold, and $m \geq 32(1-1/n)/C n \log(n)$. Then, with probability at least $1-3/n$,
$$
\mse(\widehat{\theta},\theta^\star) \le 
32 D^2 \left(1-\frac{1}{n}\right)^2\gamma_{F,k} \frac{n(\log(n)+2)}{m}.
$$
\label{cor:ksize}
\end{corollary}

\subsection{Lower Bound}

In this section, we present a lower bound for the mean squared error of the maximum likelihood parameter estimator, which establishes minimax optimality of the established upper bounds. We define the following conditions:
\begin{itemize}
\item[{\bf A1'}] There exists $\widetilde{A} > 0$ such that for all $S\subseteq N$ with $|S|\in K$, all $y\in S$, all $i,j\in S$ such that $i\neq j$, and all $\theta\in [-b,b]^n$, it holds
$$
\frac{\partial^2 (-\log(p_{y,S}(\theta)))}{\partial \theta_i \partial \theta_j} \geq \widetilde{A}\frac{\partial^2 (-\log(p_{y,S}(\vec{0})))}{\partial \theta_i \partial \theta_j}.
$$

\item[{\bf A3'}] There exists $\widetilde{C} > 0$ such that for all $S\subseteq N$ with $|S|\in K$, all $y\in S$, and all $\theta \in [-b,b]^n$, it holds
$$
p_{y,S}(\theta) \leq \widetilde{C} p_{y,S}({\bm 0}). 
$$
\end{itemize}

Notice that, in particular, for the special case of $F$ being a double-exponential distribution with parameter $\beta$, we have that {\bf A1'} and {\bf A3'} hold with $\widetilde{A} = e^{4b/\beta}$ and $\widetilde{C} = e^{2b/\beta}$.

\begin{theorem} Under conditions {\bf A1'} and {\bf A3'}, for any unbiased estimator $\widehat \theta$, we have
$$
\E[\mse(\widehat{\theta},\theta^\star)] \ge \frac{1}{\widetilde{A}\widetilde{C}} \left(\sum_{i=2}^n\frac{1}{\lambda_i(L_{\overline{\vec{M}}_F})}\right)\, \frac{1}{m}. 
$$ 
\label{thm:cramer-rao-bound}
\end{theorem}

The following two corollaries follow from the last theorem.

\begin{corollary} If all comparison sets are of cardinality $k\geq 2$, then any unbiased estimator $\widehat \theta$ satisfies
$$
\E[\mse(\widehat{\theta},\theta^\star)] \ge \frac{1}{\widetilde{A}\widetilde{C}}\left(1-\frac{1}{k}\right)\gamma_{F,k} \left(\sum_{i=2}^n\frac{1}{\lambda_i(L_{\overline{\vec{M}}})}\right)\, \frac{1}{m}. 
$$ 
\end{corollary}

\begin{corollary} If, in addition, each comparison set is drawn independently, uniformly at random from the set of all items, then any unbiased estimator $\widehat \theta$ satisfies 
$$
\E[\mse(\widehat{\theta},\theta^\star)] \ge \frac{1}{\widetilde{A}\widetilde{C}} \left(1-\frac{1}{n}\right)^2 \gamma_{F,k} \frac{n}{m}. 
$$
\end{corollary}

The last corollary implies that under the given assumptions, for the mean squared error to be smaller than a constant, it is necessary that the number of observed comparisons is $m =\Omega(\gamma_{F,k} n)$.

{\bf Proof of Theorem~\ref{thm:cramer-rao-bound}} We denote by  $\cov[Y]$ the covariance matrix  of a multivariate random variable $Y$ i.e.,
$$
\cov[Y] = \E[(Y-\E[Y])(Y-\E[Y])^\top].
$$

The proof uses the Cram\'{e}r-Rao inequality, which is stated as follows.

\begin{lemma}[Cram\'{e}r-Rao bound] Let $X$ be a multivariate random variable with distribution $p(x;\theta)$, for parameter $\theta \in \Theta$, and let $\psi: \Theta \rightarrow \reals^r$ be a differentiable function. Then, for any unbiased estimator $\vec{T}(X) = (T_1 (X), \dots, T_r (X) )^\top$ of $\vec{\psi}(\theta) = (\psi_1 (\theta), \dots, \psi_r (\theta ) )^\top$, we have
$$
\cov[\vec{T}(X)] \ge \frac{\partial \boldsymbol{\psi}(\theta)}{\partial \theta} F^{-1}(\theta) \frac{\partial \boldsymbol{\psi}(\theta)}{\partial \theta}^\top
$$
where $\frac{\partial \vec{\psi}(\theta)}{\partial \theta}$ is the Jacobian matrix of $\psi$ and $F(\theta)$ is the Fisher information matrix given by
$$
F(\theta) = \E[\nabla^2(-\log(p(X;\theta)))].
$$
\label{prop:cramer}
\end{lemma}

Let us define $\psi_i(\theta) = \theta_i - \frac{1}{n}\sum_{l=1}^n\theta_l$ for all $i = 1,2,\ldots,n$. Since $\sum_{i=1}^n \theta_i =0$, we have $\sum_{i=1}^n \psi_i (\theta) = 0$. Note that we can write 
\begin{equation}
\frac{\partial \boldsymbol{\psi}(\theta)}{\partial \theta} = \vec{I} - \frac{1}{n}\vec{1}\vec{1}^\top.\label{eq:psi}
\end{equation} 

Let $F(\theta)$ be the Fisher information matrix given by
\begin{equation}
F(\theta) = \sum_{t=1}^m \E\left[\nabla^2(-\log(p_{y_t,S_t}(\theta))) \right] .\label{eq:fisher-p1}
\end{equation}

By conditions {\bf A1'} and {\bf A3'}, and Lemma~\ref{lem:explog}, we have
\begin{align}
\E\left[\nabla^2(-\log(p_{y_t,S_t}(\theta)))\right]
& =\sum_{y \in S_t} p_{y,S_t}(\theta) \nabla^2 (-\log(p_{y,S_t}(\theta))) \cr
& \preceq \sum_{y \in S_t} \frac{\widetilde{C}}{|S_t|} \nabla^2 (-\log(p_{y,S_t}(\theta)))\cr
& \preceq \sum_{y \in S_t} \frac{\widetilde{A}\widetilde{C}}{|S_t|} \nabla^2 (-\log(p_{y,S_t}({\bm 0})))\cr
& = \widetilde{A}\widetilde{C} \left(|S_t|\frac{\partial p_{|S_t|}(\vec{0})}{\partial x_1}\right)^2L_{\vec{M}_{S_t}}
\label{eq:hessian-s1}
\end{align}
where $\vec{M}_S$ is a matrix in $\reals^{n\times n}$ that has each element $(i,j)$ such $\{i,j\}\subseteq S$ equal to $1$ and all other elements equal to $0$.

From \eqref{eq:fisher-p1} and \eqref{eq:hessian-s1},
\begin{equation}
F(\theta) \preceq  \widetilde{A}\widetilde{C} \frac{m}{n} L_{\overline{\vec{M}}_F}. 
\label{eq:fisher-up}
\end{equation}

Note that
$$
\E[\|\widehat{\theta}-\theta\|_2^2] = \tr(\cov[\vec{T}(X)])
= \sum_{i=1}^n \lambda_i(\cov[\vec{T}(X)]).
$$

By the Cram\'{e}r-Rao bound and (\ref{eq:fisher-up}), we have 
\begin{eqnarray*}
\frac{1}{n}\E[\|\widehat{\theta}-\theta\|_2^2]  
&\geq& \frac{1}{n}\sum_{i=1}^n \lambda_i \left(\frac{\partial \boldsymbol{\psi}(\theta)}{\partial \theta} F^{-1}(\theta) \frac{\partial \boldsymbol{\psi}(\theta)}{\partial \theta}^\top \right)\\
&=& \frac{1}{n}\sum_{i=1}^{n-1}  \lambda_i(\vec{U}^\top F^{-1}(\vec{0}) \vec{U})\\
&=& \frac{1}{n}\sum_{i=1}^{n-1}  \frac{1}{\lambda_i(\vec{U}^\top F(\vec{0}) \vec{U})}\\
& \geq &\frac{1}{\widetilde{A}\widetilde{C}} \sum_{i=1}^{n-1}  \frac{1}{\lambda_i(\vec{U}^\top L_{\overline{\vec{M}}_F} \vec{U})}\\
& = &\frac{1}{\widetilde{A}\widetilde{C} m} \sum_{i=2}^{n}  \frac{1}{\lambda_i(L_{\overline{\vec{M}}_F})}.
\end{eqnarray*}

\section{Rank-Breaking Parameter Estimation Methods}
\label{sec:rankbreaking}
The maximum likelihood parameter estimation requires to find a maximizer of a log-likelihood function which for comparison sets of cardinality $k$ has the form of a sum of $\binom{n}{k}$ elements in the worst case. For pair comparisons, this sum consists of at most $n^2$ elements. It is common for parameter estimators to be defined as maximizers of a pseudo log-likelihood function, which is defined as the log-likelihood function of pair comparisons deduced from the input observations under assumption that these pair comparisons are independent (which in general is false under a Thurstone choice model for comparison sets of three or more items). This is commonly referred as \emph{rank breaking}. In what follows, we consider two different rank-breaking methods: (a) one that deduces $k-1$ pair comparisons from a choice from a comparison set of cardinality $k$, we refer to as rank-breaking method $\mathsf{ALL}$, and (b) one that deduces $1$ pair comparison from a choice from a comparison set of cardinality $k$, we refer to as rank-breaking method $\mathsf{ONE}$. 

\label{sec:rankA}
\paragraph{Rank-breaking method $\mathsf{ALL}$} This rank-breaking method deduces $k-1$ pair comparisons from each comparison set of cardinality $k$. Specifically, for every comparison set $S_t$, the method uses pair comparisons between the chosen item $y_t$ and each non-chosen item $v\in S_t\setminus\{ y_t\}$. Notice that for any comparison set of three or more items, the pair comparisons selected from this set by the given rank-breaking method are not independent. 

For the Luce choice model, the pseudo log-likelihood function is given by (\ref{equ:pseudolikk}), which can be written in a more explicit form as follows
$$
\ell_{k-1}(\theta) = \sum_{t=1}^m\sum_{v\in S_t\setminus\{y_t \}} \log\left(\frac{e^{\theta_{y_t}}}{e^{\theta_{y_t}} + e^{\theta_v}}\right).
$$
We consider the maximum pseudo log-likelihood estimator $\widehat{\theta}_{k-1} := \arg \max_{\theta \in \Theta} \ell_{k-1}(\theta)$.
 
\begin{theorem} If $\lambda_2\left(L_\vec{M} \right) \ge 128(k-1)^2e^{2b} n\log (n)/m$, then with probability at least $1-3/n$,
$$
\mse(\widehat{\theta}_{k-1}, \theta^\star) \le D^2 
\frac{n(\log(n)+2)}{\lambda_2(L_\vec{M})^2}\frac{1}{m}. 
$$
where $D = 16\sqrt{2}\sqrt{k(k-1)^3} e^{2b}$.
\label{thm:break-1}
\end{theorem} 

The proof of Theorem~\ref{thm:break-1} is given in Appendix~\ref{sec:proof-break-1}.

The mean squared error upper bound in Theorem~\ref{thm:break-1} implies that the given parameter estimator is consistent. For any fixed size of a comparison set, the bound in Theorem~\ref{thm:break-1} is equal to that in Theorem~\ref{thm:full} up to a constant factor. Both bounds have the same scaling with parameter $k$.

\label{sec:rankB}
\paragraph{Rank-breaking method $\mathsf{ONE}$} This rank-breaking method deduces $1$ pair comparison from a comparison set of cardinality $k$. From each comparison set $S_t$, this rank-breaking methods selects a pair that consists of the chosen item $y_t$ and an item $z_t$ selected uniformly at random from the set of non-chosen items $S_t\setminus \{y_t\}$. 

For the Luce choice model, the pseudo log-likelihood function is given by (\ref{equ:pseudolik1}), which can be written in a more explicit form as follows
$$
\ell_1(\theta) = \sum_{t=1}^m \log\left(\frac{e^{\theta_{y_t}}}{e^{\theta_{y_t}} + e^{\theta_{z_t}}}\right).
$$
We consider the maximum pseudo log-likelihood estimator $\widehat{\theta}_1 := \arg\max_{\theta\in \Theta} \ell_1(\theta)$.

\begin{theorem} If $\lambda_2\left(L_\vec{M} \right) \ge 8k(k-1)e^{2b}n\log (n)/m$, then with probability at least $1-3/n$,
$$
\mse(\widehat{\theta}_1, \theta^\star) 
\le D^2\frac{n(\log(n)+2)}{\lambda_2(L_\vec{M})^2}\frac{1}{m}
$$
where $D = 4k(k-1) e^{2b}$.
\label{thm:break}
\end{theorem}

The proof of Theorem~\ref{thm:break} is given in Appendix~\ref{sec:proof-break}.

It is noteworthy that the mean squared error upper bounds in Theorem~\ref{thm:break-1} and Theorem~\ref{thm:break} are equal up to a constant factor. Intuitively, one would expect that rank-breaking method $\mathsf{ALL}$ would yield a smaller mean squared error than rank-breaking method $\mathsf{ONE}$ because it uses more information from each observed choice. The reason why the two mean squared error upper bounds are equal up to a constant factor is as follows. When applying Lemma~\ref{lem:mle-taylor} we need to find an upper bound $\alpha$ for the norm of the gradient of the negative pseudo log-likelihood function and a lower bound $\beta$ for the second-smallest eigenvalue of the Hessian matrix of the negative pseudo log-likelihood function. In our proofs, for the case of Theorem~\ref{thm:break}, we obtained $\alpha$ and $\beta$ that scale with $k$ as $1$ and $1/k^2$, respectively. On the other hand, for the case of Theorem~\ref{thm:break-1}, we obtained $\alpha$ and $\beta$ that scale with $k$ as $k$ and $1/k$, respectively. For both cases, it follows that the ratio $\alpha/\beta$ scales as $k^2$.

\section{Binary Classification}
\label{sec:classy}
We now consider a Thurstone choice model \GT\ where the strength parameter vector $\theta$ takes value in $\Theta = \{-b,b\}^n$, for a parameter $b > 0$. Here each individual strength parameter takes either a low or a high value. We say that each item is either of a low or a high class. We consider a binary classification problem, where the goal is to correctly classify all items with a prescribed probability of classification error. Let $N_1$ be the set of high class items and $N_2$ be the set of low class items. We shall consider the case when the total number of items is even and the number of high class items is equal to the number of low class items. 

We consider a simple classification algorithm that uses \emph{point scores} defined as follows: each item is associated with a point score equal to the number of comparison sets in which the given item is the chosen item. The algorithm outputs a classification of items with $\widehat N_1$ and $\widehat N_2$ denoting the sets of items classified to be high class or low class, respectively. The algorithm follows the following three steps: (a) for each item compute its point score, (b) sort the items in decreasing order of point scores, and (c) let $\widehat{N}_1$ contain $n/2$ items with highest point scores and $\widehat{N}_2$ contain remaining items. We refer to this algorithm as a \emph{point score ranking method}.

\begin{theorem} Suppose that $b \leq 4/(k^2 \partial p_k(\vec{0})/\partial x_1)$ and
\begin{equation}
b \max_{\vec{x}\in [-2b,2b]^{k-1}} \|\nabla^2 p_k (\vec{x}) \|_2 \le \frac{\partial p_k ({\bm 0})}{\partial x_1}.
\label{equ:bcond}
\end{equation}
Then, for every $\delta \in (0,1]$, if
\begin{equation}
m \ge 64 \frac{1}{b^2}\, \left(1-\frac{1}{k}\right)\gamma_{F,k}\, n(\log(n) + \log(1/\delta))
\label{equ:classifyub}
\end{equation}
then, the point score ranking method correctly classifies all items with probability at least $1-\delta$.
\label{thm:clustering-example}
\end{theorem}

The proof of Theorem~\ref{thm:clustering-example} is provided in Appendix~\ref{sec:clustering-example}. 

The sufficient condition in (\ref{equ:classifyub}) for the point score ranking method to correctly classify all the items with probability at least $1-\delta$ is shown to be necessary up to a constant-factor for any classification algorithm, which is given in the following theorem. 

\begin{theorem} Suppose that $b \le 1/(6 k^2\partial p_k(\vec{0})/\partial x_1)$ and that condition (\ref{equ:bcond}) holds. Then, for every even $n \geq 16$ and $\delta \in (0,1/4]$, for any algorithm to correctly classify all the items with probability at least $1-\delta$, it is necessary that the following condition holds
$$
m \geq \frac{1}{62}\frac{1}{b^2}\, \left(1-\frac{1}{k}\right)\gamma_{F,k}\, n(\log(n) + \log(1/\delta)).
$$
\label{thm:clustering-low}
\end{theorem}

The proof of Theorem~\ref{thm:clustering-low} is given in Appendix~\ref{sec:clustering-low}. In the proof, we use the statistical difference between the case when all the items are correctly classified and the case that an item is incorrectly classified. This proof strategy is motivated by that in \cite{yun2014community} where it was used to analyze the classification error of the stochastic block model.

\section{Discussion of Results}
\label{sec:disc}

In this section, we discuss how the number of observations needed to attain a prescribed parameter estimation error depends on the cardinality of comparison sets. 

In Section~\ref{sec:mse}, we found that for a priori unbiased input comparisons, where each comparison set is of cardinality $k$ and is drawn uniformly at random from the set of all items, the number of observations needed for the mean squared error to be within a prescribed tolerance is of the order $\gamma_{F,k}$, defined by (\ref{equ:gamma}). In Section~\ref{sec:classy}, we found that this also so to ensure that the probability of classification error is within a prescribed tolerance. 

\begin{table}[h]
\caption{The values of parameters for our examples of \GT.}
\begin{center}
\begin{tabular}{c|cc}
$F$ & $\frac{\partial p_k(\vec{0})}{\partial x_1}$ & $\gamma_{F,k}$\\\hline
Gaussian & $O(\frac{1}{k^{2 - \epsilon}})$ & $\Omega(\frac{1}{k^{2\epsilon}})$\\
Double-exponential & $\frac{1}{\beta k^2}$ & $\beta^2\frac{k}{k-1}$\\
Laplace & $\frac{1-1/2^{k-1}}{\beta k(k-1)}$ & $\beta^2\frac{k-1}{k(1-1/2^{k-1})^2}$\\
Uniform & $\frac{1}{2a(k-1)}$ & $4a^2\frac{k-1}{k^3}$
\end{tabular}
\end{center}
\label{tab:gamma}
\end{table}

In Table~\ref{tab:gamma}, we show the values of the parameter $\gamma_{F,k}$ for several special instances of the Thurstone choice model, along with the values of $\partial p_k(\vec{0})/\partial x_1$. From the expressions in Table~\ref{tab:gamma}, we observe that for all the cases but the case of uniform distribution of noise, $\gamma_{F,k}$ decreases with the cardinality of comparison sets, but in a slow manner according to a diminishing returns relation. In particular, observe that for both double-exponential and Laplace distribution of noise, $\gamma_{F,k} = \Theta(1)$ and for Gaussian distribution of noise $\gamma_{F,k} = O(1/k^\epsilon)$. On the other hand, for uniform distribution of noise, $\gamma_{F,k} = \Theta(1/k^2)$. 

It is noteworthy that $\gamma_{F,k}$ satisfies the following bounds.

\begin{lemma} For any cumulative distribution function $F$ with density function $f$: 
\begin{enumerate}
\item If $f$ is even and continuously differentiable, then $\gamma_{F,k} = O(1)$. 
\item If $f$ is such that $f(x) \leq C$ for all $x\in \reals$, for a constant $C > 0$, then $\gamma_{F,k} = \Omega(1/k^2)$.
\end{enumerate}
\label{prop:gamma}
\end{lemma}

We observe that both double-exponential distribution and Laplace distribution of noise are extremal in the sense that they achieve the upper bound $\gamma_{F,k}=O(1)$. On the other hand, uniform distribution of noise is extremal in the sense that it achieves the lower bound $\gamma_{F,k}=\Omega(1/k^2)$. 

%{\red Note that $A$, $B$, and $C$ could depend on $k$ for a given $b>0$ as well as $\gamma_{F,k}$. Therefore, we should be careful to say that uniform distribution of noise is the extreme case. }

%More generally, it can be shown that $\gamma_{F,k} = \Theta(1/k^2)$ for any cumulative distribution function $F$ that has density function such that $f(x)\geq C$ for every point $x$ of its support, for a constant $C > 0$.

%\footnote{This follows by noting that by the second claim in Proposition~\ref{prop:gamma}, $\gamma_{F,k} = \Omega(1/k^2)$, and by (\ref{equ:gamma}) and the fact that by (\ref{equ:ch2}) $\partial p_k(\vec{0})/\partial x_1 = \Omega(1/k)$, we have $\gamma_{F,k} = O(1/k^2)$.} 

%TBD - can we exhibit a "natural" family of distribution of noise that can achieve $\gamma_{F,k} = \Theta(1/k^\alpha)$, for $0 \leq \alpha \leq 2$ ? If this is possible, here is one way to do this making use of (\ref{equ:ch2}): show that for every $c \in (0,1)$, there exists $F$ such that
%$$
%\int_{-a}^{a} (-f'(x))dF(x)^k = \Theta(k^c)
%$$ 
%where the support of $F$ is contained in $[-a,a]$. From (\ref{equ:ch2}), note that it is necessary that $f(a) = 0$.

\section{Experimental Results}
\label{sec:exp}
In this section, we present our experimental results using both simulations and real-world data. Our first goal is to provide experimental validation of the claim that the mean squared error can depend on the cardinality of comparison sets in different ways for different Thurstone choice models, which is suggested by our theory. Our second goal is to evaluate Fiedler value for different weighted-adjacency matrices observed in real-world data, which demonstrates that it can assume a wide range of values depending on the application scenario. 

\subsection{$\mse$ versus Cardinality of Comparison Sets}

We consider the following simulation experiment. We fix the number of items $n$ and the number of observations $m$. We then run experiments for different values of the cardinality of comparison sets $k$. For each given value of parameter $k$, we generate comparison sets as independent uniform random sets of cardinality $k$ from the set of all items. We then draw choices according to a Thurstone choice model \GT\ for the value of parameter vector $\theta^\star = \vec{0}$. For every fixed value of $k$, we run $100$ repetitions to estimate the mean squared error. We do this for the distribution of noise according to a double-exponential distribution (Bradley-Terry model) and according to a uniform distribution, both with unit variance.

Figure~\ref{fig:synt} shows the results for the case of $n = 10$ and $m = 100$. The results clearly demonstrate that the mean squared error exhibits qualitatively different decay with the cardinality of comparison sets for the two Thurstone choice models under consideration. Our theoretical results in Section~\ref{sec:kary} suggest that the mean squared error should decrease with the cardinality of comparison sets as $1/(1-1/k)$ for the double-exponential distribution, and as $1/k^2$ for the uniform distribution of noise. Observe that the latter two terms decrease with $k$ to a strictly positive value and to zero value, respectively. The empirical results in Figure~\ref{fig:synt} confirm these claims. 

\begin{figure}[t]
\centering
\vspace*{-1.85cm}
\includegraphics[width=0.49\textwidth]{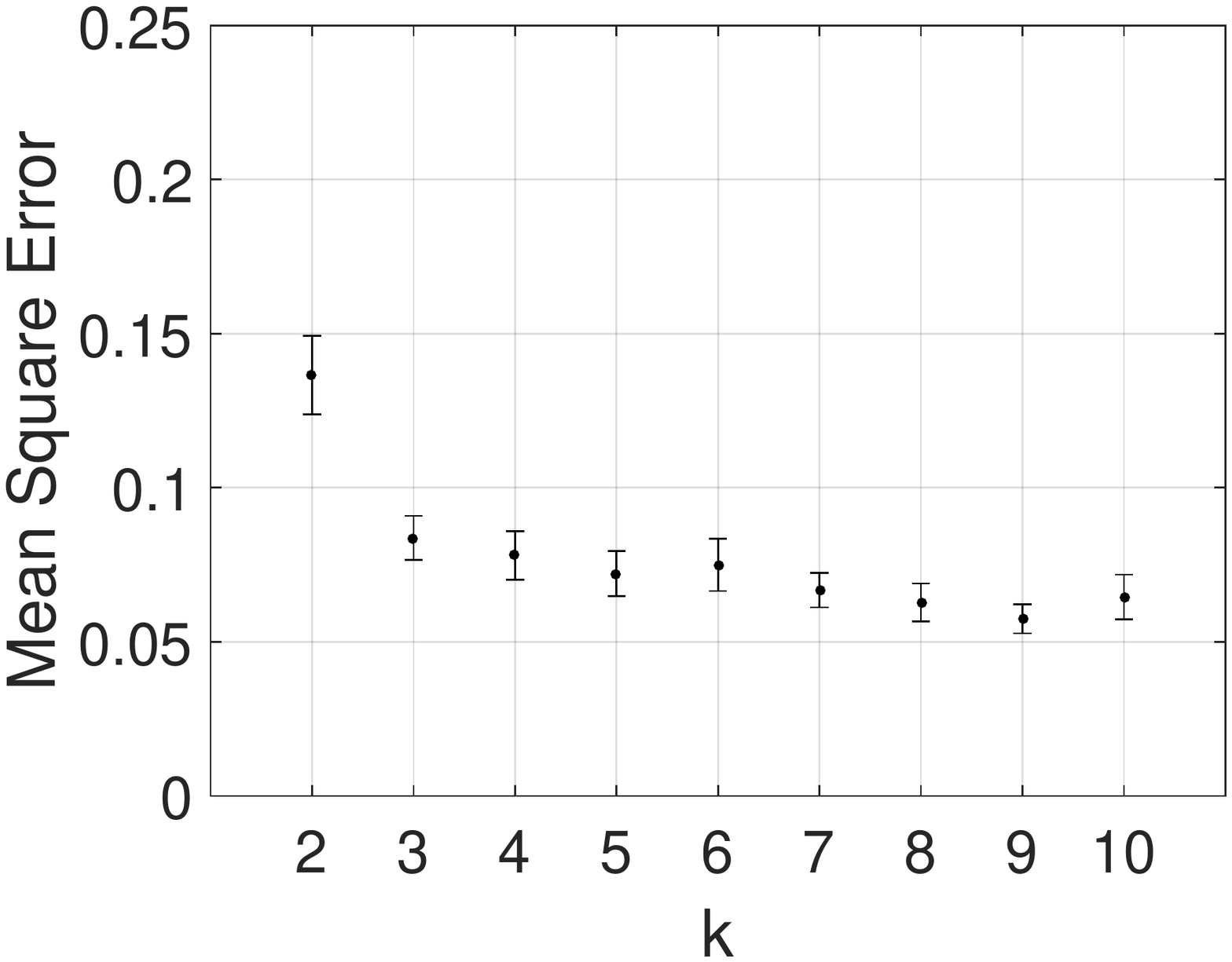}
\includegraphics[width=0.49\textwidth]{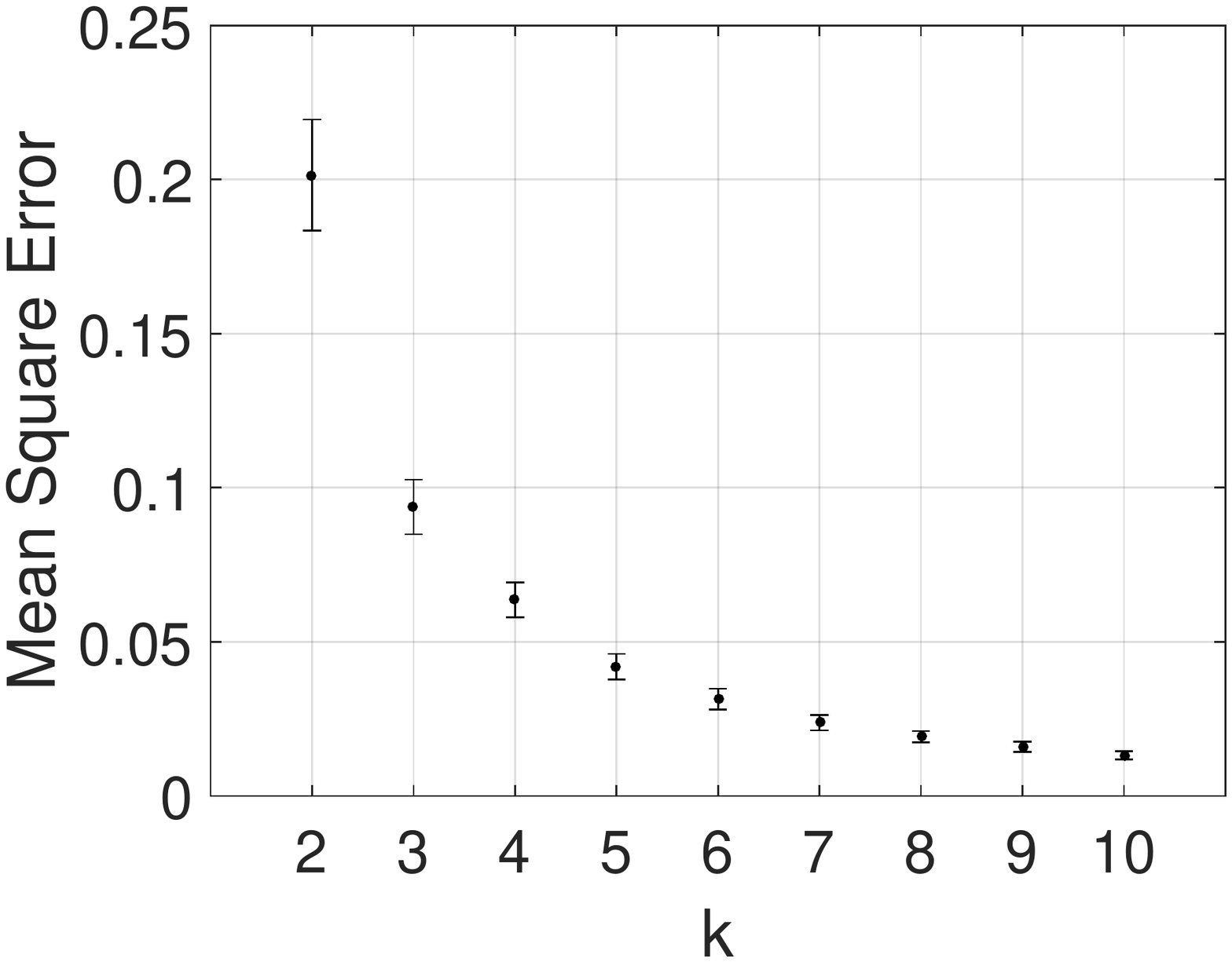}
\vspace*{-2cm}
\caption{Mean squared error for two different Thurstone choice models \GT: (left) double-exponential distribution of noise, and (right) uniform distribution of noise. The vertical bars denote 95\% confidence intervals. The results demonstrate two qualitatively different relations between the mean squared error and the cardinality of comparison sets, which confirm the theory.}
\label{fig:synt}
\end{figure}

\subsection{Fiedler Values of Weighted-Adjacency Matrices}

We found that Fiedler value of a weighted-adjacency matrix plays a key role in upper bounds on the mean squared error of parameter estimator in Section~\ref{sec:pairs} and Section~\ref{sec:kary}. Here we evaluate Fiedler value for different weighted-adjacency matrices of different schedules of comparisons. Throughout this section, we use the definition of a weighted-adjacency matrix in (\ref{equ:weightmatrix}) with the weight function $w(k) = 1/k^2$. Our first two examples are representative of schedules in sport competitions, which are typically carefully designed by sport associations and exhibit a large degree of regularity. Our second two examples are representative of comparisons that are induced by user choices in the context of online services, which exhibit much more irregularity. 

\paragraph{Sport competitions} We consider the fixtures of games for the season 2014-2015 for (i) football Barclays premier league and (ii) basketball NBA league. In the Barclays premier league, there are 20 teams, each team plays a home and an away game with each other team; thus there are 380 games in total. In the NBA league, there are 30 teams, 1,230 regular games, and 81 playoff games.\footnote{The NBA league consists of two conferences, each with three divisions, and the fixture of games has to obey constraints on the number of games played between teams from different divisions.} We evaluate Fiedler value of weighted-adjacency matrices defined for first $m$ matches of each season; see Figure~\ref{fig:premier-nba}.

%The NBA league consists of two conferences, namely Eastern and Western conference; each conference consists of three divisions; each division consists of five teams. The schedule of games is determined subject to the following constraints: each team has to play 4 games against each other team in the same division; 4 games against 6 out-of-division teams from the same conference; 3 games against the remaining 4 teams in the same conference; 2 games against each team of the other conference. The playoff round is an elimination tournament played by the top 8 teams from each conference; the winner between a pair of teams is determined by the best-of-seven rule, i.e., as soon one of the two teams wins four games, requiring at least 4 and at most 7 games.

For the Barclays premier league dataset, at the end of the season, the Fiedler value of the weighted-adjacency matrix is of value $n/[2(n-1)] \approx 1/2$. The schedule of matches is such that at the middle of the season, each team played against each other team exactly once, at which point the Fiedler value is $n/[4(n-1)] \approx 1/4$. The Fiedler value is of a strictly positive value after the first round of matches. For most part of the season, its value is near to $1/4$ and it grows to the highest value of approximately $1/2$ in the last round of the matches.

For the NBA league dataset, at the end of the season, the Fiedler value of the weighted-adjacency matrix is approximately $0.375$. It grows more slowly with the number of games played than for the Barclays premier league; this is intuitive as the schedule of games is more irregular, with each team not playing against each other team the same number of times.

\begin{figure}[t]
\centering
\vspace*{-1.85cm}
\includegraphics[width=0.49\textwidth]{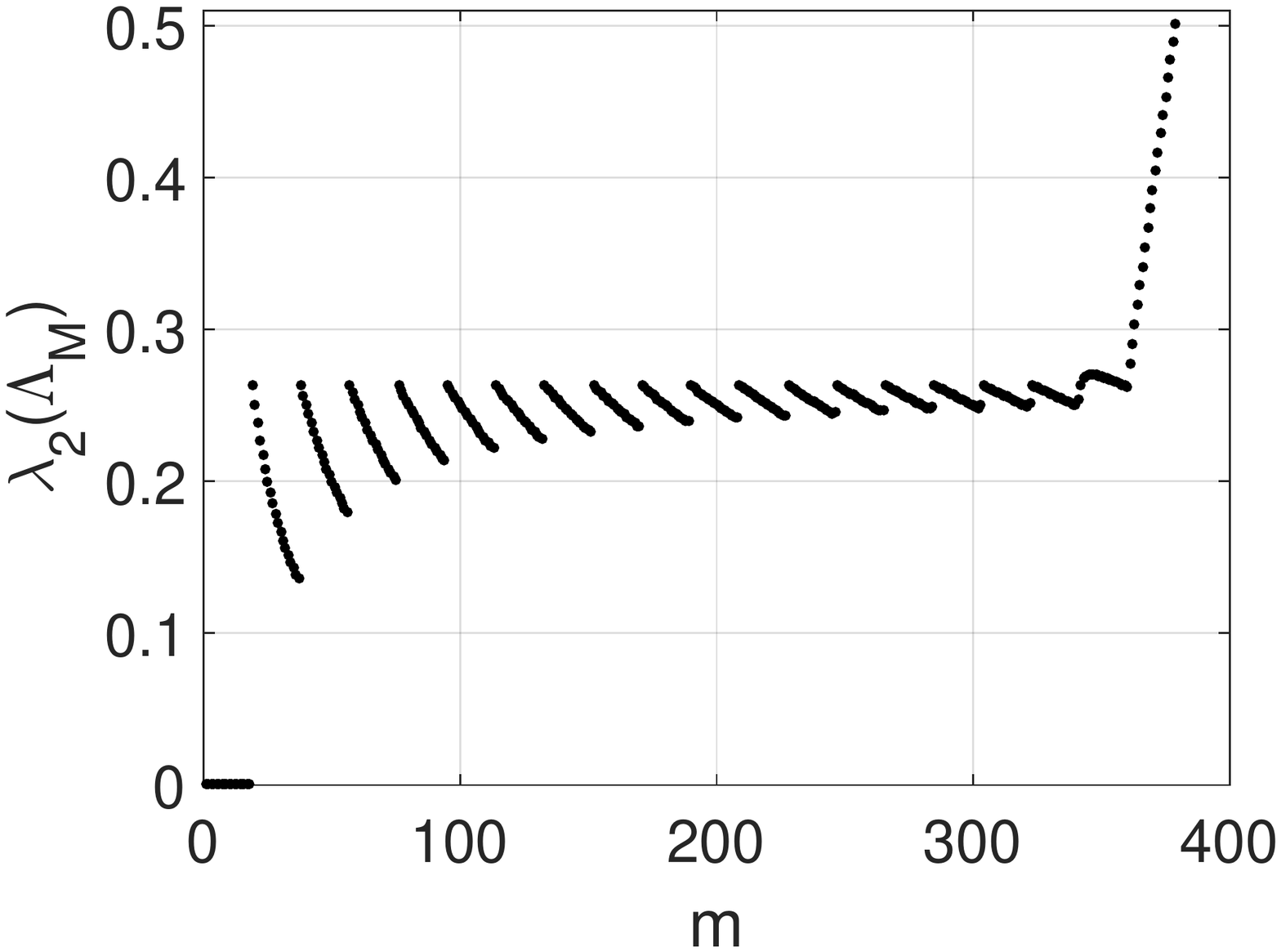}\hspace*{-0.2cm}
\includegraphics[width=0.49\textwidth]{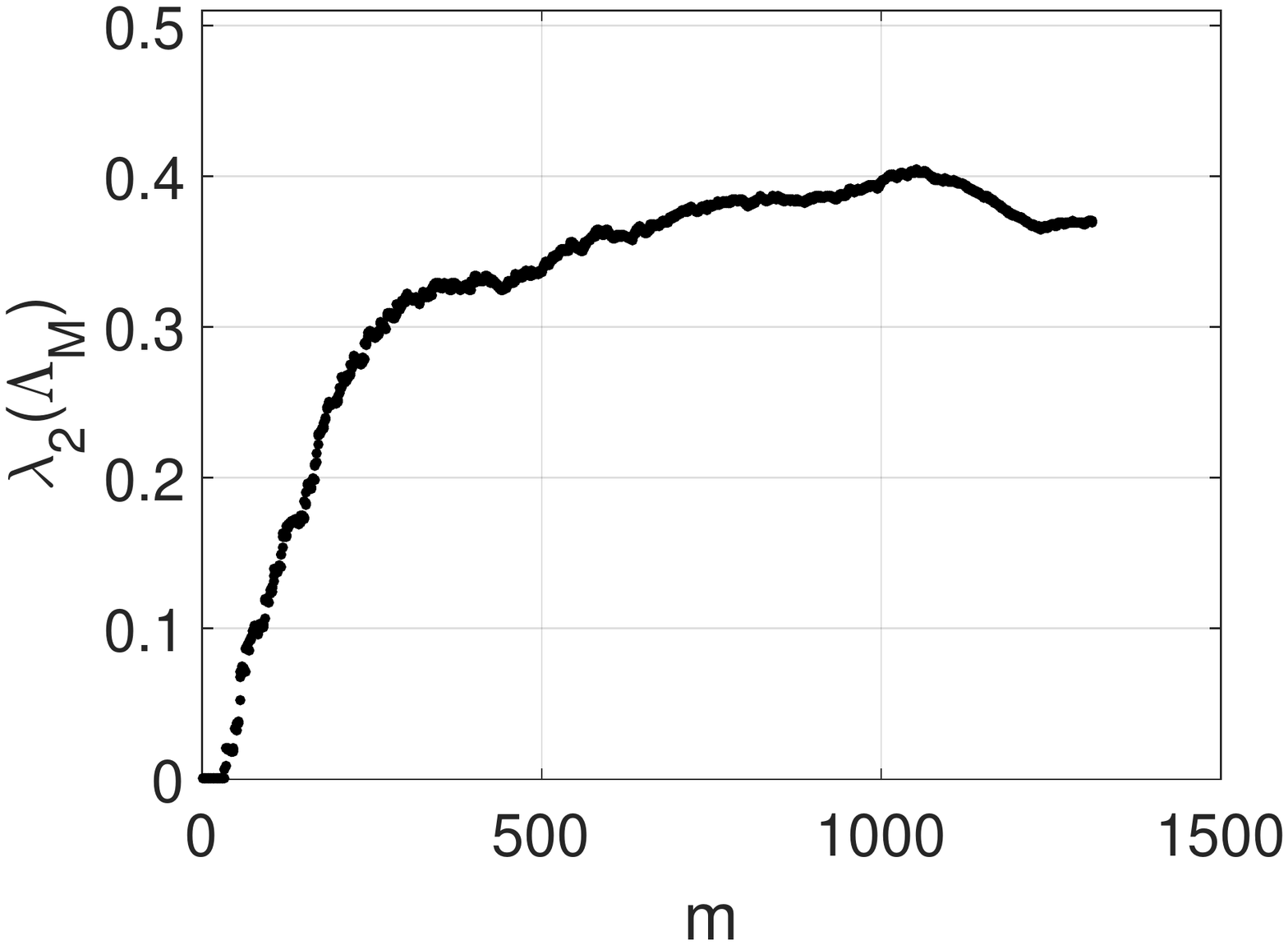}
\vspace*{-2cm}
\caption{Fiedler value of the weighted-adjacency matrices for the game fixtures of two sports in the season 2014-2015: (left) football Barclays premier league, and (right) basketball NBA league.}
\label{fig:premier-nba}
\end{figure}

\paragraph{Crowdsourcing contests} We consider participation of users in contests of two competition-based online labour platforms: (i) online platform for software development TopCoder and (ii) online platform for various kinds of tasks Tackcn. We refer to coders in TopCoder and workers in Taskcn as users. We consider contests of different categories observed in year 2012. In both these systems, the participation in contests is according to choices made by users. 

For each set of tasks of given category, we conduct the following analysis. We consider a conditioned dataset that consists only of a set of top-$n$ users with respect to the number of contests they participated in given year, and of all contests attended by at least two users from this set. We then evaluate Fiedler value of the weighted-adjacency matrix for parameter $n$ ranging from $2$ to the smaller of $100$ or the total number of users. Our analysis reveals that the Fiedler value tends to decrease with $n$. This indicates that the larger the number of users included, the less connected the weighted-adjacency matrix is. See the top plots in Figure~\ref{fig:TopCoder}. 

We also evaluated the smallest number of contests from the beginning of the year that is needed for the Fiedler value of the weighted-adjacency matrix to assume a strictly positive value. See the bottom plots in Figure~\ref{fig:TopCoder}. We observe that this threshold number of contests tends to increase with the number of top users considered. There are instances for which this threshold substantially increases for some number of the top users. This, again, indicates that the algebraic connectivity of the weighted-adjacency matrices tends to decrease with the number of top users considered.

\begin{figure}[t]
\centering
\vspace*{-1.8cm}
\includegraphics[width=0.49\textwidth]{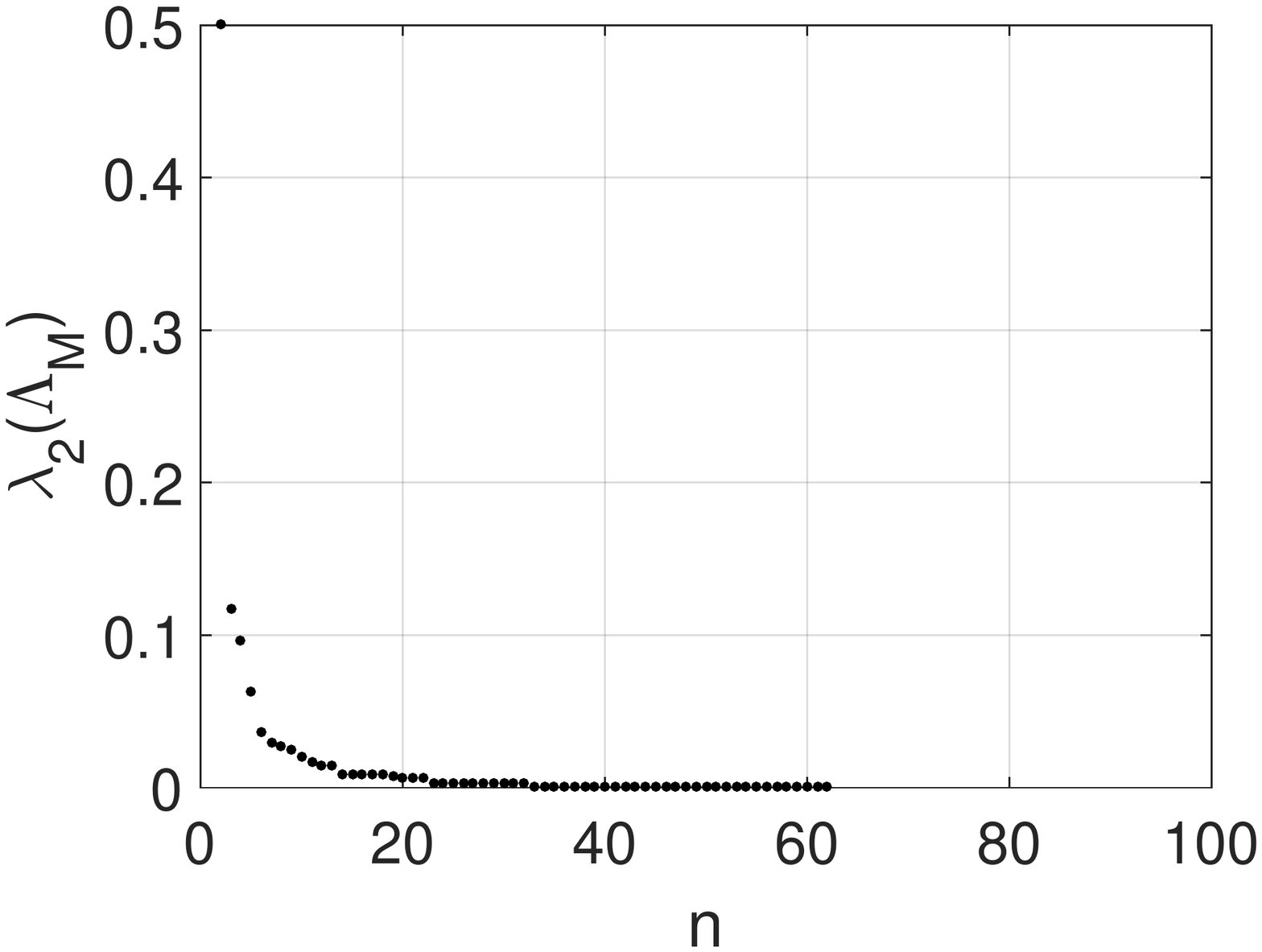}\hspace*{-0.2cm}
\includegraphics[width=0.49\textwidth]{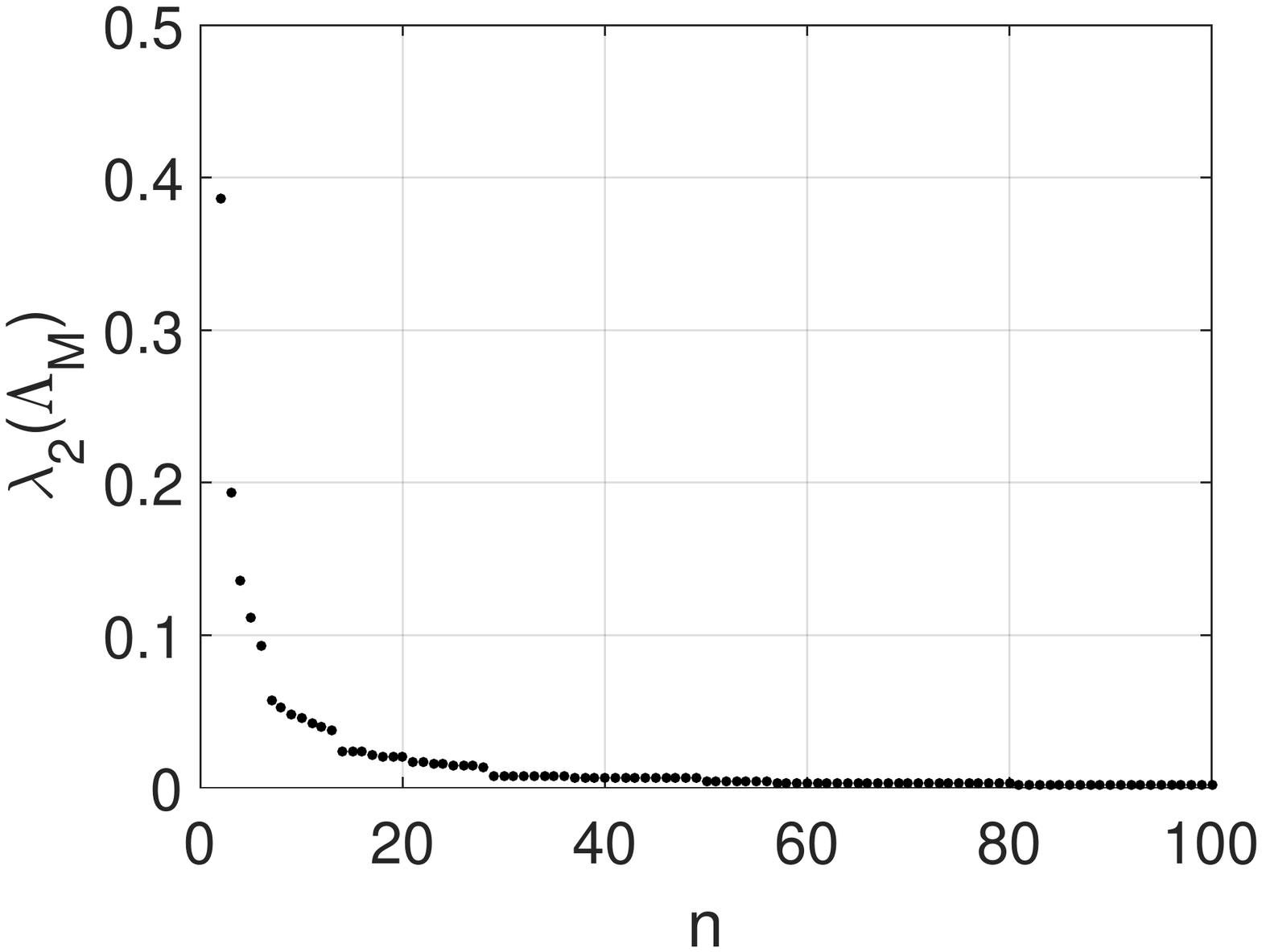}\hspace*{-0.2cm}
\\\vspace*{-4cm}
\includegraphics[width=0.49\textwidth]{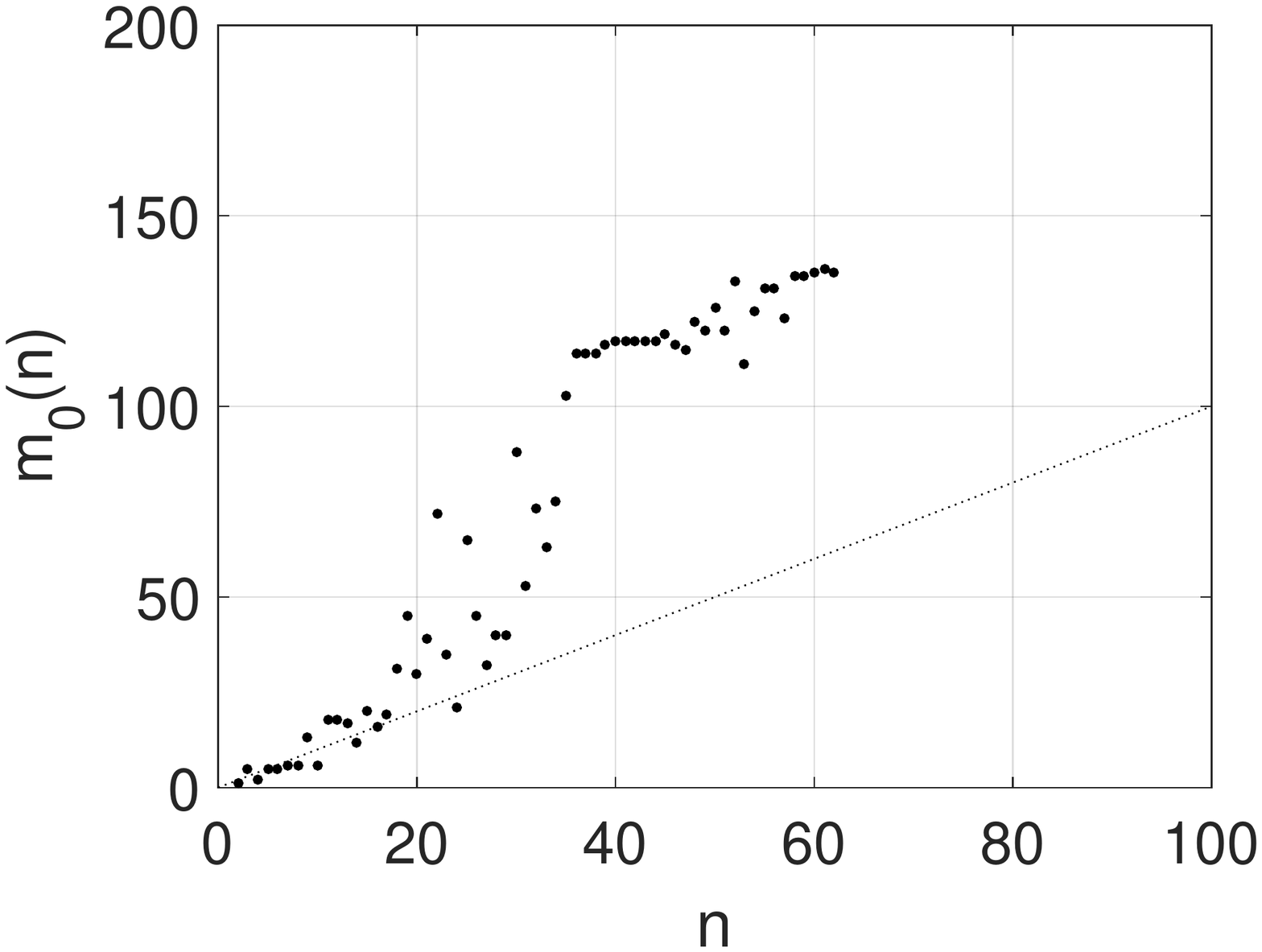}\hspace*{-0.2cm}
\includegraphics[width=0.49\textwidth]{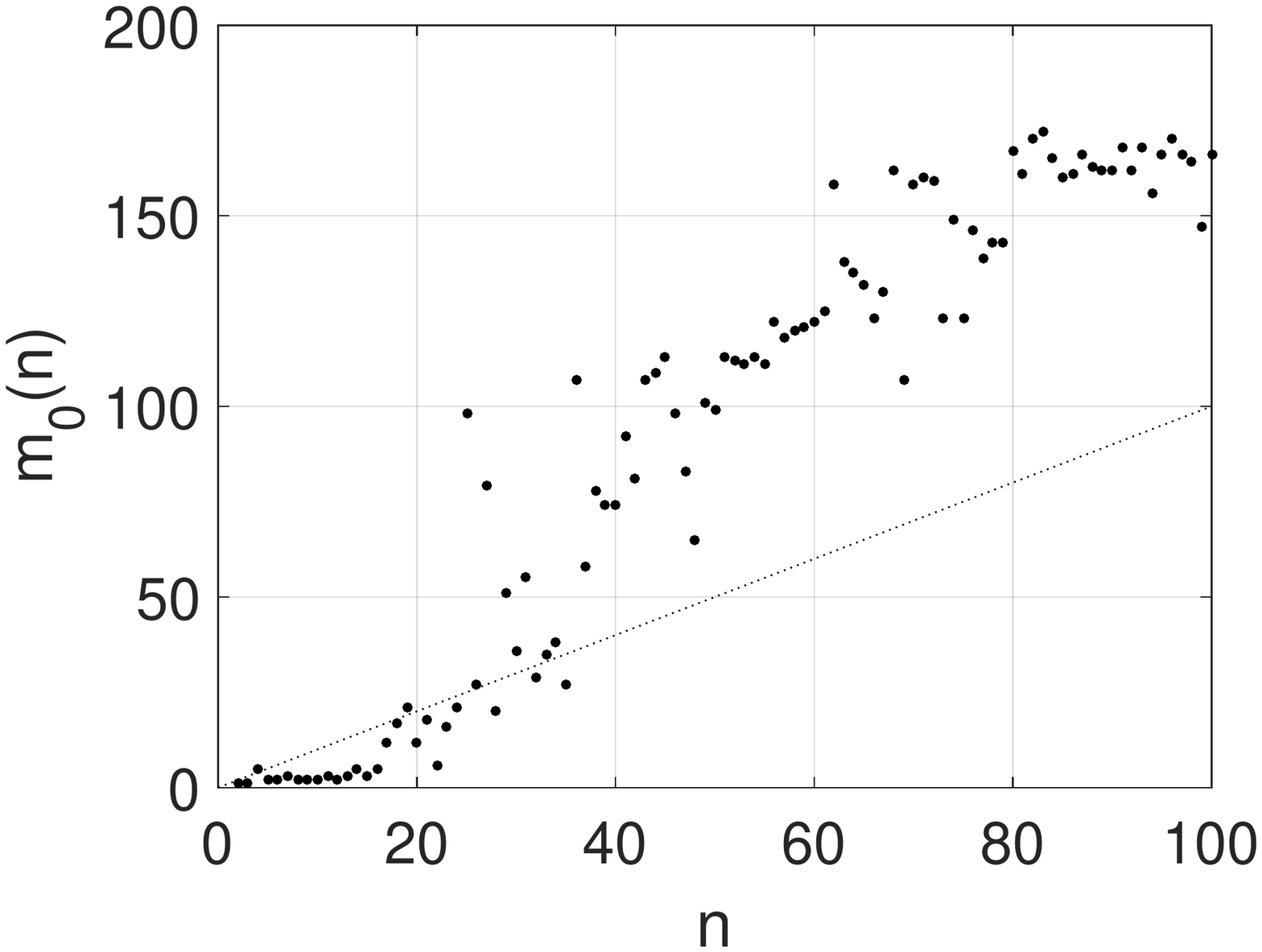}\hspace*{-0.2cm}
\vspace*{-1.5cm}
\caption{(Left) Topcoder data restricted to top-$n$ coders and (Right) same as left but for Taskcn, for Design and Website task categories, respectively. The top plots show the Fiedler value and the bottom plots show the minimum number of contests to observe a strictly positive Fiedler value.}
\label{fig:TopCoder}
\end{figure}

%\begin{figure*}[t]
%\centering
%\vspace*{-1cm}
%\includegraphics[width=0.2\textwidth]{Top_lambda2_TopCoder_pCat0.pdf}\hspace*{-0.2cm}
%\includegraphics[width=0.2\textwidth]{Top_lambda2_TopCoder_pCat1.pdf}\hspace*{-0.2cm}
%\includegraphics[width=0.2\textwidth]{Top_lambda2_TopCoder_pCat2.pdf}\hspace*{-0.2cm}
%\includegraphics[width=0.2\textwidth]{Top_lambda2_TopCoder_pCat3.pdf}\hspace*{-0.2cm}
%\includegraphics[width=0.2\textwidth]{Top_lambda2_TopCoder_pCat4.pdf}\\\vspace*{-2cm}
%\includegraphics[width=0.2\textwidth]{Top_mzero_TopCoder_pCat0.pdf}\hspace*{-0.2cm}
%\includegraphics[width=0.2\textwidth]{Top_mzero_TopCoder_pCat1.pdf}\hspace*{-0.2cm}
%\includegraphics[width=0.2\textwidth]{Top_mzero_TopCoder_pCat2.pdf}\hspace*{-0.2cm}
%\includegraphics[width=0.2\textwidth]{Top_mzero_TopCoder_pCat3.pdf}\hspace*{-0.2cm}
%\includegraphics[width=0.2\textwidth]{Top_mzero_TopCoder_pCat4.pdf}
%\vspace*{-1cm}
%\caption{Topcoder data restricted to top-$n$ coders: (top) Fiedler value, (bottom) minimum number of contests for strictly positive Fiedler value.}
%\label{fig:TopCoder}
%\end{figure*}

%\begin{figure*}[t]
%\centering
%\vspace*{-3cm}
%%\includegraphics[width=0.25\textwidth]{lambda_Tackcn_Top10_pCat1.pdf}
%\includegraphics[width=0.45\textwidth]{lambda_Tackcn_Top10_pCat2.pdf}
%%\includegraphics[width=0.25\textwidth]{lambda_Tackcn_Top10_pCat3.pdf}
%\includegraphics[width=0.45\textwidth]{lambda_Tackcn_Top10_pCat4.pdf}
%\vspace*{-2.5cm}
%\caption{Taskcn.}
%\label{fig:taskcn}
%\end{figure*}

\section{Conclusion}
\label{sec:conc}
The results of this paper elucidate how the estimation accuracy of a Thurstone choice model parameter depends on the given model and the structure of comparison sets. They show that a key factor is the algebraic connectivity of a weighted-adjacency matrix, which is specific to given model. It is shown that for a large class of Thurstone choice models, including well-known instances, there is a diminishing returns decrease of the estimation error with the cardinality of comparison sets at a slow rate for comparison sets of three of more items. 

The results provide guidelines to the designers of competition schedules such as to ensure that a schedule has a well-connected weighted-adjacency matrix and to expect limited estimation accuracy gains by enlarging the size of comparison sets.

\bibliography{ref}

\begin{thebibliography}{32}
\providecommand{\natexlab}[1]{#1}
\providecommand{\url}[1]{\texttt{#1}}
\expandafter\ifx\csname urlstyle\endcsname\relax
  \providecommand{\doi}[1]{doi: #1}\else
  \providecommand{\doi}{doi: \begingroup \urlstyle{rm}\Url}\fi

\bibitem[Bradley and Terry(1952)]{BT52}
Ralph~Allan Bradley and Milton~E. Terry.
\newblock Rank analysis of incomplete block designs: I. method of paired
  comparisons.
\newblock \emph{Biometrika}, 39\penalty0 (3/4):\penalty0 324--345, Dec 1952.

\bibitem[Bradley and Terry(1954)]{BT54}
Ralph~Allan Bradley and Milton~E. Terry.
\newblock Rank analysis of incomplete block designs: {II}. additional tables
  for the method of paired comparisons.
\newblock \emph{Biometrika}, 41\penalty0 (3/4):\penalty0 502--537, Dec 1954.

\bibitem[Elo(1978)]{E78}
Arpad~E. Elo.
\newblock \emph{The Rating of Chessplayers}.
\newblock Ishi Press International, 1978.

\bibitem[Fiedler(1973)]{F73}
Miroslav Fiedler.
\newblock Algebraic connectivity of graphs.
\newblock \emph{Czechoslovak Mathematical Journal}, 23\penalty0 (98):\penalty0
  298--305, 1973.

\bibitem[Fiedler(1989)]{F89}
Miroslav Fiedler.
\newblock Laplacian of graphs and algebraic connectivity.
\newblock \emph{Combinatorics and Graph Theory}, 25:\penalty0 57--70, 1989.

\bibitem[Graepel et~al.(2006)Graepel, Minka, and Herbrich]{GMH07}
Thore Graepel, Tom Minka, and Ralf Herbrich.
\newblock Trueskill(tm): A bayesian skill rating system.
\newblock In \emph{Proc. of NIPS 2006}, volume~19, pages 569--576, 2006.

\bibitem[Hajek et~al.(2014)Hajek, Oh, and Xu]{hajek2014minimax}
Bruce Hajek, Sewoong Oh, and Jiaming Xu.
\newblock Minimax-optimal inference from partial rankings.
\newblock In \emph{Proc. of NIPS 2014}, pages 1475--1483, 2014.

\bibitem[Hayes(2003)]{H03}
Thomas~P. Hayes.
\newblock A large-deviation inequality for vector-valued martingales.
\newblock 2003.
\newblock URL
  \url{http://www.cs.unm.edu/~hayes/papers/VectorAzuma/VectorAzuma20030207.pdf}.

\bibitem[Horn and Johnson(1985)]{horn}
Roger~A. Horn and Charles~R. Johnson.
\newblock \emph{Matrix Analysis}.
\newblock Cambridge University Press, 1985.

\bibitem[Hunter(2004)]{hunter2004mm}
David~R Hunter.
\newblock Mm algorithms for generalized bradley-terry models.
\newblock \emph{Annals of Statistics}, pages 384--406, 2004.

\bibitem[Khetan and Oh(2016{\natexlab{a}})]{khetan2016}
Ashih Khetan and Sewoong Oh.
\newblock Computational statistical tradeoffs in learning to rank.
\newblock In \emph{Proc. of NIPS 2016}, 2016{\natexlab{a}}.

\bibitem[Khetan and Oh(2016{\natexlab{b}})]{khetan2016data}
Ashish Khetan and Sewoong Oh.
\newblock Data-driven rank breaking for efficient rank aggregation.
\newblock \emph{Journal of Machine Learning Research}, 17\penalty0
  (193):\penalty0 1--54, 2016{\natexlab{b}}.

\bibitem[Luce(1959)]{L59}
R.~Duncan Luce.
\newblock \emph{Individual Choice Behavior: A Theoretical Analysis}.
\newblock John Wiley \& Sons, 1959.

\bibitem[Maydeau-Olivares(1999)]{M99}
Albert Maydeau-Olivares.
\newblock Thurstonian modeling of ranking data via mean and covariance
  structure analysis.
\newblock \emph{Psychometrika}, 64\penalty0 (3):\penalty0 325--340, 1999.

\bibitem[Maystre and Grossglauser(2015)]{maystre2015fast}
Lucas Maystre and Matthias Grossglauser.
\newblock Fast and accurate inference of plackett--luce models.
\newblock In \emph{Advances in Neural Information Processing Systems}, pages
  172--180, 2015.

\bibitem[McCullagh and Nelder(1989)]{MN89}
P.~McCullagh and J.~A. Nelder.
\newblock \emph{Generalized Linear Models}.
\newblock Chapman \& Hall, New York, 2 edition, 1989.

\bibitem[Murphy(2012)]{M12}
Kevin~P. Murphy.
\newblock \emph{Machine Learning: A Probabilistic Perspective}.
\newblock MIT Press, 2012.

\bibitem[Negahban et~al.(2012)Negahban, Oh, and Shah]{negahban2012rank}
Sahand Negahban, Sewoong Oh, and Devavrat Shah.
\newblock Iterative ranking from pair-wise comparisons.
\newblock In \emph{Proc. of NIPS 2012}, pages 2483--2491, 2012.

\bibitem[Nelder and Wedderburn(1972)]{NW72}
J.~A. Nelder and R.~W. Wedderburn.
\newblock Generalized linear models.
\newblock \emph{Journal of the Royal Statistical Society, Series A},
  135:\penalty0 370--384, 1972.

\bibitem[Plackett(1975)]{P75}
Robin~L. Plackett.
\newblock The analysis of permutations.
\newblock \emph{Journal of the Royal Statistical Society. Series C (Applied
  Statistis)}, 24\penalty0 (2):\penalty0 193--202, 1975.

\bibitem[Rajkumar and Agarwal(2014)]{RA14}
Arun Rajkumar and Shivani Agarwal.
\newblock A statistical convergence perspective of algorithms for rank
  aggregation from pairwise data.
\newblock In \emph{Proc. of {ICML} 2014}, pages 118--126, 2014.

\bibitem[Shah et~al.(2016)Shah, Balakrishnan, Bradley, Parekh, Ramchandran, and
  Wainwright]{SBBPRW16}
Nihar~B. Shah, Sivaraman Balakrishnan, Joseph Bradley, Abhay Parekh, Kannan
  Ramchandran, and Martin~J. Wainwright.
\newblock Estimation from pairwise comparisons: Sharp minimax bounds with
  topology dependence.
\newblock \emph{J. Mach. Learn. Res.}, 17\penalty0 (1):\penalty0 2049--2095,
  January 2016.

\bibitem[Simons and Yao(1999)]{SY99}
Gordon Simons and Yi-Ching Yao.
\newblock Asymptotics when the number of parameters tends to infinity in the
  bradley-terry model for paired comparisons.
\newblock \emph{The Annals of Statistics}, 27\penalty0 (3):\penalty0
  1041--1060, 1999.

\bibitem[Soufiani et~al.(2014)Soufiani, Parkes, and Xia]{soufiani14}
Azari Soufiani, David Parkes, and Lirong Xia.
\newblock Computing parametric ranking models via rank-breaking.
\newblock In \emph{Proceedings of the 31st International Conference on Machine
  Learning (ICML)}, pages 360--368, 2014.

\bibitem[Soufiani et~al.(2013)Soufiani, Chen, Parkes, and
  Xia]{soufiani2013generalized}
Hossein~Azari Soufiani, William Chen, David~C Parkes, and Lirong Xia.
\newblock Generalized method-of-moments for rank aggregation.
\newblock In \emph{Proc. of NIPS 2013}, 2013.

\bibitem[Stern(1992)]{S92}
Hal Stern.
\newblock Are all linear paired comparison models empirically equivalent?
\newblock \emph{Mathematical Social Sciences}, 23\penalty0 (1):\penalty0
  103--117, 1992.

\bibitem[Thurstone(1927)]{T27}
L.~L. Thurstone.
\newblock A law of comparative judgment.
\newblock \emph{Psychological Review}, 34\penalty0 (2):\penalty0 273--286,
  1927.

\bibitem[Tropp(2015)]{tropp2015introduction}
Joel~A Tropp.
\newblock An introduction to matrix concentration inequalities.
\newblock \emph{arXiv preprint arXiv:1501.01571}, 2015.

\bibitem[{Vojnovi\' c}(2016)]{V16}
Milan {Vojnovi\' c}.
\newblock \emph{Contest Theory: Incentive Mechanisms and Ranking Methods}.
\newblock Cambridge University Press, 2016.

\bibitem[Yellott(1977)]{Y77}
John~I. Yellott.
\newblock The relationship between {Luce's} choice axiom, {Thurstone's} theory
  of comparative judgement and the double exponential distribution.
\newblock \emph{Journal of Mathematical Psychology}, 15:\penalty0 109--144,
  1977.

\bibitem[Yun and Proutiere(2014)]{yun2014community}
Se-Young Yun and Alexandre Proutiere.
\newblock Community detection via random and adaptive sampling.
\newblock In \emph{COLT}, pages 138--175, 2014.

\bibitem[Zermelo(1929)]{Z29}
E.~Zermelo.
\newblock Die berechnung der turnier-ergebnisse als ein maximumproblem der
  wahrscheinlichkeitsrechnung.
\newblock \emph{Math. Z.}, 29:\penalty0 436--460, 1929.

\end{thebibliography}

\newpage
\appendix

\section{Background Material}
\label{sec:back}

\paragraph{Location of Eigenvalues} We make note of the well-known Ger\v sgorin circles theorem, which we state as at the following lemma:

\begin{lemma} Let $\vec{A}\in \reals^{n\times n}$, then all eigenvalues of $\vec{A}$ are located in the union of $n$ discs
$$
\cup_{i=1}^n \left\{z\in C: |z-a_{ii}| \leq \sum_{j\neq i} |a_{i,j}|\right\}.
$$
\label{lem:gersgorin}
\end{lemma}

\paragraph{Properties of Positive Definite and Laplacian Matrices} A symmetric matrix $\vec{A}\in \reals^{n\times n}$ is said to be \emph{positive semidefinite} if $
\vec{x}^\top \vec{A} \vec{x} \geq 0$ for all nonzero $\vec{x}\in \reals^n$. If the inequality is replaced with strict inequality, $\vec{A}$ is said to be \emph{positive definite}. 

Each eigenvalue of a positive definite matrix is a positive real number. Each eigenvalue of a positive semidefinite matrix is a nonnegative real number. 

For two matrices $\vec{A},\vec{B}\in \reals^{n\times n}$, we write $\vec{A} \succeq \vec{B}$ for the \emph{positive semidefinite ordering}, which means that $\vec{A} - \vec{B}$ is a positive semidefinite matrix. Similarly, we write $\vec{A} \succ \vec{B}$ for the \emph{positive definite ordering}, which means that $\vec{A}-\vec{B}$ is a positive definite matrix. Note that $\vec{A} \succeq \vec{0}$ means that $\vec{A}$ is a positive semidefinite matrix, and $\vec{A}\succ \vec{0}$ means that $\vec{A}$ is a positive definite matrix.

We note the following ordering relations for eigenvalues of two positive definite matrices that satisfy the positive semidefinite ordering (e.g., see Corollary~7.7.4~\cite{horn}).

\begin{lemma} For any two positive definite matrices $\vec{A},\vec{B}\in \reals^{n\times n}$ such that $\vec{A}\succeq \vec{B}$, $\lambda_i(\vec{A})\geq \lambda_i(\vec{B})$ for all $i = 1,2,\ldots,n$.
\label{lem:EigAsuccB}
\end{lemma}

For any $\vec{A}\in \reals^{n\times n}$ with zero diagonal such that $\vec{A}\geq \vec{0}$, the Laplacian matrix $L_{\vec{A}}$ is positive semidefinite. This follows from the localization of eigenvalues by the Ger\v sgorin circles theorem. 

\begin{lemma} If $\vec{A}\in \reals^n$ is a symmetric matrix with zero diagonal, then
\begin{enumerate}
\item [(i)] $\lambda_1(L_{\vec{A}}) = 0$, and
\item [(ii)] $\lambda_i(L_{\vec{A}}) = \lambda_{i-1}(\vec{U}^\top L_{\vec{A}} \vec{U})$, for $i\in \{2,3,\ldots,n\}$
\end{enumerate}
where $\vec{U}\in \vec{R}^{n\times n-1}$ is any matrix such that $\vec{U}^\top \vec{U} = \vec{I}$.
\label{lem:Leig}
\end{lemma}

We shall use the following property of real symmetric matrices:

\begin{lemma} If $\vec{A},\vec{B}\in \reals^{n\times n}$ are real symmetric matrices with zero diagonals such that $\vec{B} \geq \vec{A}$ where the inequality holds element-wise, then $L_{\vec{B}} \succeq L_{\vec{A}}$. 
\label{lem:laplaceposdef}
\end{lemma}

\paragraph{Chernoff Tail Bounds} The following bounds follow from the Chernoff bound:

\begin{lemma} Suppose that $X$ is a sum of $m$ independent Bernoulli random variables each with mean $p$, then if $q \leq p \leq 2q$,
\begin{equation}
\Pr[X \leq q m] \leq \exp\left(-\frac{(q-p)^2}{4q} m\right)
\label{equ:che1}
\end{equation}
and, if $p \leq q$, 
\begin{equation}
\Pr[X \geq q m] \leq \exp\left(-\frac{(q-p)^2}{4q} m\right).
\label{equ:che2}
\end{equation}
\label{lem:chernoff}
\end{lemma}

\begin{proof} We prove only prove (\ref{equ:che2}) as (\ref{equ:che1}) follows by similar arguments. By the Chernoff's bound, for every $s > 0$,
\begin{eqnarray*}
\Pr[X \geq q m] & \leq & e^{-s qm}\E[e^{sX}]\\
&=& e^{-sqm} (1-p+pe^s)^m\\
&=& e^{-m h(s)}
\end{eqnarray*}
where $h(s) = qs - \log(1-p+pe^s)$.
Since $\log(1-x) \leq -x$ for all $x\in \reals$, we have $h(s) \geq qs +p - p e^s$. Take $s = s^* := \log(q/p)$ to obtain $h(s^*) \geq q \log(q/p) + p - q$.

Now, let $\epsilon = q-p$, and note that $q \log(q/p) + p - q := g(\epsilon)$ where $g(\epsilon) = q \log(q/(q-\epsilon) - \epsilon$. Since $g'(\epsilon) = q/(q-\epsilon) - 1 
= \epsilon/(q-\epsilon) \geq \epsilon/(2q)$, we have $g(\epsilon) = \int_0^\epsilon g'(x)dx \geq \epsilon^2/(4q)$.

Hence, it follows that $h(s^*) \geq (p-q)^2/(4q)$, and, thus
$$
\Pr[X\geq q m] \leq \exp\left(-\frac{1}{4q}(p-q)^2\right).
$$
\end{proof}

\section{Proof of Lemma~\ref{lem:mle-taylor}}

Let $\Delta = \widehat{\vec{x}} - \vec{x}^\star$. By the Taylor expansion, we have
\begin{equation}
g(\widehat{\vec{x}}) \ge g(\vec{x}^\star) + \nabla g (\vec{x}^\star)^\top \Delta  +\frac{1}{2} \min_{\alpha \in [0,1]} \Delta^\top \nabla^2 g (\vec{x}^\star+\alpha \Delta) \Delta.
\label{eq:taylor-pair}
\end{equation}

Since $g(\widehat{\vec{x}}) \leq g(\vec{x}^\star)$, we have
$$
\min_{\alpha \in [0,1]} \Delta^\top \nabla^2 g (\vec{x}^\star+\alpha \Delta) \Delta \leq -\nabla g (\vec{x}^\star)^\top \Delta.
$$
Hence, 
\begin{equation}
\min_{\vec{x} \in \calX} \Delta^\top \nabla^2 g(\vec{x}) \Delta \le 2  \left\|\nabla g(\vec{x}^\star)\right\|_2 \| \Delta \|_2.
\label{eq:taylor-2}
\end{equation}

Fix an arbitrary $\vec{x}\in \calX$. From condition (i) it follows that $\nabla^2 g(\vec{x})$ has eigenvalue $0$ with eigenvector $\vec{1}$. Combining with condition (ii), we have

\begin{equation}
0=\lambda_1(\nabla^2 g(\vec{x})) < \lambda_2(\nabla^2 g(\vec{x})) \le \dots\le \lambda_n(\nabla^2 g(\vec{x})).
\label{equ:lam}
\end{equation}

Let $\vec{U} = [\vec{u}_1,\vec{u}_2,\ldots,\vec{u}_n] \in \reals^{n\times n}$ where $\vec{u}_1,\vec{u}_2,\ldots,\vec{u}_n$ are ortonormal eigenvectors of $\nabla^2 g(\vec{x})$, which correspond to eigenvalues $\lambda_1(\nabla^2 g(\vec{x})), \lambda_2(\nabla^2 g(\vec{x})), \ldots, \lambda_n(\nabla^2 g(\vec{x}))$, respectively. Note that 
\begin{equation}
\vec{u}_1^\top \Delta = 0. 
\label{equ:ud}
\end{equation}

Let $\Lambda = \mathrm{diag}(\lambda_1(\nabla^2 g(\vec{x})), \lambda_2(\nabla^2 g(\vec{x})), \ldots, \lambda_n(\nabla^2 g(\vec{x})))$.

We have the following relations:
\begin{eqnarray*}
\Delta^\top \nabla^2 g(\vec{x})\Delta & = & \Delta^\top \vec{U}\Lambda\vec{U}^\top\Delta\\
&=& \sum_{i=1}^n \lambda_i(\nabla^2 g(\vec{x}))|(\vec{U}^\top\Delta)_i|^2\\
&=& \sum_{i=1}^n \lambda_i(\nabla^2 g(\vec{x}))|\vec{u}_i^\top\Delta|^2\\
&=& \sum_{i=2}^n \lambda_i(\nabla^2 g(\vec{x}))|\vec{u}_i^\top\Delta|^2\\
&\geq & \lambda_2(\nabla^2 g(\vec{x})) \sum_{i=2}^n |\vec{u}_i^\top\Delta|^2\\
&= & \lambda_2(\nabla^2 g(\vec{x})) \sum_{i=1}^n |\vec{u}_i^\top\Delta|^2\\
&= & \lambda_2(\nabla^2 g(\vec{x})) ||\Delta||_2^2
\end{eqnarray*}
where we use the properties in (\ref{equ:lam}) and (\ref{equ:ud}).

Hence, it follows that
\begin{equation}
\min_{\vec{x}\in \calX}\Delta^\top \nabla^2 g(\vec{x})\Delta \geq ||\Delta||_2^2 \min_{\vec{x}\in \calX} \lambda_2(\nabla^2 g(\vec{x})).
\label{eq:delta-hessian}
\end{equation}

By combining \eqref{eq:delta-hessian} and \eqref{eq:taylor-2}, we conclude the proof of the lemma.

%\section{Proof of Theorem~\ref{thm:mle}}

%\subsection{Proof of Lemma~\ref{lem:LT21}}

%\subsection{Proof of Lemma~\ref{lem:LT22}}

\section{Proof of Theorem~\ref{thm:full}} \label{sec:proof-full}

The log-likelihood function in (\ref{equ:loglik0}) can be written in the following more explicit form:
\begin{equation}
\ell(\theta) = \sum_{t=1}^m \log\left(\frac{e^{\theta_{y_t}}}{\sum_{v\in S_t}e^{\theta_v}}\right).
\label{equ:llikluce}
\end{equation}
Since $\nabla^2(-\ell(\theta))\vec{1} = \vec{0}$, for all $\theta\in \reals^n$, under assumption that $\min_{\theta\in \Theta}\lambda_2 \left(-\nabla^2 \ell (\theta) \right)>0$, the upper bound in Lemma~\ref{lem:mle-taylor} holds. This combined with the following two lemmas yields the statement of the theorem. 

%\subsection{Proof of Lemma~\ref{lem:L11}}

\begin{lemma} The following lower bound holds:
\begin{equation}
\min_{\theta \in \Theta}\lambda_2 \left(\nabla^2 (-\ell (\theta)) \right) \ge \frac{1}{k^2 e^{4b}}\frac{m}{n} \lambda_2 (L_\vec{M}). 
\label{eq:w-l}
\end{equation}
\label{lem:L11}
\end{lemma}

\begin{proof}
From (\ref{equ:llikluce}), we have for $i,j\in \{1,2,\ldots,n\}$ such that $j\neq i$,
$$
\frac{\partial^2 \ell (\theta)}{\partial \theta_i \partial \theta_j} = \sum_{t:\{ i,j\}\subseteq S_t} p_{i,S_t}(\theta)p_{j,S_t}(\theta) 
\hbox{ and }
 \frac{\partial^2 \ell(\theta)}{\partial \theta_i^2} = -\sum_{j\neq i} \frac{\partial^2 \ell (\theta)}{\partial \theta_i \partial \theta_j}.
$$

Since for all $i\neq j$,
$$
\frac{\partial^2 \ell (\theta)}{\partial \theta_i \partial \theta_j} \ge \frac{1}{k^2 e^{4b}}m_{i,j}
$$
we have that $\nabla^2(-\ell (\theta)) -  \frac{1}{k^2 e^{4b}}\frac{m}{n} L_\vec{M}$ is a Laplacian matrix of a matrix with nonnegative elements. Every such Laplacian matrix is positive semidefinite, hence 
\begin{equation}
\nabla^2 (-\ell (\theta)) \succeq \frac{1}{k^2 e^{4b}}\frac{m}{n} L_\vec{M}.
\label{eq:pdm-ell}
\end{equation}

From \eqref{eq:pdm-ell}, we conclude \eqref{eq:w-l}.
\end{proof}

%\subsection{Proof of Lemma~\ref{lem:L12}}

\begin{lemma} With probability at least $1-2/n$,
\begin{equation}
\left\| \nabla (-\ell (\theta^\star)) \right\|_2 \le  2\sqrt{m (\log (n)+2)}. 
\label{eq:w-u}
\end{equation}
\label{lem:L12}
\end{lemma}

\begin{proof}
From (\ref{equ:llikluce}), we have
\begin{equation}
\nabla(-\ell(\theta)) = \sum_{t=1}^m \nabla(-\log(p_{y_t,S_t}(\theta)))
\label{equ:nll}
\end{equation}
where $p_{y_t,S_t}(\theta) = e^{\theta^\star_{y_t}}/\sum_{v\in S_t}e^{\theta_v}$. 

It is straightforward to derive that $\nabla(-\log(p_{y_t,S_t}(\theta)))$ has elements given by
\begin{equation}
\frac{\partial (-\log(p_{y_t,S_t}(\theta)))}{\partial \theta_i}  = \left \{
\begin{array}{ll}
-(1-p_{i,S_t}(\theta)), & \hbox{ if } i = y_t\\
p_{i,S_t}(\theta), & \hbox{ if } i\in S_t\setminus \{y_t\}\\
0, & \hbox{ otherwise}.
\end{array}
\right .
\label{equ:nllik}
\end{equation}

From (\ref{equ:nll}), $\nabla \ell(\theta)$ is a sum of independent random vectors in $\reals^n$ that satisfy for all $t\in \{1,2,\ldots,m\}$,
\begin{equation}
\E\left[\nabla (-\log(p_{y_t,S_t}(\theta^\star)))\right] = \vec{0} 
\label{eq:w-az1}
\end{equation}
and
\begin{equation}
\left\|\nabla(-\log(p_{y_t,S_t}(\theta^\star)))\right\|_2 \leq \sqrt{2}.
\label{eq:w-az2}
\end{equation}

The last two relations are easy to establish using (\ref{equ:nllik}) as follows. Equation~(\ref{eq:w-az1}) holds because for every $i\in N$,
\begin{eqnarray*}
\E\left[\frac{\partial (-\log(p_{y_t,S_t}(\theta^\star)))}{\partial \theta_i}\right]
% &=& \Pr[y_t = i] \frac{\partial}{\partial \theta_i} \log\left(p_{i,S_t}(\theta^\star)\right) + \sum_{j\in S_t\setminus \{y_t\}}\Pr[y_t=j] \frac{\partial}{\partial \theta_i} \log\left(p_{j,S_t}(\theta^\star)\right)\\
&=& -p_{i,S_t}(\theta^\star)(1-p_{i,S_t}(\theta^\star)) + (1-p_{i,S_t}(\theta^\star))p_{i,S_t}(\theta^\star) = 0.
\end{eqnarray*}

Equation~(\ref{eq:w-az1}) follows from
$$
\left\|\nabla (-\log(p_{y_t,S_t}(\theta^\star)))\right\|_2^2 = (1-p_{y_t,S_t}(\theta^\star))^2 + \sum_{j\in S_t\setminus \{y_t\}} p_{j,S_t}(\theta^\star)^2 \leq 2.
$$

By the vector Azuma-Hoeffding bound in Lemma~\ref{prop:azuma-hoeffding}, we have
$$
\Pr[\|\nabla(-\ell(\theta^\star))\|_2\geq 2\sqrt{m(\log(n)+2)}]  \leq \frac{2}{n}.
$$
which completes the proof of \eqref{eq:w-u}.
\end{proof}

\section{Proof of Theorem~\ref{thm:karyub}} \label{sec:proof-karyub}

Since $\nabla^2 (-\ell (\theta))$ is a Laplacian matrix, by condition {\bf A1}, 
$$
\nabla^2(-\ell (\theta)) \succeq A\nabla^2(-\ell ({\bm 0})) \hbox{ for all } \theta\in [-b,b]^n.
$$
Hence, in particular, 
\begin{equation}
\min_{\theta \in [-b,b]^n}\lambda_2(-\ell (\theta)) \geq A\lambda_2(\nabla^2(-\ell ({\bm 0}))).
\label{eq:semi}
\end{equation}

We have the following two lemmas.

%\subsection{Proof of Lemma~\ref{lem:khessi}}

\begin{lemma} If $\lambda_2(L_{\overline{\vec{M}}_{w^\star}}) \ge 32 (\sigma_{F,K}/C) (n\log(n))/m$, then with probability at least $1-1/n$,
$$
\lambda_2(\nabla^2(-\ell ({\bm 0}))) \ge \frac{1}{2}C \frac{m}{n} \lambda_2(L_{\overline{\vec{M}}_{w^\star}}).
$$
\label{lem:khessi}
\end{lemma}

\begin{proof}
$\nabla^2(-\ell({\bm 0}))$ is a sum of independent random matrices given by
$$
\nabla^2(-\ell({\bm 0})) = \sum_{t=1}^m \nabla^2 (-\log (p_{y_t,S_t}({\bm 0}))).
$$

We have the following relations:
\begin{eqnarray}
\E\left[\nabla^2 (-\ell({\bm 0}))\right] 
&=& \sum_{t=1}^m \E\left[\nabla^2(-\log (p_{y_t,S_t}({\bm 0})))\right]  \cr
&=& \sum_{t=1}^m \sum_{y\in S_t} p_{y,S_t}(\theta^\star)\nabla^2(-\log (p_{y,S_t}({\bm 0})))\cr
&\succeq & C \sum_{t=1}^m \sum_{y\in S_t} \frac{1}{|S_t|} \nabla^2(-\log (p_{y,S_t}({\bm 0})))\cr
& = & C \sum_{t=1}^m \sum_{y\in S_t} \frac{1}{|S_t|} L_{\vec{M}_{S_t}} \left(|S_t|\frac{\partial p_{|S_t|}(\vec{0})}{\partial x_1}\right)^2\nonumber 
\label{equ:temp}\\
&=& C \frac{m}{n}  L_{\overline{\vec{M}}_F}\nonumber
\label{eq:E-hessian}
\end{eqnarray}
where we use Lemma~\ref{lem:explog}.

Hence,
\begin{equation}
\lambda_2(\E[\nabla^2 (-\ell ({\bm 0}))]) \ge C \frac{m}{n} \lambda_2(L_{\overline{\vec{M}}_F}).
\label{eq:cher1}
\end{equation} 

By Lemma~\ref{lem:hessian-p2}, for all $t\in \{1,2,\ldots,m\}$,
\begin{equation}
\left\| \nabla^2 \log(p_{y_t,S_t} (\vec{0})) \right\|_2 \le \frac{2}{\gamma_{F, |S_t|}} \leq 2 \sigma_{F,K}.
\label{eq:cher2}
\end{equation}

Using the matrix Chernoff bound with $\epsilon = 1/2$, (\ref{eq:cher1}) and (\ref{eq:cher2}, we conclude the statement of the lemma.
\end{proof}

%\subsection{Proof of Lemma~\ref{lem:kgrad}} \label{sec:proof-lem12}

\begin{lemma} With probability at least $1-2/n$,
$$
\|\nabla \ell (\theta^\star) \|_2 \le  B\sqrt{\sigma_{F,K}}\sqrt{2m(\log(n) +2)}.
$$
\label{lem:kgrad}
\end{lemma}

\begin{proof} For every $S\subseteq N$ and $i,j\in S$ such that $j\neq i$,
\begin{equation}
\frac{\partial \log(p_{j,S} (\theta))}{\partial \theta_i} = \frac{1}{p_{j,S} (\theta)}\frac{\partial p_{j,S} (\theta)}{\partial \theta_i}.
\label{eq:id2}
\end{equation}
and
\begin{equation}
\frac{\partial \log(p_{i,S} (\theta))}{\partial \theta_i} = -\frac{1}{p_{i,S} (\theta)} \sum_{v\in S\setminus \{i\}}\frac{\partial p_{v,S} (\theta)}{\partial \theta_i}.
\label{eq:id1}
\end{equation}

From \eqref{eq:id2} and \eqref{eq:id1}, for every $t\in \{1,2,\ldots,m\}$, 
\begin{equation}
\E\left[\nabla \log p_{y_t,S_t} (\theta^\star) \right]  = {\bm 0}.
\label{eq:aver-grad}
\end{equation}

From \eqref{eq:id2} and \eqref{eq:id1}, for every $S\subseteq N$ such that $|S|=k\geq 2$ and $i,j\in S$ such that $i\neq j$, 
$$
\frac{\partial \log(p_{j,S}(\vec{0}))}{\partial \theta_i} = k \frac{\partial p_k(\vec{0})}{\partial x_1}
$$
$$
\frac{\partial \log(p_{i,S}(\vec{0}))}{\partial \theta_i} = k(k-1)\frac{\partial p_k(\vec{0})}{\partial x_1}.
$$

Hence, for every $t\in \{1,2,\ldots,m\}$ such that $|S_t|=k$
\begin{equation}
\left\|\nabla \log(p_{y_t,S_t} ({\bm 0})) \right\|_2^2  = k^3(k-1)\left(\frac{\partial p_k({\bm 0})}{\partial x_1}\right)^2 = \frac{1}{\gamma_{F,k}}.
\label{equ:l2bound}
\end{equation}

By condition {\bf A3} and (\ref{equ:l2bound}), for every $t\in \{1,2,\ldots,m\}$,
\begin{equation}
\|\nabla \log p_{y_t,S_t} (\theta^\star)\|_2^2 
\leq B^2 \|\nabla \log p_{y_t,S_t} ({\bm 0})\|_2^2 
\leq B^2 \sigma_{F,K}.
\label{eq:aver-grad2}
\end{equation}

Using (\ref{eq:aver-grad}) and (\ref{eq:aver-grad2}) and the vector Azuma-Hoeffding bound in Lemma~\ref{prop:azuma-hoeffding}, with probability at least $1-2/n$,
$$
\|\nabla \ell (\theta^\star) \|_2 \le B\sqrt{\sigma_{F,K}}\sqrt{2m(\log(n) +2)}.
$$
\end{proof}

The negative log-likelihood function satisfies the bound in Lemma~\ref{lem:mle-taylor}. This combined with relation (\ref{eq:semi}), Lemma~\ref{lem:khessi} and Lemma~\ref{lem:kgrad} implies that if $\lambda_2(L_{\overline{\vec{M}}_{w^\star}}) \geq 32 (\sigma_{F,K}/C)n\log(n)/m$, then with probability at least $1-3/n$, 
$$
\frac{1}{n}\|\widehat{\theta}-\theta^\star\|_2^2 \leq 32 \left(\frac{B}{A C}\right)^2 \sigma_{F,K} \frac{n(\log(n)+2)}{\lambda_2(L_{\overline{\vec{M}}_{w^\star}})^2}\frac{1}{m}
$$
which proves the theorem.

\section{Proof of Theorem~\ref{thm:break-1}} \label{sec:proof-break-1}

Under condition that $\lambda_2\left(\nabla^2 (-\ell_{k-1}(\theta) \right)>0$, the negative pseudo log-likelihood function satisfies the bound in Lemma~\ref{lem:mle-taylor}. This, combined with the following two lemmas implies the statement of the theorem.

%\subsection{Proof of Lemma~\ref{lem:break1lambda_2}}

\begin{lemma} If $\lambda_2(L_{\vec{M}}) \geq 8 k(k-1)e^{2b} n\log(n)/m$, then with probability at least $1-1/n$,
$$
\min_{\theta \in \Theta}\lambda_2(\nabla^2(-\ell_{k-1}(\theta))) \geq \frac{1}{4ke^{2b}} \frac{m}{n}\lambda_2(L_{\vec{M}}). 
$$
\label{lem:break1lambda_2}
\end{lemma}

\begin{proof} We will establish the lemma by using the matrix Chernoff bound in Lemma~\ref{cor:matrix} as follows. Note that $\nabla^2(-\ell_{k-1}(\theta))$ is a sum independent random matrices given by: 
$$
\nabla^2(-\ell_{k-1}(\theta)) = \sum_{t=1}^m \nabla^2 \left(\sum_{v\in S_t\setminus \{y_t\}}-\log(p_{y_t,v}(\theta))\right). 
$$ 
The nonzero elements of $\nabla^2 \left(\sum_{v\in S_t\setminus \{y_t\}}-\log(p_{y_t,v}(\theta))\right)$ are, for every $i,j \in S_t$ such that $i\neq j$,
\begin{equation}
\frac{\partial^2\left(\sum_{v\in S_t\setminus \{y_t\}} -\log(p_{y_t,v}(\theta))\right)}{\partial \theta_i \partial \theta_j}
= \left\{
\begin{array}{ll}
-p_{i,j}(\theta)(1-p_{i,j}(\theta)), & \hbox{ if } {\{i,j\}\subseteq S_t, y_t \in \{i,j\}}\\ 
0, & \hbox{ otherwise}
\end{array}
\right .
\label{equ:part2}
\end{equation}
and
$$
\frac{\partial^2\left(\sum_{v\in S_t\setminus \{y_t\}} -\log(p_{y_t,v}(\theta))\right)}{\partial \theta_i^2} = -\sum_{j\in S_t\setminus \{i\}} \frac{\partial^2\left(\sum_{v\in S_t\setminus \{y_t\}} -\log(p_{y_t,v}(\theta))\right)}{\partial \theta_i \partial \theta_j}.
$$

From (\ref{equ:part2}), we have
\begin{eqnarray}
\E\left[ \frac{\partial^2 \left(\sum_{v\in S_t\setminus \{y_t\}}-\log(p_{y_t,v}(\theta))\right)}{\partial \theta_i \partial \theta_j}\right] 
&=& -p_{i,j}(\theta)(1-p_{i,j}(\theta))(p_{i,S_t}(\theta) + p_{j,S_t}(\theta))\nonumber\\
&=& -\frac{e^{\theta_i}+e^{\theta_j}}{\sum_{v\in S_t}e^{\theta_v}}\frac{e^{\theta_i}e^{\theta_j}}{(e^{\theta_i}+e^{\theta_j})^2}\nonumber\\
&=& -\frac{1}{\sum_{v\in S_t}e^{\theta_v}}\frac{1}{e^{-\theta_i}+e^{-\theta_j}}\nonumber\\
&\leq & -\frac{1}{2ke^{2b}}.\nonumber 
\label{eq:k-ex} 
\end{eqnarray}

Hence, we have
$$
\E\left[ \nabla^2(-\ell_{k-1}(\theta))\right] \succeq \frac{1}{2ke^{2b}}\frac{m}{n} L_\vec{M}. 
$$
and, in particular,
\begin{equation}
\lambda_2(\E\left[ \nabla^2(-\ell_{k-1}(\theta))\right]) \geq \frac{1}{2ke^{2b}}\frac{m}{n} \lambda_2(L_\vec{M}). 
\label{eq:k-ex2}
\end{equation}

To apply the matrix Chernoff bound in Lemma~\ref{cor:matrix}, we use the following identities that follow by Lemma~\ref{lem:Leig},
\begin{equation}
\lambda_2(\nabla^2((-\ell_{k-1}(\theta)))) = \lambda_1(\vec{U}^\top\nabla^2((-\ell_{k-1}(\theta))))\vec{U})
\label{equ:lambdas1}
\end{equation}
and
\begin{equation}
\lambda_2(\nabla^2(\E[(-\ell_{k-1}(\theta))])) = \lambda_1(\vec{U}^\top\nabla^2(\E[(-\ell_{k-1}(\theta))]))\vec{U})
\label{equ:lambdas2}
\end{equation}
and the following fact:
\begin{eqnarray}
&& \left\|\vec{U}^\top\nabla^2\left(\sum_{v\in S_t\setminus \{y_t\}}-\log(p_{y_t,v}(\theta))\right) \vec{U}\right\|_2 \\
& \leq & \left\|\vec{U}^\top\right\|_2 \left\|\nabla^2\left(\sum_{v\in S_t\setminus \{y_t\}}-\log(p_{y_t,v}(\theta))\right)\right\|_2\left\|\vec{U}\right\|_2\nonumber\\
&=& \left\|\nabla^2\left(\sum_{v\in S_t\setminus \{y_t\}}-\log(p_{y_t,v}(\theta))\right)\right\|_2\nonumber\\
&\leq & \max_{i}\left| \frac{\partial^2 }{\partial \theta_i^2} \left(\sum_{v\in S_t\setminus \{y_t\}}\log(p_{y_t,v}(\theta))\right) \right|+ \sum_{j\neq i} \left| \frac{\partial^2 }{\partial \theta_i \partial \theta_j} \left(\sum_{v\in S_t\setminus \{y_t\}}\log(p_{y_t,v}(\theta))\right) \right| \nonumber\\
&=& 2 \max_{i} \sum_{j\neq i}  p_{i,j}(\theta)(1-p_{i,j}(\theta))1_{\{i,j\}\subseteq S_t, y_t \in \{i,j\}} \\
&\le&  \frac{1}{2}(k-1).
\label{eq:k-l2}
 \end{eqnarray} 
where the first equation follows by $\|\vec{U}^\top\|_2 = \|\vec{U}\|_2 = 1$, and the second inequality is by the Ger\v sgorin circles theorem (Lemma~\ref{lem:gersgorin}).
\begin{eqnarray*}
&& \Pr[\lambda_2(\nabla^2 (-\ell_{k-1}(\theta)))\leq \frac{1}{4ke^{2b}}\frac{n}{m}\lambda_2(L_{\vec{M}})] \\
&\leq & \Pr[\lambda_2(\nabla^2 (-\ell_{k-1}(\theta)))\leq \frac{1}{2}\lambda_2(\nabla^2 \E[(-\ell_{k-1}(\theta))])]\\
&=& \Pr[\lambda_1(\vec{U}^\top\nabla^2 (-\ell_{k-1}(\theta))\vec{U})\leq \frac{1}{2}\lambda_1(\vec{U}^\top \nabla^2 \E[(-\ell_{k-1}(\theta))]\vec{U})]\\
&\leq & n e^{-\frac{\lambda_1(\vec{U}^\top \nabla^2 \E[(-\ell_{k-1}(\theta))]\vec{U})}{4\frac{k-1}{2}}}\\
& \leq & n e^{-\frac{\lambda_2(L_{\vec{M}})m}{4 k(k-1)e^{2b}n}}\\
& \leq & \frac{1}{n} 
\end{eqnarray*}
where the first inequality is by (\ref{eq:k-ex2}), the equality is by (\ref{equ:lambdas1}) and (\ref{equ:lambdas2}), the second inequality is by the matrix Chernoff bound in Lemma~\ref{cor:matrix} and (\ref{eq:k-l2}), the third inequality is by (\ref{eq:k-ex2}), and the last inequality is by the condition $\lambda_2(L_{\vec{M}}) \geq 8k(k-1)e^{2b} n\log(n)/m$.
\end{proof}

%\subsection{Proof of Lemma~\ref{lem:break1nabla}}

\begin{lemma} With probability at least $1-2/n$,
\begin{equation}
\left\| \nabla(-\ell_{k-1}(\theta^\star))\right\|_2  \le  2\sqrt{k(k-1) m (\log (n)+2)}.
\label{eq:s1-grad-u}
\end{equation}
\label{lem:break1nabla}
\end{lemma}

\begin{proof} $\nabla(-\ell_{k-1}(\theta))$ is a sum of independent random vectors given by:
$$
\nabla(-\ell_{k-1}(\theta)) = \sum_{t=1}^m \nabla\left(\sum_{v\in S_t\setminus \{y_t\}} -\log(p_{y_t,v}(\theta))\right).
$$ 

It is straightforward to show that $\nabla\left(\sum_{v\in S_t\setminus \{y_t\}} -\log(p_{y_t,v}(\theta))\right)$ has the elements given by
\begin{equation}
\frac{\partial \left(\sum_{v\in S_t\setminus \{y_t\}} -\log(p_{y_t,v}(\theta))\right)}{\partial \theta_i} = \left\{
\begin{array}{ll}
-\sum_{v\in S_t\setminus \{y_t\}}p_{v,y_t}(\theta), & \hbox{ if } i=y_t\\
p_{i,y_t}(\theta), & \hbox{ if } i\in S_t\setminus \{y_t\}\\
0, & \hbox{ otherwise }
\end{array}
\right .
\label{equ:partlik1}
\end{equation}

For every $t\in \{1,2,\ldots,m\}$ and $i \in \{1,2,\ldots,n\}$, we have
$$
\E\left[\frac{\partial \left(\sum_{v\in S_t\setminus \{y_t\}}-\log(p_{y_t,v}(\theta^\star))\right)}{\partial \theta_i}\right] = 0.
$$
The last equation obviously holds for every $i\notin S_t$, and it holds for $i\in S_t$ by the following derivations 
\begin{eqnarray*}
\E\left[\frac{\partial}{\partial \theta_i}\left(\sum_{v\in S_t\setminus \{y_t\}}-\log(p_{y_t,v}(\theta^\star))\right)\right]
 & = & -p_{i,S_t}(\theta^\star)\sum_{v\in S_t \setminus\{ i\}} p_{v,i}(\theta^\star) 
+ \sum_{v\in S_t \setminus\{ i\}} p_{v,S_t}(\theta^\star)p_{i,v}(\theta^\star)\\
& = & \sum_{v\in S_t \setminus\{ i\}} (-p_{i,S_t}(\theta^\star) p_{v,i}(\theta^\star) 
+ p_{v,S_t}(\theta^\star)p_{i,v}(\theta^\star))=0.
\end{eqnarray*}

From (\ref{equ:partlik1}), we have
\begin{eqnarray}
\left\|\nabla\left(\sum_{v\in S_t\setminus \{y_t\}}-\log(p_{y_t,v}(\theta))\right)\right\|_2^2
& = & \left(\sum_{v\in S_t\setminus \{y_t\}} p_{v,y_t}(\theta)\right)^2 + \sum_{v\in S_t\setminus \{y_t\}} p_{v,y_t}(\theta)^2\nonumber\\ 
& \le &  (k-1)^2 + k-1 = k(k-1).\nonumber
\end{eqnarray}

Hence,
\begin{equation}
\left\|\nabla\left(\sum_{v\in S_t\setminus \{y_t\}}-\log(p_{y_t,v}(\theta))\right)\right\|_2 \leq \sqrt{k(k-1)}.
\label{equ:lk1sigma}
\end{equation}

Therefore, by vector Azuma-Hoeffding bound in Lemma~\ref{prop:azuma-hoeffding}, (\ref{eq:s1-grad-u}) holds with probability at least $1-2/n$.
\end{proof}

\section{Proof of Theorem~\ref{thm:break}} \label{sec:proof-break}

The proof follows by the same steps as that of Theorem~\ref{thm:break-1}, and the following two lemmas.

%\subsection{Proof of Lemma~\ref{lem:breakk1lambda}}

\begin{lemma} If $\lambda_2\left( L_{\vec{M}} \right) \ge 8k(k-1)e^{2b}n\log(n)/m$, then with probability at least $1-1/n$,
\begin{equation}
\lambda_2(\nabla^2(-\ell_1(\theta))) \ge \frac{1}{4k(k-1) e^{2b}}\frac{m}{n}\lambda_2 (L_{\vec{M}}).
\label{eq:2-chernoff}
\end{equation}
\label{lem:breakk1lambda}
\end{lemma}

\begin{proof} $\nabla^2(-\ell_1(\theta))$ is a sum of random matrices given by:
$$
\nabla^2(-\ell_1(\theta)) = \sum_{t=1}^m \nabla^2 \left(-\log(p_{y_t,z_t}(\theta))\right)
$$
where $p_{y_t,z_t}(\theta) = e^{\theta_{y_t}}/(e^{\theta_{y_t}}+e^{\theta_{z_t}})$.

It is easy to establish that for all $t\in \{1,2,\ldots,m\}$ and $i\neq j$,
\begin{equation}
\frac{\partial^2 (-\log(p_{y_t,z_t}(\theta)))}{\partial \theta_i \partial \theta_j} = \left\{
\begin{array}{ll}
-p_{i,j}(\theta)(1-p_{i,j}(\theta)), & \hbox{ if } {\{i,j\} = \{y_t,z_t\}}\\
0, & \hbox{ otherwise}
\end{array}
\right .
\label{equ:n2logp1}
\end{equation}
and
\begin{equation}
\frac{\partial^2 (-\log(p_{y_t,z_t}(\theta)))}{\partial \theta_i^2} = -\sum_{v\in S_t\setminus \{i\}} \frac{\partial^2 (-\log(p_{y_t,z_t}(\theta)))}{\partial \theta_i\partial \theta_v} .
\label{equ:n2logp2}
\end{equation}

Hence, for every $i,j\in \{1,2,\ldots,n\}$ such that $i\neq j$,
\begin{eqnarray*}
  \E\left[ \frac{\partial^2 (-\log(p_{y_t,z_t}(\theta)))}{\partial \theta_i \partial \theta_j}\right] 
	&=& -p_{i,j}(\theta)(1-p_{i,j}(\theta))\Pr[\{y_t,z_t\}=\{i,j\}]\\
	&=& -p_{i,j}(\theta)(1-p_{i,j}(\theta)) \frac{1}{k-1}(p_{i,S_t}(\theta)+p_{j,S_t}(\theta))\\
	&=& -\frac{1}{k-1}\frac{1}{\sum_{v\in S_t}e^{\theta_v}}\frac{1}{e^{-\theta_i}+e^{-\theta_j}}\\
	& \leq & -\frac{1}{2k(k-1)e^{2b}}.
\end{eqnarray*}

Hence,
$$
\nabla^2 \E[(-\ell_1(\theta))] \succeq \frac{1}{2k(k-1)e^{2b}}\frac{m}{n}L_{\vec{M}}
$$
and, in particular,
\begin{equation}
\lambda_2(\nabla^2 \E[(-\ell_1(\theta))]) \geq \frac{1}{2k(k-1)e^{2b}}\frac{m}{n}\lambda_2(L_{\vec{M}}).
\label{equ:en2lik1}
\end{equation}

By (\ref{equ:n2logp1}) and (\ref{equ:n2logp2}) and Gre\v sgorin circle theorem,
\begin{equation}
\left\|\nabla^2\left(-\log(p_{y_t,z_t}(\theta))\right) \right\|_2 \leq 2 p_{y_t,z_t}(\theta)(1-p_{y_t,z_t}(\theta)) \leq \frac{1}{2}.
\label{equ:en2lik12}
\end{equation} 

The rest of the proof follows by the same arguments as in the proof of Lemma~\ref{lem:break1lambda_2}.
\end{proof}

%\subsection{Proof of Lemma~\ref{lem:breakk1nabla}}

\begin{lemma} With probability at least $1-2/n$,
\begin{equation}
\left\| \nabla(-\ell_1(\theta^\star)) \right\|_2  
\le  2\sqrt{m (\log (n)+2)}.
\label{eq:2-azu}
\end{equation}
\label{lem:breakk1nabla}
\end{lemma}

\begin{proof} $\nabla (-\ell_1(\theta))$ is a sum of independent random vectors in $\reals^n$ with elements given by
$$
\frac{\partial (-\log(p_{y_t,z_t}(\theta)))}{\partial \theta_i} = \left\{
\begin{array}{ll}
-p_{z_t,i}(\theta), & \hbox{ if } i = y_t\\
p_{i,y_t}(\theta), & \hbox{ if } i = z_t\\
0, & \hbox{ otherwise}.
\end{array}
\right .
$$

It follows that for all $t\in \{1,2,\ldots,m\}$ and $i\in \{1,2,\ldots,n\}$,
\begin{eqnarray*}
\E\left[\frac{\partial (-\log(p_{y_t,z_t}(\theta^\star)))}{\partial \theta_i}\right] 
&=& -p_{i,S_t}(\theta^\star) \sum_{j\in S_t \setminus\{ i\}} \frac{1}{k-1}p_{j,i}(\theta^\star) + 
\sum_{j\in S_t \setminus\{ i\}} p_{j,S_t}(\theta^\star)\frac{1}{k-1}p_{i,j}(\theta^\star)\\
&=& \frac{1}{k-1}\sum_{j\in S_t\setminus\{i\}} (-p_{i,S_t}(\theta^\star)p_{j,i}(\theta^\star) + p_{j,S_t}(\theta^\star)p_{i,j}(\theta^\star)) = 0
\end{eqnarray*}
and
\begin{equation}
\left\|\nabla (-\log(p_{y_t,z_t}(\theta)))\right\|_2^2 = p_{z_t,y_t}(\theta)^2 + p_{z_t,y_t}(\theta)^2 \le  2.
\end{equation}
The statement of the lemma than follows by vector Azuma-Hoeffding bound in Lemma~\ref{prop:azuma-hoeffding}.
\end{proof}

\section{Proof of Lemma~\ref{cor:matrix}}

From Theorem 5.1.1 in \cite{tropp2015introduction},
\begin{equation}
\Pr\left[ \lambda_{1}\left(S_m\right) \le
  (1-\epsilon) \lambda_1(\E[S_m]) \right]
	\le n
  \left(\frac{e^{-\epsilon}}{(1-\epsilon)^{1-\epsilon}}
  \right)^{\frac{\lambda_1(\E[S_m])}{\sigma}}\quad\mbox{for}~\epsilon\in [0,1).\label{equ:pro2}
\end{equation}
which combined with the fact
$$
\frac{e^{-\epsilon}}{(1-\epsilon)^{1-\epsilon}} \leq e^{-\frac{\epsilon^2}{2}}, \hbox{ for all } \epsilon \in (0,1]
$$
yields the lemma.

\subsection{Proof of Lemma~\ref{lem:Leig}}

Since $L_{\vec{A}} = \hbox{diag}(\vec{A}\vec{1}) - \vec{A}$, we have
$$
L_{\vec{A}} \vec{1} = \hbox{diag}(\vec{A}\vec{1})\vec{1} - \vec{A}\vec{1} = \vec{A}\vec{1} - \vec{A}\vec{1} = \vec{0}.
$$
Hence, $0$ is an eigenvalue of $L_{\vec{A}}$ for eigenvector $\vec{1}$.

Let $\vec{y}\in \reals^{n}$ be an eigenvector with corresponding eigenvalue $\lambda$ of $L_{\vec{A}}$. Since the columns of $\vec{U}$ are independent, $\vec{U}$ is nonsingular and it has inverse $\vec{U}^{-1}$ (e.g., Section~0.5~\cite{horn}. Let $\vec{x} = \vec{U}^{-1}\vec{y}$). 

Note that the following equations hold:
\begin{eqnarray*}
\vec{x}^T \vec{U}^T L_{\vec{A}} \vec{U}\vec{x} &=& (\vec{U}\vec{x})^\top L_{\vec{A}} (\vec{U}\vec{x})\\
&=& \vec{y}^\top L_{\vec{A}} \vec{y}\\
&=&  \lambda \vec{y}^\top\vec{y}\\
&=&  \lambda \vec{x}^\top \vec{U}^\top \vec{U}\vec{x}\\
&=& \lambda \vec{x}^\top \vec{x}.
\end{eqnarray*}

Hence, it follows that $\lambda$ is an eigenvalue of $\vec{U}^\top L_{\vec{A}} \vec{U}$ with corresponding eigenvector $\vec{x}$. 

\section{Proof of Lemma~\ref{lem:laplaceposdef}}

Let $\vec{C} = \vec{B}-\vec{A} \geq \vec{0}$. Note that 
\begin{eqnarray*}
L_{\vec{B}} - L_{\vec{A}} &=& (\hbox{diag}(\vec{B}\vec{1}) - \vec{B}) - (\hbox{diag}(\vec{A}\vec{1}) - \vec{A})\\
&=& \hbox{diag}((\vec{B}-\vec{A})\vec{1}) - (\vec{B}-\vec{A})\\
&=& L_{\vec{C}}. 
\end{eqnarray*}
Since $L_{\vec{C}}$ is positive semidefinite, it follows that $L_{\vec{B}} - L_{\vec{A}}$ is positive semidefinite, i.e. $L_{\vec{B}}\succeq L_{\vec{A}}$.

\begin{lemma} For all $S\subseteq \{1,2,\ldots,n\}$ such that $|S| = k\geq 2$ we have: for all $i,j\in \{1,2,\ldots,n\}$ such that $i\neq j$, for $v \in \{1,2,\ldots,n\} \setminus \{i,j\}$, 
\begin{equation}
\frac{\partial^2 (-\log(p_k(\vec{x}_v(\vec{0}))))}{\partial \theta_i \partial \theta_j} 
= - k\frac{\partial^2 p_k(\vec{0})}{\partial x_1 \partial x_2}
 + k^2 \left(\frac{\partial p_k(\vec{0})}{\partial x_1}\right)^2
\label{equ:case1}
\end{equation}
and
\begin{equation}
\frac{\partial^2 (-\log(p_k(\vec{x}_j(\vec{0}))))}{\partial \theta_i \partial \theta_j} = \frac{k(k-2)}{2}\frac{\partial^2 p_k(\vec{0})}{\partial x_1 \partial x_2} - k^2(k-1)\left(\frac{\partial p_k(\vec{0})}{\partial x_1}\right)^2.
\label{equ:case2}
\end{equation}
Moreover,
\begin{equation}
\frac{\partial^2 (-\log(p_k(\vec{x}_v(\vec{0}))))}{\partial \theta_i^2} = - \sum_{j\in S\setminus \{i\}} \frac{\partial^2 (-\log(p_k(\vec{x}_v(\vec{0}))))}{\partial \theta_i \partial \theta_j}.
\label{equ:diff2i}
\end{equation}
\label{lem:partiallogs}
\end{lemma}

\begin{proof} Let $S\subseteq N$ be such that $|S| = k$, for an integer $2\leq k \leq n$. Without loss of generality, let $S = \{1,2,\ldots,k\}$. Let $\vec{x}_v(\theta) = (\theta_v - \theta_u, u\in S\setminus \{v\})$, for $v\in S$. We first consider $\frac{\partial^2}{\partial \theta_i \partial \theta_j}(-\log(p_k(\vec{x}_v(\theta))))$ for $i\neq j$. It is easy to note that
\begin{eqnarray*}
\frac{\partial^2 (-\log(p_k(\vec{x}_v(\theta))))}{\partial \theta_i \partial \theta_j} 
&=& -\frac{1}{p_k(\vec{x}_v(\theta))}\frac{\partial^2 p_k(\vec{x}_v(\theta))}{\partial\theta_i \partial \theta_j}\nonumber\\
&&  + \frac{1}{p_k(\vec{x}_v(\theta))^2} \frac{\partial p_k(\vec{x}_v(\theta))}{\partial \theta_i}  \frac{\partial p_k(\vec{x}_v(\theta))}{\partial \theta_j}. 
\end{eqnarray*}

We separately consider two different cases. 

Consider first the case when $\{i\}\cap \{j\}\cap \{v\} = \emptyset$. By differentiation, we have 
$$
\frac{\partial^2 (-\log(p_k(\vec{x}_v(\vec{0}))))}{\partial \theta_i \partial \theta_j} 
= - k\frac{\partial^2 p_k(\vec{0})}{\partial x_1 \partial x_2}
 + k^2 \left(\frac{\partial p_k(\vec{0})}{\partial x_1}\right)^2
$$
which establishes (\ref{equ:case1}).

Consider now the case when $i \neq v$ and $j = v$. First, note
\begin{eqnarray}
\frac{\partial^2}{\partial \theta_i \partial \theta_j}(-\log(p_k(\vec{x}_j(\theta))))
&=& -\frac{1}{p_k(\vec{x}_j(\theta))}\frac{\partial^2 p_k(\vec{x}_j(\theta))}{\partial\theta_i \partial \theta_j}\nonumber\\
&&  + \frac{1}{p_k(\vec{x}_j(\theta))^2} \frac{\partial p_k(\vec{x}_j(\theta))}{\partial \theta_i}  \frac{\partial p_k(\vec{x}_j(\theta))}{\partial \theta_j}.
\label{equ:e0}
\end{eqnarray}

For every $u\in S$, $p_k(\vec{x}_u(\theta))$ does not change its value by changing $\theta$ with $\theta + c \vec{1}$, for every constant $c\in \reals$. Hence, by full differentiation, we have
\begin{equation}
\frac{\partial p_k(\vec{x}_j(\theta))}{\partial \theta_j} = -\sum_{v\in S \setminus \{j\}} \frac{\partial p_k(\vec{x}_u(\theta))}{\partial \theta_v}.
\label{equ:fulldiff}
\end{equation}

From (\ref{equ:fulldiff}),
\begin{equation}
\frac{\partial^2 p_k(\vec{x}_j(\theta))}{\partial\theta_i \partial \theta_j} = -\frac{\partial^2 p_k(\vec{x}_j(\theta))}{\partial \theta_i^2} - \sum_{v\in S \setminus \{i,j\}} \frac{\partial^2 p_k(\vec{x}_j(\theta))}{\partial \theta_i \partial \theta_v}.
\label{equ:inter}
\end{equation}

Now, note that
$$
\frac{\partial^2 p_k(\vec{x}_j)}{\partial \theta_i^2} = \int_\reals f(z) f'(x_i+z)\prod_{v\in S\setminus \{i,j\}}F(x_v+z)dz.
$$
Hence, 
\begin{eqnarray*}
\frac{\partial^2 p_k(\vec{x}_j(\vec{0}))}{\partial \theta_i^2}
&=&  \int_\reals f(z) f'(z) F(z)^{k-2}dz\\
&=& f(z)^2F(z)^{k-2}|_{-\infty}^\infty - \int_\reals f(z) (f(z)F(z)^{k-2})'dz\\
&=& - \int_\reals f(z)f'(z)F(z)^{k-2}dz - (k-2) \int_\reals f(z)^2F(z)^{k-3}dz\\
&=& - \frac{\partial^2 p_k(\vec{x}_j(\vec{0}))}{\partial \theta_i^2} - (k-2) \frac{\partial^2 p_k(\vec{0})}{\partial x_1 \partial x_2}.
\end{eqnarray*}

From this it follows
\begin{equation}
\frac{\partial^2 p_k(\vec{x}_j(\vec{0}))}{\partial \theta_i^2} = -\frac{k-2}{2}\frac{\partial^2 p_k(\vec{0})}{\partial x_1 \partial x_2}.
\label{equ:e2}
\end{equation}

From (\ref{equ:inter}) and (\ref{equ:e2}), 
\begin{equation}
\frac{\partial^2 p_k(\vec{x}_j(\vec{0}))}{\partial\theta_i \partial \theta_j} = -\frac{k-2}{2}\frac{\partial^2 p_k(\vec{0})}{\partial x_1\partial x_2}.
\label{equ:e3}
\end{equation}

From (\ref{equ:fulldiff}),
\begin{equation}
\frac{\partial p_k(\vec{x}_j(\vec{0}))}{\partial \theta_j} = (k-1)\frac{\partial p_k(\vec{0})}{\partial x_1}.
\label{equ:e1}
\end{equation}

Combining (\ref{equ:e0}), (\ref{equ:e3}) and (\ref{equ:e1}), we obtain (\ref{equ:case2}).

Since for every $v\in S$,  $\log(p_k(x_v(\theta)))$ does not change its value by changing $\theta$ to $\theta + c\vec{1}$ for all $c\in \reals$, by full differentiation
$$
\frac{\partial (-\log(p_k(x_v(\theta))))}{\partial \theta_i} = -\sum_{j\in S\setminus \{i\}}\frac{\partial (-\log(p_k(x_v(\theta)))}{\partial \theta_j}.
$$
Taking partial derivative with respect to $\theta_i$ on both sides implies (\ref{equ:diff2i}).
\end{proof}

\begin{lemma} Let $S\subseteq N$ be such that $|S| = k\geq 2$ and $Y$ be a random variable according to distribution $p_{y,S}(\theta)$, for $\theta \in \reals^n$. Then, 
$$
\E\left[\nabla^2(-\log(p_{Y,S}(\vec{0})))\right] =  \left(k\frac{\partial p_k(\vec{0})}{\partial x_1}\right)^2 L_{\vec{M}_S}
$$
where $\vec{M}_S = [m_{i,j}^S]\in \reals^{n\times n}$ is such that $m_{i,j}^S = 1$ if $i,j\in S$ and $i\neq j$, and $m_{i,j}^S = 0 $, otherwise.
\label{lem:explog}
\end{lemma}

\begin{proof} By (\ref{equ:case1}) and (\ref{equ:case2}) in Lemma~\ref{lem:partiallogs}, we have for all $i,j\in \{1,2,\ldots,n\}$ such that $i\neq j$,
\begin{eqnarray}
\E\left[\frac{\partial^2 (-\log(p_{Y,S}(\vec{0})))}{\partial\theta_i \partial \theta_j}\right] 
&=&  \sum_{v\in S\setminus \{i,j\}} p_k(\vec{x}_v(\vec{0})) \frac{\partial^2 (-\log(p_k(\vec{x}_v(\vec{0}))))}{\partial \theta_i \partial \theta_j}\nonumber\\
&& + p_k(\vec{x}_i(\vec{0})) \frac{\partial^2 (-\log(p_k(\vec{x}_i(\vec{0}))))}{\partial \theta_i \partial \theta_j} \nonumber\\
&& + p_k(\vec{x}_j(\vec{0})) \frac{\partial^2 (-\log(p_k(\vec{x}_j(\vec{0}))))}{\partial \theta_i \partial \theta_j} \nonumber\\
&=& \frac{k-2}{k}\frac{\partial^2 (-\log(p_k(\vec{x}_v(\vec{0}))))}{\partial \theta_i \partial \theta_j} \nonumber \\
&& + \frac{2}{k}\frac{\partial^2 (-\log(p_k(\vec{x}_i(\vec{0}))))}{\partial \theta_i \partial \theta_j}, \hbox{ for any } v \in S\setminus\{i,j\}\nonumber\\
&=& -k^2 \left(\frac{\partial p_k(\vec{0})}{\partial x_1}\right)^2.\label{equ:gijdiff}
\end{eqnarray}

By (\ref{equ:diff2i}) in Lemma~\ref{lem:partiallogs}, we have
\begin{eqnarray*}
\E\left[\frac{\partial^2 (-\log(p_{Y,S}(\vec{0})))}{\partial\theta_i^2}\right] &=& -\sum_{u\in S\setminus \{i\}} \sum_{v\in S} p_k(\vec{x}_v(\vec{0}))\frac{\partial^2 (-\log(p_k(\vec{x}_v(\vec{0}))))}{\partial \theta_i\partial \theta_u}\\ 
&=& k^2(k-1) \left(\frac{\partial p_k(\vec{0})}{\partial x_1}\right)^2.
\end{eqnarray*}
\end{proof}

\begin{lemma} If for $S\subseteq N$ such that $|S| =k \geq 2$,
$$
\frac{\partial^2 (-\log(p_{v,S}(\vec{0})))}{\partial \theta_i\partial \theta_j} \leq 0 \hbox{ for all } i,j,v\in \{1,2,\ldots,n\} \hbox{ such that } i\neq j
$$
then
$$
\left\| \nabla^2(-\log(p_{y,S} ({\bm 0}))) \right\|_2 \le \frac{1}{\gamma_{F,k}}
$$
where $1/\gamma_{F,k} = \left(k^2 \partial p_k(\vec{0})/\partial x_1\right)^2$.
\label{lem:hessian-p2}
\end{lemma}

\begin{proof} Without loss of generality, let $y=1$ and $S = \{1,2,\dots,k \}$. By Lemma~\ref{lem:partiallogs}, we have
\begin{equation}
\frac{\partial^2(-\log(p_{1,S} ({\bm 0})))}{\partial \theta_1 \partial \theta_2} 
= - k^2(k-1)\left( \frac{\partial p_k({\bm 0})}{\partial x_1}\right)^2
 + \frac{k(k-2)}{2}\frac{\partial^2 p_k({\bm 0})}{\partial x_1 \partial x_2} 
\label{eq:hessi-1}
\end{equation}
and for $i\neq 1$, $j\neq 1$ and $j\neq i$,
\begin{equation}
\frac{\partial^2 (-\log(p_{1,S} ({\bm 0})))}{\partial \theta_i \partial \theta_j} = -k \frac{\partial^2 p_k({\bm 0})}{\partial x_1 \partial x_2}
+ k^2\left(\frac{\partial p_k({\bm 0})}{\partial x_1}\right)^2. 
\label{eq:hessi-2}
\end{equation}

Since by condition of the lemma the left-hand side in (\ref{eq:hessi-1}) is non positive, we have 
\begin{equation}
(k-2)\frac{\partial^2 p_k({\bm 0})}{\partial x_1 \partial x_2} \le 2k(k-1) \left( \frac{\partial p_k({\bm 0})}{\partial x_1}\right)^2
\label{eq:hessi-3}
\end{equation}
and, since by condition of the lemma the left-hand side in (\ref{eq:hessi-2}) is non positive, we have 
\begin{equation}
\frac{\partial^2 p_k({\bm 0})}{\partial x_1 \partial x_2}\ge 0.
\label{eq:hessi-4}
\end{equation} 

From (\ref{eq:hessi-1}) and (\ref{eq:hessi-4}),
$$
\frac{\partial^2(-\log(p_{1,S} ({\bm 0})))}{\partial \theta_1 \partial \theta_2} 
\geq - k^2(k-1)\left( \frac{\partial p_k({\bm 0})}{\partial x_1}\right)^2.
$$
Consider now the case when $i\neq 1$, $j\neq 1$ and $i\neq j$. If $k = 2$, then obviously 
$$
\frac{\partial^2 (-\log(p_{1,S} ({\bm 0})))}{\partial \theta_i \partial \theta_j} = 0.
$$
Otherwise, if $k > 2$, from (\ref{eq:hessi-2}) and (\ref{eq:hessi-3}), we have
\begin{eqnarray*}
\frac{\partial^2 (-\log(p_{1,S} ({\bm 0})))}{\partial \theta_i \partial \theta_j} & \geq & k^2 \left(\frac{\partial p_k(\vec{0})}{\partial x_1}\right)^2 - \frac{2k^2(k-1)}{k-2} \left(\frac{\partial p_k(\vec{0})}{\partial x_1}\right)^2\\
& = & -k^2\left(-1 + \frac{2(k-1)}{k-2}\right)\left(\frac{\partial p_k(\vec{0})}{\partial x_1}\right)^2\\
& = & -k^2 \frac{k}{k-2}\left(\frac{\partial p_k(\vec{0})}{\partial x_1}\right)^2\\
& \geq & -k^3 \left(\frac{\partial p_k(\vec{0})}{\partial x_1}\right)^2.
\end{eqnarray*}

Hence, it follows  
$$
k^3 \left( \frac{\partial p_k({\bm 0})}{\partial x_1}\right)^2 L_{\vec{M}_S} \succeq \nabla^2 (-\log(p_{y,S} ({\bm 0}))).
$$
Therefore, we conclude
\begin{eqnarray*}
\|\nabla^2 (-\log(p_{y,S} ({\bm 0}))) \|_2 &\leq & k^3 \left( \frac{\partial p_k({\bm 0})}{\partial x_1}\right)^2 \|L_{\vec{M}_S}\|_2\\
& = & k^4 \left(\frac{\partial p_k(\vec{0})}{\partial x_1}\right)^2.
\end{eqnarray*}
\end{proof}

\section{Remark for Theorem~\ref{thm:karyub}}

For the special case of noise according to the double-exponential distribution with parameter $\beta$, we have
$$
p_k(\vec{x})=\frac{1}{1+\sum_{i=1}^{k-1}e^{-x_i/\beta}}.
$$

For every $\theta\in \theta_n$ and every $S\subseteq N$ of cardinality $k$ and $i,j,y\in S$, we can easily check that
$$
\frac{\partial^2}{\partial \theta_i \partial \theta_j} (-\log(p_{y,S}(\theta))) = -\frac{1}{\beta^2} p_{i,S}(\theta)p_{j,S}(\theta).
$$
Furthermore, the following two relations hold
$$
\frac{k}{\beta(k-1)}\left(1-p_{y,S}(\theta)\right)^2 \le \left\|\nabla p_{y,S}(\theta)\right\|_2 \le \frac{2}{\beta}\left(1-p_{y,S}(\theta) \right)^2.
$$

Since 
$$
\min_{y\in S,\theta\in [-b,b]^n} p_{y,S}(\theta) 
= \frac{1}{1+(k-1)e^{2b/\beta}}
\ge p_{y,S}({\bm 0}) e^{-2b/\beta}
$$
and
$$
\max_{y\in S,\theta\in [-b,b]^n} p_{y,S}(\theta) 
= \frac{1}{1+(k-1)e^{-2b/\beta}}
\le p_{y,S}({\bm 0}) e^{2b/\beta}
$$
we have that 
$$
\sigma_{F,K} \le \frac{1}{\beta^2}
$$
and
\begin{align}
&e^{-4b/\beta}\le A \le \widetilde{A} \le e^{4b/\beta}, \\
&e^{-4b/\beta}\le \widetilde{B} \le B \le 4,\\
&e^{-2b/\beta}\le C \le \widetilde{C} \le e^{2b/\beta}.
\end{align}

\section{Proof of Theorem~\ref{thm:clustering-example}}
\label{sec:clustering-example}

Let $p^e$ denote the probability that the point score ranking method incorrectly classifies at least one item: 
$$
p^e = \Pr\left[\bigcup_{v\in N_1} \{v\in \widehat N_1\} \cup \bigcup_{v\in N_2} \{v\in \widehat N_2\}\right]
$$

Let $R_i$ denote the point score of item $i\in N$. If the point scores are such that $R_v > m/n$ for every $v\in N_1$ and $R_v < m/n$ for every $v\in N_2$, then this implies a correct classification. Hence, it must be that in the event of a misclassification of an item, $R_v \leq m/n$ for some $v\in N_1$ or $R_v \geq m/n$ for some $v\in N_2$. Combining this with the union bound, we have
\begin{equation}
p^e 
%\leq \Pr\left[\bigcup_{v\in N_1} \left\{R_v \leq \frac{m}{n}\right\} \cup \bigcup_{v\in N_2} \left\{R_v \geq \frac{m}{n}\right\}\right]
\leq \sum_{v\in N_v}\Pr\left[R_v \leq \frac{m}{n}\right] + \sum_{l\in N_2}\Pr\left[R_v \geq \frac{m}{n}\right].
\label{equ:tobound}
\end{equation}

Let $i$ and $j$ be arbitrarily fixed items such that $i\in N_1$ and $j\in N_2$. We will show that for $t\in \{1,2,\ldots,m\}$,
\begin{equation}
\Pr[y_t = i] \geq \frac{1}{n} + \frac{bk^2}{4n}\frac{\partial p_k(\vec{0})}{\partial x_1}
\label{equ:yti}
\end{equation}
and
\begin{equation}
\Pr[y_t = j] \leq \frac{1}{n} - \frac{bk^2}{4n}\frac{\partial p_k(\vec{0})}{\partial x_1}.
\label{equ:ytj}
\end{equation}

By the Chernoff bound (\ref{equ:che1}) we have
\begin{eqnarray*}
\Pr\left[R_i \leq \frac{m}{n}\right] & \leq & \exp\left(-\frac{1}{4}n\left(\frac{1}{n} - \E[y_1 = i]\right)^2 m\right)\\
& \leq & \exp\left(-\frac{1}{4}\left(\frac{bk^2}{4n}\frac{\partial p_k(\vec{0})}{\partial x_1}\right)^2 m\right)\\
& \leq & \exp\left(-\log(n/\delta)\right)\\
&=& \frac{\delta}{n}.
\end{eqnarray*}

Similarly, by the Chernoff bound (\ref{equ:che2}), we have
$$
\Pr\left[R_j \geq \frac{m}{n}\right] \leq \frac{\delta}{n}.
$$

Combining with (\ref{equ:tobound}), we have $p^e \leq \delta$.

In the remainder of the proof we show that (\ref{equ:yti}) and (\ref{equ:ytj}) hold.

Let $\calA$ contain all $A \subseteq N$ such that $|A| = k-1$ and $A\cap\{i,j\}=\emptyset$ and $\calB$ contain all $B\subseteq N$ such that $|B| = k-2$ and $B\cap\{i,j\}=\emptyset$. Then, we have
\begin{equation}
\Pr[ y_t = i ] - \Pr[ y_t = j ] = \sum_{A\in \mathcal{A}} \Pr[S_t = A\cup\{i\}] D_{i,j}(A) +\sum_{B\in \mathcal{B}} \Pr[S_t = B \cup \{i,j\}] D_{i,j}(B) \label{eq:prdifer}
\end{equation}
where
$$
D_{i,j}(A) = \Pr[ y_t = i | S_t= A\cup\{i \} ] - \Pr[ y_t = j| S_t= A\cup\{j \}]
$$
and
$$
D_{i,j}(B) = \Pr[ y_t = i | S_t= B\cup\{i,j \}] -\Pr[ y_t = j | S_t= B\cup\{i,j \}].
$$

Let $\vec{b}$ be a $k-1$-dimensional vector with all elements equal to $b$. Then, note that
$$
D_{i,j}(A) = p_k (\vec{b}-\theta_{A})- p_k (-\vec{b}-\theta_{A}).
$$

By limited Taylor series development, we have
\begin{eqnarray}
p_k (\vec{x}) &\ge& p_k (\vec{0}) + \nabla p_k (\vec{0})^\top \vec{x} - \frac{1}{2} \beta \|\vec{x}\|_2^2 \label{equ:ptaylor1}\\
p_k (\vec{x}) &\le& p_k (\vec{0}) + \nabla p_k (\vec{0})^\top \vec{x} + \frac{1}{2} \beta \|\vec{x}\|_2^2
\label{equ:ptaylor2}
\end{eqnarray}
where 
\begin{equation}
\beta = \max_{\vec{x}\in [-2b,2b]^{k-1}} \|\nabla^2 p_k (\vec{x}) \|_2.
\label{equ:betaparam}
\end{equation}

Hence, it follows that for every $\theta_{A}\in \{-b,b\}^{k-1}$, 
\begin{equation}
D_{i,j}(A) \geq 2(k-1)b \frac{\partial p_k (\vec{0})}{\partial x_1} - 
 4(k-1)b^2\beta.
\label{eq:cb1}
\end{equation}
Under the condition of the theorem, we have
$$
\beta \le \frac{1}{4b}\frac{\partial p_k (\vec{0})}{\partial x_1}.
$$
Hence, combining with \eqref{eq:cb1}, for every $\theta_{A}\in \{-b,b\}^{k-1}$,
\begin{equation}
D_{i,j}(A) \geq (k-1)b \frac{\partial p_k (\vec{0})}{\partial x_1} \geq  \frac{kb}{2} \frac{\partial p_k (\vec{0})}{\partial x_1}. 
\label{eq:thea}
\end{equation}

By the same arguments, we can show that 
\begin{equation}
D_{i,j}(B) = p_k ({\bm b}-\theta^{(-b)}_{B})- p_k (-{\bm b}-\theta^{(b)}_{B}) \geq \frac{kb}{2} \frac{\partial p_k (\vec{0})}{\partial x_1}
\label{eq:theb}
\end{equation}
where $\theta^{(b)}_{B}\in \{-b,b \}^{k-1}$ and $\theta^{(-b)}_{B}\in \{-b,b \}^{k-1}$ are $(k-1)$-dimensional vectors with the first two elements equal to $b$ and $-b$, respectively, and other elements equal to the parameters of items $B$. 

Since comparison sets are sampled uniformly at random without replacement, 
\begin{equation}
\Pr[S_t = A\cup \{i\}] = \frac{\binom{n-1}{k-1}}{\binom{n}{k}}, \hbox{ for all } A \in {\mathcal A}
\label{equ:setA}
\end{equation}
and
\begin{equation}
\Pr[S_t = B\cup \{i,j\}] = \frac{\binom{n-2}{k-2}}{\binom{n}{k}}, \hbox{ for all } B \in {\mathcal B}.
\label{equ:setB}
\end{equation}

From \eqref{eq:prdifer}, \eqref{eq:thea}, \eqref{eq:theb}, \eqref{equ:setA} and \eqref{equ:setB}, we have
$$
\Pr[ y_t = i ] - \Pr[ y_t = j ] \ge  \frac{k^2b}{2n} \frac{\partial p_k (\vec{0})}{\partial x_1}.
$$

Using this inequality together with the following facts (i) $\Pr[y_t = v] = \Pr[y_t = i]$ for every $v\in N_1$, (ii) $\Pr[y_t = v] = \Pr[y_t = j]$ for every $v\in N_2$, (iii) $\sum_{v\in N} \Pr[y_t = v] = 1$, and (iv) $|N_1| = |N_2| = n/2$, it can be readily shown that
$$
\Pr[ y_t = i ] \ge \frac{1}{n} + \frac{k^2b}{4n} \frac{\partial p_k ({\bm 0})}{\partial x_1},
$$
which establishes (\ref{equ:yti}). By the same arguments one can establish (\ref{equ:ytj}).

\section{Proof of Theorem~\ref{thm:clustering-low}}
\label{sec:clustering-low}

Suppose that $n$ is a positive even integer and $\theta$ is the parameter vector such that $\theta_i =b$ for $i\in N_1$ and $\theta_i = -b$ for $i \in N_2$, where $N_1 =\{1,2,\ldots,n/2\}$ and $N_2 = \{n/2+1,\ldots,n\}$. Let $\theta'$ be the parameter vector that is identical to $\theta$ except for swapping the first and the last item, i.e. $\theta'_i = b$ for $i\in N_1'$ and $\theta'_i = -b$ for $i\in N_2'$, where $N_1' = \{n,2,\ldots,n/2\}$ and $N_2'=\{n/2+1,\ldots,n-1,1\}$. 

We denote with $\Pr_\theta[A]$ and $\Pr_{\theta'}[A]$ the probabilities of an event $A$ under hypothesis that the generalized Thurstone model is according to parameter $\theta$ and $\theta'$, respectively. We denote with $\E_\theta$ and $\E_{\theta'}$ the expectations under the two respective distributions. 

Given observed data $(\vec{S},\vec{y}) = (S_1,y_1), \ldots, (S_m,y_m)$, we denote the log-likelihood ratio statistic $L(\vec{S},\vec{y})$ as follows
\begin{equation}
L(\vec{S},\vec{y}) = \sum_{t=1}^{m}\log\left(\frac{p_{y_t,S_t}(\theta')\rho_t(S_t)}{p_{y_t,S_t}(\theta)\rho_t(S_t)}\right),
\end{equation}
where $\rho_t (S)$ is the probability that $S$ is drawn at time $t$.

The proof follows the following two steps:

\paragraph{Step 1:} We show that for given $\delta \in [0,1]$, for the existence of an algorithm that correctly classifies all the items with probability at least $1-\delta$, it is necessary that the following condition holds
\begin{equation}
\Pr_{\theta'}[L(\vec{S},\vec{y}) \ge \log(n/\delta)] \ge \frac{1}{2}.
\label{equ:plrt}
\end{equation}

\paragraph{Step 2:} We show that 
\begin{eqnarray}
\E_{\theta'}[L(\vec{S},\vec{y})] &\le& 36\frac{m}{n}\left(k^2b\frac{\partial p_k ({\bm 0})}{\partial x_1}\right)^2 \label{equ:el}\\
\sigma^2_{\theta'}[L(\vec{S},\vec{y})] &\le& 144\frac{m}{n}\left(k^2b\frac{\partial p_k ({\bm 0})}{\partial x_1}\right)^2\label{equ:varl}
\end{eqnarray}
where $\sigma^2_{\theta'}[L(\vec{S},\vec{y})]$ denotes the variance of random variable $L(\vec{S},\vec{y})$ under a generalized Thurstone model with parameter $\theta'$.

By Chebyshev's inequality, for every $g\in \reals$,
$$
\Pr_{\theta'}[|L(\vec{S},\vec{y})-\E_{\theta'}[L(\vec{S},\vec{y})]| \geq |g|] \leq \frac{\sigma^2_{\theta'}[L(\vec{S},\vec{y})]}{g^2}.
$$

Using this for  $g = \log(n/\delta)-\E_{\theta'}[L(\vec{S},\vec{y})]$, it follows that (\ref{equ:plrt}) implies the following condition:
\begin{eqnarray*}
\log (n/\delta) - \E_{\theta'}[L(\vec{S},\vec{y})]  & \le &  |\log (n/\delta) - \E_{\theta'}[L(\vec{S},\vec{y})]|\\
& \leq & \sqrt{2} \sigma_{\theta'}[L(\vec{S},\vec{y})].
\end{eqnarray*}

Further combining with (\ref{equ:el}) and (\ref{equ:varl}), we obtain
$$
m \geq \frac{1}{62}\frac{1}{b^2 k^4 (\partial p_k(\vec{0})/\partial x_1)^2} n(\log(n) + \log(1/\delta))
$$
which is the condition asserted in the theorem.

\paragraph{Proof of Step 1.} Let us define the following two events 
$$
A = \{|N_1\setminus \widehat N_1| = 1\} \cap \{|N_2 \setminus \widehat N_2| = 1\}
$$
and
$$
B = \{\widehat N_1 = N_1'\} \cap \{\widehat N_2 = N_2'\}.
$$
Let $B^c$ denote the complement of event $B$.

Note that
$$
\Pr_{\theta}[B] = \Pr_{\theta}[B|A] \Pr_\theta[A]
= \left(\frac{2}{n} \right)^2 \Pr_{\theta}[A]
\le \frac{4}{n^2} \delta
$$
where the second equation holds because $B \subseteq A$ and every possible partition in $A$ has the same probability under $\theta$.

For every $g\in \reals$, we have
$$
\Pr_{\theta'}[L(\vec{S},\vec{y}) \le g] =  \Pr_{\theta'}[L(\vec{S},\vec{y}) \le g,B] + \Pr_{\theta'}[L(\vec{S},\vec{y}) \le g, B^c].
$$

Now, note
\begin{eqnarray*}
\Pr_{\theta'} [L(\vec{S},\vec{y}) \le  g , B] &= & \E_{\theta'}[\Ind(L(\vec{S},\vec{y})\leq g,B)]\\
&= & \E_\theta[e^{L(\vec{S},\vec{y})}\Ind(L(\vec{S},\vec{y})\leq g,B)]\\
&\le & \E_\theta[e^{g}\Ind(L(\vec{S},\vec{y})\leq g,B)]\\
&=& e^g \Pr_{\theta}[L(\vec{S},\vec{y})\leq g, B]\\
&\le & e^g \Pr_{\theta}[B]\\
&\le & e^g\frac{4}{n^2}\delta
\label{eq:bdl-1} 
\end{eqnarray*}
where in the second equation we use the standard change of measure argument.

Since the algorithm correctly classifies all the items with probability at least $1-\delta$, we have 
\begin{equation}
\Pr_{\theta'}[L(\vec{S},\vec{y}) \le g,B^c] \le \Pr_{\theta'}[B^c] \le \delta.
\label{eq:bdl-2}
\end{equation}

For $g = \log(n/\delta)$, from \eqref{eq:bdl-1} and \eqref{eq:bdl-2}, it follows that
$$
\Pr_{\theta'}[L(\vec{S},\vec{y}) \le \log(n/\delta)] \le \delta + \frac{4}{n} \leq \frac{1}{2}
$$
where the last inequality is by the conditions of the theorem.

\paragraph{Proof of Step 2.} If the observed comparison sets $S_1,S_2,\ldots,S_m$ are such that $S_t \cap \{1,n \} = \emptyset$, for every observation $t$, then we obviously have
$$
\log\left(\frac{p_{y_t,S_t}(\theta')}{p_{y_t,S_t}(\theta)}\right) = 0, \hbox{ for all } t.
$$
We therefore consider the case when $S_t \cap \{1,n \} \neq \emptyset$. 

Using (\ref{equ:ptaylor1}), (\ref{equ:ptaylor2}), and (\ref{equ:betaparam}), we have for every $S$ and $i\in S$,
\begin{equation}
|p_{i,S}(\theta')-p_{i,S}(\theta)| 
\le 2kb \frac{\partial p_k (\vec{0})}{\partial x_1} + 4 \beta bk
\le 3kb \frac{\partial p_k (\vec{0})}{\partial x_1} \label{eq:pdifer}
\end{equation}
where the last inequality is obtained from the condition of this theorem. 

From \eqref{eq:pdifer}, for every comparison set $S$ such that $S \cap \{1,n\} \neq \emptyset$, we have
\begin{align}
\sum_{i\in S}\left(p_{i,S}(\theta')-p_{i,S}(\theta)\right)^2 
\le& \sum_{i\in \{1,n \}\cap S}\left(p_{i,S}(\theta')-p_{i,S}(\theta)\right)^2 + 
\left(\sum_{i\in S\setminus \{1,n \} }p_{i,S}(\theta')-p_{i,S}(\theta)\right)^2 \cr
\le& 2\left(3kb \frac{\partial p_k (\vec{0})}{\partial x_1} \right)^2,
\label{eq:dsq}
\end{align}
which is because for every comparison set $S$ such that $1 \in S$,
$$
p_{1,S}(\theta') \le \frac{1}{k} \le p_{1,S}(\theta) \hbox{ and } 
p_{i,S}(\theta') \ge p_{i,S}(\theta)\quad \forall i \neq 1;
$$
and, for every comparison set $S$ such that $n\in S$,
$$
p_{n,S}(\theta') \ge \frac{1}{k} \ge p_{n,S}(\theta) \hbox{ and }
p_{i,S}(\theta') \le p_{i,S}(\theta)\quad\forall i \neq n.
$$

From \eqref{eq:pdifer} and the assumption of the theorem, we have
\begin{equation}
\min_{S}\min_{i\in S} p_{i,S}(\theta) = \min_{S: n\in S} p_{n,S}(\theta)
\ge \frac{1}{k} - 3kb \frac{\partial p_k (\vec{0})}{\partial x_1}
\ge \frac{1}{2k}.
\label{eq:pmin}
\end{equation}

For simplicity of notation, let
\begin{equation}
D = 3kb \frac{\partial p_k (\vec{0})}{\partial x_1}.
\end{equation}

Then, for all $S$ such that $S\cap \{1,n\} \neq \emptyset$, we have
\begin{equation}
\sum_{i\in S} p_{i,S}(\theta') \log\left(\frac{p_{i,S}(\theta')}{p_{i,S}(\theta)}\right) \le 2 k D^2
\label{eq:meanL}
\end{equation}
which is obtained from 
\begin{itemize}
\item[(i)] $p_{i,S}(\theta) \ge 1/(2k)$ for all $i\in S$ that holds by \eqref{eq:pmin},
\item[(ii)] $\sum_{i\in S}(p_{i,S}(\theta')-p_{i,S}(\theta))^2 = 2D^2$ from \eqref{eq:dsq}, 
\item[(iii)] 
$a\log\frac{a}{b} \le \frac{(a-b)^2}{2b}+a-b$. 
\end{itemize}
Similarly to \eqref{eq:meanL}, from (i) and (ii) and $a\left(\log\frac{a}{b}\right)^2 \le \frac{(a-b)^2}{a \wedge b}\left(1+\frac{|a-b|}{3(a\wedge b)}\right) $, we have
\begin{equation}
\sum_{i\in S} p_{i,S}(\theta')\left(\log\left(\frac{p_{i,S}(\theta')}{p_{i,S}(\theta)}\right)\right)^2 \le 8k D^2.
\label{eq:varL}
\end{equation}
Since 
$$
\Pr_{\theta'}[\{S_t\cap \{1,n \}\neq \emptyset\}] = 1- \frac{{n-2 \choose k}}{{n \choose k}} \le 2\frac{k}{n}
$$ 
and according to the model, the input observations are independent, from \eqref{eq:meanL} and \eqref{eq:varL}, we have
\begin{eqnarray}
\E_{\theta'}[L(\vec{S},\vec{y})] 
&=&  m \E_{\theta'}\left[\log\left(\frac{p_{y_1,S_1}(\theta')}{p_{y_1,S_1}(\theta)}\right) \right] \cr
&= & m\sum_{S: S\cap\{1,n\}\neq \emptyset} \Pr_{\theta'}[S_1 = S] \sum_{y\in S}p_{y,S}(\theta')\left[\log\left(\frac{p_{y,S}(\theta')}{p_{y,S}(\theta)}\right) \right] \cr
&\le& 4\frac{m}{n}k^2 D^2
\end{eqnarray}
and
\begin{eqnarray}
\sigma^2_{\theta'}[L(\vec{S},\vec{y})]
&= & m \sigma^2_{\theta'}\left[\log\left(\frac{p_{y_1,S_1}(\theta')}{p_{y_1,S_1}(\theta)}\right) \right] \cr
&\le & m \E_{\theta'}\left[\left(\log\left(\frac{p_{y_1,S_1}(\theta')}{p_{y_1,S_1}(\theta)}\right)\right)^2 \right] \cr
&= & m\sum_{S: S\cap\{1,n\}\neq \emptyset} \Pr_{\theta'}[S_1 = S]\sum_{y\in S}p_{y,S}(\theta')\left[\left(\log\left(\frac{p_{y,S}(\theta')}{p_{y,S}(\theta)}\right)\right)^2 \right] \cr
&\le & 16\frac{m}{n}k^2 D^2.
\end{eqnarray}

\section{Characterizations of $\partial p_k(\vec{0})/\partial x_1$}

In this section, we note several different representations of the parameter $\partial p_k(\vec{0})/\partial x_1$. 

First, note that
\begin{equation}
\frac{\partial p_k(\vec{0})}{\partial x_1} = \frac{1}{k-1}\int_\reals f(x)dF(x)^{k-1}.
\label{equ:ch1}
\end{equation} 
The integral corresponds to $\E[f(X)]$ where $X$ is a random variable whose distribution is equal to that of a maximum of $k-1$ independent and identically distributed random variables with cumulative distribution $F$. 

Second, suppose that $F$ is a cumulative distribution function with its support contained in $[-a,a]$, and that has a differentiable density function $f$. Then, we have
\begin{equation}
\frac{\partial p_k(\vec{0})}{\partial x_1} = A_{F,k} + B_{F,k}
\label{equ:ch2}
\end{equation}

where
$$
A_{F,k} = \frac{1}{k-1}f(a)
$$
and
$$
B_{F,k} = \frac{1}{k(k-1)}\int_{-a}^{a} (-f'(x))dF(x)^k.
$$

The identity (\ref{equ:ch2}) is shown to hold as follows. Note that
$$
\frac{d^2}{dx^2} F(x)^k = \frac{d}{dx}(k F(x)^{k-1}f(x))
= k(k-1)F(x)^{k-2}f(x)^2 + kF(x)^{k-1}f'(x).
$$
By integrating over $[-a,a]$, we obtain
$$
\frac{d}{dx}F(x)^k |_{-a}^{a} = k(k-1) \int_{-a}^{a} f(x)^2 F(x)^{k-2}dx + k \int_{-a}^{a} f'(x) F(x)^{k-1}dx. 
$$
Combining with the fact 
$$
\frac{d}{dx}F(x)^k |_{-a}^{a} = k f(x)F^{k-1}(x) |_{-a}^{a} = k f(a), 
$$
we obtain (\ref{equ:ch2}). 

Note that $B_{F,k} = \E[-f'(X)]/(k(k-1))$ where $X$ is a random variable with distribution that corresponds to that of a maximum of $k$ independent samples from the cumulative distribution function $F$. Note also that if, in addition, $f$ is an even function, then (i) $B_{F,k} \geq 0$ and (ii) $B_{F,k}$ is increasing in $k$.

Third, for any cumulative distribution function $F$ with an even density function $f$, we have $F(-x) = 1-F(x)$ for all $x\in \reals$. In this case, we have the identity
\begin{equation}
\frac{\partial p_k(\vec{0})}{\partial x_1} = \int_0^\infty f(x)^2 (F(x)^{k-2} + (1-F(x))^{k-2}) dx.
\label{equ:gammafeven}
\end{equation}

\section{Proof of Lemma~\ref{prop:gamma}}
\label{sec:gamma}

The upper bound follows by noting that that $B_{F,k}$ in (\ref{equ:ch2}) is such that $B_{F,k} = \Omega(1/k^2)$. Hence, it follows that 
$$
\gamma_{F,k} = O(1).
$$

The lower bound follows by noting that for every cumulative distribution function $F$ such that there exists a constant $C > 0$ such that $f(x) \leq C$ for all $x\in \reals$,
$$
\frac{\partial p_k(\vec{0})}{\partial x_1} = \int_\reals f(x)^2 F(x)^{k-2}dx
\leq C \int_\reals f(x) F(x)^{k-2}dx
= C\frac{1}{k-1}.
$$
Hence, $\gamma_{F,k} \geq (1/C)(k-1)/k^3 = \Omega(1/k^2)$.

\section{Derivations of parameter $\gamma_{F,k}$}

We derive explicit expressions for parameter $\gamma_{F,k}$ for our example generalized Thurstone choice models introduced in Section~\ref{sec:defs}

Recall from (\ref{equ:gamma}) that we have that
$$
\gamma_{F,k} = \frac{1}{(k-1)k^3}\frac{1}{(\partial p_k(\vec{0})/\partial x_1)^2}
$$
where
$$
\frac{\partial p_k(\vec{0})}{\partial x_1} = \int_\reals f(x)^2 F(x)^{k-2}dx
$$

\paragraph{Gaussian distribution} A cumulative distribution function $F$ is said to have a type-3 domain of maximum attraction if the maximum of $r$ independent and identically distributed random variables with cumulative distribution function $F$ has as a limit a double-exponential cumulative distribution function:
$$
e^{-e^{-\frac{x-a_r}{b_r}}}
$$
where
$$
a_r = F^{-1}\left(1-\frac{1}{r}\right)
$$
and
$$
b_r = F^{-1}\left(1-\frac{1}{er}\right) - F^{-1}\left(1-\frac{1}{r}\right).
$$

It is a well known fact that any Gaussian cumulative distribution function has a type-$3$ domain of maximum attraction. Let $\Phi$ denote the cumulative distribution function of a standard normal random variable, and let $\phi$ denotes its density. 

Note that 
\begin{eqnarray*}
\int_\reals \phi(x) d\Phi(x)^r
&\sim & \frac{1}{\sqrt{2\pi}}\int_\reals e^{-\frac{x^2}{2}} d(e^{-e^{-\frac{x-a_r}{b_r}}})\\
&=& \frac{1}{\sqrt{2\pi}}\int_0^{\infty} e^{-\frac{1}{2}(a_r + b_r \log(1/z))^2} e^{-z}dz\\
&=& \frac{1}{\sqrt{2\pi}} e^{-\frac{1}{2}a_r^2} \int_0^\infty z^{a_rb_r} e^{-\frac{1}{2}b_r^2 \log(1/z)^2} e^{-z}dz\\
&\leq & \frac{1}{\sqrt{2\pi}} e^{-\frac{1}{2}a_r^2} \int_0^\infty z^{a_rb_r}  e^{-z}dz\\
&=& \frac{1}{\sqrt{2\pi}} e^{-\frac{1}{2}a_r^2} \Gamma(a_r b_r + 1).
\end{eqnarray*}

Now, note that
$$
a_r \sim \sqrt{2\log(r)} \hbox{ and } b_r = \Theta(1), \hbox{ for large } r.
$$

It is readily checked that $e^{-a_r^2/2} \sim 1/r$ and $\Gamma(a_r b_r + 1) = O(r^\epsilon)$ for every constant $\epsilon > 0$. Hence, we have that
$$
\int_\reals \phi(x) d\Phi(x)^r = O(1/r^{1-\epsilon})
$$
and thus, $\partial p_k(\vec{0})/\partial x_1 = O(1/k^{2-\epsilon})$. Hence,
$$
\gamma_{F,k} = \Omega(1/k^{2\epsilon}).
$$

\paragraph{Double-exponential distribution} Note that $f(x) = \frac{1}{\beta} e^{-\frac{x + \beta \gamma}{\beta}} F(x)$. Hence, we have

\begin{eqnarray*}
\frac{\partial p_k(\vec{0})}{\partial x_1} &=& \int_\reals f(x)^2 F(x)^{k-2}dx\\
&=& \frac{1}{\beta^2}\int_\reals e^{-2\frac{x+\beta\gamma}{\beta}} F(x)^k dx\\
&=& \frac{1}{\beta} \int_0^\infty z e^{-k z}dz\\
&=& \frac{1}{\beta k^2}.
\end{eqnarray*}

\paragraph{Laplace distribution} Let $\beta = \sigma / \sqrt{2}$. Note that 
$$
F(x) = 1-\frac{1}{2}e^{-x/\beta} \hbox{ and } f(x) = \frac{1}{2\beta} e^{-x/\beta}, \hbox{ for } x\in \reals_+.
$$
\begin{eqnarray*}
A &=& \int_0^\infty f(x)^2 F(x)^{k-2}dx\\
&=& \int_0^\infty \left(\frac{1}{2\beta}\right)^2 e^{-2x/\beta} \left(1-\frac{1}{2}e^{-x/\beta}\right)^{k-2} dx\\
&=& \frac{1}{2\beta} \int_{1/2}^1 2(1-z) z^{k-2} dz\\
&=& \frac{1}{\beta} \left(\frac{1}{k-1}\left(1-\frac{1}{2^{k-1}}\right) - \frac{1}{k}\left(1-\frac{1}{2^{k}}\right)\right)\\
&=& \frac{1}{\beta k(k-1)}\left(1-\frac{k}{2^{k-1}} + \frac{k-1}{2^k}\right)
\end{eqnarray*}

and

\begin{eqnarray*}
B &=& \int_0^\infty f(x)^2 (1-F(x))^{k-2}dx\\
&=& \int_0^\infty \left(\frac{1}{2\beta}\right)^2 e^{-2x/\beta} \frac{1}{2^{k-2}}e^{-(k-2)x/\beta} dx\\
&=& \frac{1}{\beta^2 2^k} \int_0^\infty e^{-kx/\beta}dx\\
&=& \frac{1}{\beta k 2^k}.
\end{eqnarray*}

Combining with (\ref{equ:gammafeven}), we obtain
$$
\frac{\partial p_k(\vec{0})}{\partial x_1} = A+B = \frac{1}{\beta k(k-1)}\left(1-\frac{1}{2^{k-1}}\right).
$$

\paragraph{Uniform distribution} Note that
\begin{eqnarray*}
\frac{\partial p_k(\vec{0})}{\partial x_1} &=& \int_\reals f(x)^2 F(x)^{k-2}dx\\
&=& \frac{1}{(2a)^2}\int_{-a}^a \left(\frac{x+a}{2a}\right)^{k-2}dx\\
&=& \frac{1}{2a} \int_0^1 z^{k-2}dz\\
&=& \frac{1}{2a (k-1)}.
\end{eqnarray*}

\end{document}